\documentclass[12pt,leqno]{article}
\usepackage[latin1]{inputenc}
\usepackage{amsxtra}
\usepackage{amsmath}
\usepackage{amssymb}
\usepackage{amsfonts}
\usepackage[all]{xy}
\usepackage{mathrsfs}
\usepackage{amsthm}
\usepackage{enumitem}
\usepackage{mathabx}
\usepackage{bbm}

\usepackage[final]{hyperref}

\makeatletter
\hypersetup{
    unicode=true,           
    pdftoolbar=true,        
    pdfmenubar=true,        
    pdffitwindow=false,     
    pdfstartview={FitH},    
    pdftitle={\@title},     
    pdfauthor={\@author},   
    pdfsubject={},          
    pdfcreator={},          
    pdfproducer={},         
    pdfkeywords={},         
    pdfnewwindow=true,      
    colorlinks,             
    linkcolor=black,        
    citecolor=black,        
    filecolor=black,        
    urlcolor=black          
}
\makeatother

\newtheorem{thm}[equation]{Theorem}
\newtheorem{cor}[equation]{Corollary}
\newtheorem{lem}[equation]{Lemma}
\newtheorem{prop}[equation]{Proposition}

\newtheoremstyle{example}{\topsep}{\topsep}%
     {}
     {}
     {\bfseries}
     {.}
     {2pt}
     {\thmname{#1}\thmnumber{ #2}\thmnote{ #3}}

   \theoremstyle{example}
   
   \newtheorem{Defi}[equation]{Definition}
   \newtheorem{rem}[equation]{Remark}
   \newtheorem{rems}[equation]{Remarks}
   
   \newtheorem{exas}[equation]{Examples}
   \newtheorem{ex}[equation]{Example}

  \numberwithin{equation}{subsection}

\def\eps{{\varepsilon}}

\def\AAA{\mathbb{A}}

\def\NN{\mathbb{N}}
\def\CC{\mathbb{C}}

\def\GG{\mathbb{G}}
 \def\HH{\mathbb{H}}
 \def\LL{{\mathbb{L}}}
\def\PP{\mathbb{P}}
\def\RR{\mathbb{R}}
\def\ZZ{\mathbb{Z}}
\def\QQ{\mathbb{Q}}
\def\TT{{\mathbb{T}}}

\def\gen{\mathfrak{g}}

\def\len{\mathfrak{l}}

\def\Aen{\mathfrak{A}}

\def\Ac{\mathcal{A}}
\def\Bc{\mathcal{B}}
\def\Cc{\mathcal{C}}

\def\Dc{\mathcal{D}}

\def\Fc{\mathcal{F}}

\def\Ic{{\mathcal{I}}}
\def\Lc{\mathcal{L}}
\def\Mc{\mathcal{M}}

\def\Hc{\mathcal{H}}
\def\Oc{\mathcal{O}}
\def\Pc{\mathcal{P}}

\def\Sc{\mathcal{S}}
\def\Tc{\mathcal{T}}

\def\Cb{\mathbf{C}}
\def\Db{\mathbf{D}}

\def\Db{\mathbf{D}}

\def\Ad{{\on{Ad}}}
 \def\Aut{\operatorname{Aut}\nolimits}
 
 \newcommand{\Bun}{\operatorname{\mathbf{Bun}}}
 \newcommand{\Bunrig}{\operatorname{ \mathbf{Bun}}^{\mathrm{rig}}}

 \def\CE{{\on{CE}}}
 \def\cdga{{\Cc dga _{\k}^{\leq 0}}}
 \def\Cdga{{\on{Cdga}_\k}}
 
 \def\Coim{{\on{Coim}}}
 \def\Coker{\operatorname{Coker}\nolimits}
 \def\com{{\on{com}}}
 \def\Cone{{\on{Cone}}}

\def\dAff{{\on{dAff_\k}}}
 \def\dbar{{\ol{\partial}}}
 \def\dec{\on{dec}}
  \def\del{\partial}
  \def\Det{{\on{Det}}}
  \def\dgAlg{{\on{dgAlg}_\k}}
  \def\dgart{{\on{dgArt}_\k}}
 \def\dgLie{{\on{dgLie}_\k}}
 \def\dgVect{{\on{dgVect}_\k}}
 \def\dgCat{{\on{dgCat}_\k}}
 \newcommand{\dgMod}{\on{dgMod}}

 \def\End{\operatorname{End}\nolimits}
 \def\EM{{\on{EM}}}
  \def\Ext{\operatorname{Ext}\nolimits}
  
  \def\Fun{{\on{Fun}}}
  \def\FMP{{\mathbf{FMP}}}

 \def\gl{{{\gen\len}}}

 \def\heart{{{\heartsuit}}}
 \def\Hoch{{\on{Hoch}}}
 \def\hoind{``\hocolim"}
 \def\hopro{``\holim"} 
  \def\hocolim{{  \underrightarrow {\on{holim}} }}
 \def\holim{{  \underleftarrow {\on{holim}} }}
  \def\Hom{\operatorname{Hom}\nolimits}
 \newcommand{\HC}{\mathbf{HC}}
\newcommand{\HHH}{\mathbf{HH}}
 
 \def\Id{\operatorname{Id}\nolimits}
 \def\ILC{{\on{ILC}}}
 \def\ilim{{\text{``}\varinjlim\text{''} \hskip -.2cm  }\,\,\,}
 \def\Im{\operatorname{Im}\nolimits}
 \def\Ind{\on{Ind}}

 \def\k{\mathbf k}
 \def\K{{\mathbf{K}}}
 \def\Ker{\operatorname{Ker}\nolimits}
 
 \def\LC{{\on{LC}}}
 \def\Lie{{\operatorname{Lie}\nolimits}}
 
 \def\Map{\operatorname{Map}}
 \def\Mat{\operatorname{Mat}\nolimits}
 
\def\mult{\on{mult}}

\def\Ob{\operatorname{Ob}\nolimits}
\def\on{\operatorname}
\def\ol{\overline}
\def\op{{\on{op}}}

\def\pA {\mathring {\AAA}}
\def\Pic{{\on{Pic}}}

\def\Perf{\on{Perf}}
\def\bPerf{{\mathbf{Perf}}}
\def\plim{{\text{``}\varprojlim\text{''} \hskip -.2cm  }}

 \def\Pro{\on{Pro}}
 \def\PVect{{\on{PVect}}}

\newcommand{\Rep}{\operatorname{Rep}}
\def\Res{\operatorname{Res}\nolimits}
\newcommand{\RAut}{\operatorname{\mathbb{R}\mathbf{Aut}}}
\newcommand{\RBun}{\operatorname{\mathbb{R}\mathbf{Bun}}}

\newcommand{\spRBun}{\operatorname{\mathbb{R}\mathrm{Bun}}}
\newcommand{\RBunrig}{\operatorname{\mathbb{R}\mathbf{Bun}}^{\mathrm{rig}}}
\newcommand{\RMap}{\operatorname{\mathbb{R}\mathbf{Map}}}
\def\RHom{{\on{RHom}}}

 \def\Spec {\on{Spec}}
 \def\sSet{{s\Sc et}}
 \def\St{{\Sc t}}

\def\Ta{{\on{Ta}}}
\def\Tate{{\on{Tate}}}
\def\bTate {{\mathbf{Tate}}}

\def\Tot{{\on{Tot}}}
\def\tr{\operatorname{tr}\nolimits}
\def\TV{{\on{TopVect}}}

\def\ul{\underline}

\def\Vect{\on{Vect}}

\def\x{{\mathbf {x}}}

\def\wc{\widecheck}
\def\wh{ \widehat}
\def\wt{\widetilde}

\def\lra{\longrightarrow}
\def\lla{\longleftarrow}
\def\(({(\hskip -1mm (}
\def\)){)\hskip -1mm )}
 \def\be{\begin{equation}}
\def\ee{\end{equation}}
\def\ed{\end{document}}

\title{Higher Kac-Moody algebras and moduli spaces of $G$-bundles}
\author{Giovanni Faonte, Benjamin Hennion, Mikhail Kapranov}
 
\begin{document}

\maketitle

\begin{abstract}
We provide a generalization to the higher dimensional case of the construction of the current algebra $\gen((z))$, its Kac-Moody extension $\wt \gen$ and of the classical results relating them to the theory of $G$-bundles over a curve. For a reductive algebraic group $G$ with Lie algebra $\gen$, we define a dg-Lie algebra $\gen_n$ of $n$-dimensional currents in $\gen$. For any symmetric $G$-invariant polynomial $P$ on $\gen$ of degree $n+1$, we get a higher Kac-Moody algebra $\wt \gen_{n,P}$ as a central extension of $\gen_n$ by the base field $\k$. Further, for a smooth, projective variety $X$ of dimension $n\ge 2$, we show that $\gen_n$ acts infinitesimally on the derived moduli space $\RBunrig_G(X,x)$ of $G$-bundles over $X$ trivialized at the neighborhood of a point $x \in X$. Finally, for a representation $\phi: G \to GL_r$ we construct an associated determinantal line bundle on $\RBunrig_G(X,x)$ and prove that the action of $\gen_n$ extends to an action of $\wt \gen_{n,P_\phi}$ on such bundle for $P_\phi$ the $(n+1)^{\mathrm{th}}$ Chern character of $\phi$.
\end{abstract}
\pagebreak
\tableofcontents
\pagebreak
\addtocounter{section}{-1}
 

\section{Introduction}
\paragraph{(0.1)} Let $G$ be a reductive algebraic
group over $\CC$ with Lie algebra $\gen$. The formal current algebra $\gen((z))=\gen\otimes_\CC \CC((z))$
and its central extension $\wt \gen$ (the Kac-Moody algebra) play a fundamental role in 
many fields.  It can be considered as the algebraic completion
of the loop algebra $\Map(S^1,\gen)$, see \cite{pressley-segal}. 

\vskip .2cm

In particular, $\gen((z))$ is fundamental in the study of
$\Bun_G(X)$, the moduli stack of principal $G$-bundles on a smooth projective curve $X$ over $\CC$. 
More precisely, let $x\in X$ be a point. We then have the scheme  (of infinite type)
$\Bunrig_G(X,x)$ parametrizing bundles $P$ together with a trivialization on $\wh x$, the formal neighborhood of $x$. 
The ring of functions on $\wh x$ is $\wh \Oc_{X,x}\simeq \CC[[z]]$, the completed local ring and its field
of fractions $K_x \simeq \CC((z))$ corresponds to the punctured formal neighborhood $\wh x^\circ$. 

\vskip .2cm

 The key result \cite{tsuchiya} is that the 
  Lie algebra $\gen_x = \gen\otimes K_x$ acts on the scheme $\Bunrig_G(X,x)$ by vector fields.
Moreover, any representation $\phi$ of $G$ gives rise to the determinantal line bundle $\det^\phi$ on 
  $\Bunrig_G(X,x)$; the action of $\gen_x$ extends to the action, on $\det^\phi$,  of the central extension $\wt\gen_x$ with
  central charge given by a local version of the Riemann-Roch theorem for curves.

  \paragraph{(0.2)}
  Our goal in this paper is to generalize these results from curves to $n$-dimensional varieties $X$ over $\CC$, $n\geq 1$
  (one can replace $\CC$ by any field of characteristic $0$). The first question in this direction is what should
  play the role of $\gen((z))$. In the analytic (as opposed to the formal series) theory,
  natural generalizations of $\Map(S^1,\gen)$ are provided by the current Lie algebras $\Map(\Sigma, \gen) = \gen\otimes_\CC
  C^\infty(\Sigma)$ where $\Sigma$ is a compact $C^\infty$-manifold of dimension $>1$. 
  Our approach
  can be seen as extending this idea to the derived category. 
  
  \vskip .2cm
  
  More precisely, 
  the role of
  $\gen((z))$ will be played by the dg-Lie algebra $\gen_n^\bullet = \gen\otimes_\CC \Aen_n^\bullet$,
  where $\Aen_n^\bullet = R\Gamma(D_n^\circ, \Oc)$ is the commutative dg-algebra of derived
  global sections of the sheaf $\Oc$ on the $n$-dimensional {\em punctured formal disk}
  $D_n^\circ = \Spec \bigl(  \CC[[z_1, \dots, z_n]]\bigr) -\{0\}$.  More invariantly, we have the punctured formal disk
  $\wh x^\circ\simeq D_n^\circ$ associated to a point $x\in X$
  and the corresponding current algebra $\gen^\bullet_x\simeq \gen^\bullet_n$.  For $n>1$, passing from the non-punctured
  formal disk $\wh x$, to $\wh x^\circ$ does not increase the ring of functions (Hartogs' theorem) but  one gets new
  elements in the higher cohomology of the sheaf $\Oc$, so $\Aen_n^\bullet$ can be regarded as a
  ``higher" generalization of the Laurent series field $\CC((z))$, to which it reduces for $n=1$. 
  
  \vskip .2cm

  Principal bundles on $X$ form an Artin stack
  $\Bun_G(X)$ and we can still form a
  scheme $\Bunrig_G(X,x)$ as above. However these objects are, 
    for $n>1$, highly singular because deformation theory 
  can be obstructed. The correct object to consider is the {\em derived moduli stack} $\RBun_G(X)$
  obtained, informally, by taking the non-abelian derived functor of $\Bun_G$, i.e., extending the moduli functor to
  test rings which are commutative dg-algebras \cite {toen-vezzosi}. When $X$ is a curve, $\RBun_G(X)\simeq\Bun_G(X)$,
  but for $n>1$ there is a difference. Most importantly, the tangent complex of $\RBun_G(X)$ is perfect
  (a smoothness property). We can similarly construct the 
    {\em derived scheme} $\RBunrig_G(X,x)$,
 (an object which locally looks like the spectrum of a commutative dg-algebra)  which should also be intuitively
 considered as  being smooth. 
  
 \vskip .2cm
 
 We show, first of all,  (Theorem  \ref{prop-groupaction})
  that
 $\gen_n^\bullet$ acts on $\RBunrig(X,x)$ by vector fields, in the derived sense. At the level of  cohomology, the action gives, in particular,  a map
 \[
 H^{n-1}_\dbar(\gen_n^\bullet) \lra \HH^{n-1}\bigl( \RBunrig_G(X,x), \TT\bigr)
 \]
 (here $\TT$ is the tangent complex and $\dbar$ is the differential of $\gen_n^\bullet$). When $n=1$, it is the action by vector fields in the usual sense.
 In the first new case $n=2$, after restricting to the non-obstructed smooth part of the moduli space, on which $\TT$
 is the usual tangent sheaf, the target of this map becomes the {\em space of deformations} of the (part of the)
 moduli space. Deforming the moduli space can be understood as changing the cocycle condition defining
 $G$-bundles
 (Remark \ref{rem:def-moduli}). 
 
 \vskip .2cm
 
 Further, each invariant polynomial  $P$ on $\gen$ of degree $(n+1)$
 gives rise to a central extension $\wt \gen^\bullet_{n, P}$ (the {\em higher Kac-Moody algebra}).
 Note that unlike the case $n=1$,  we now have many non-proportional classes, even for $\gen$ simple.
 Intuitively, they correspond to degree $n+1$ characteristic classes for principal $G$-bundles. 
 As before, let $\phi$ be a representation of $G$. 
 We prove (Theorem \ref{thm:action-det}) that the determinantal line bundle $\det^\phi$ on $\RBunrig(X,x)$
 is acted upon by $\wt \gen^\bullet_{n, P_\phi}$ where $P_\phi(x) = \tr(\phi(x)^{n+1})/(n+1)!$
 is the ``$(n+1)$-th component of the  Chern character" of $\phi$.
 
 These results suggest that representations of the dg-Lie algebra $\gen_n^\bullet$ should produce
 geometric data on the derived moduli spaces of $G$-bundles on $n$-dimensional manifolds.

   \paragraph{(0.3)} The stack $\Bun_G(X)$ can be seen as a version of the non-abelian first cohomology $H^1(X, G(\Oc_X))$.
   When $X$ is a curve, 
     the above classical 
    results can be seen as forming a part of the ``adelic approach'' to the geometry
   of curves. This approach consists in  using the idealized ``\v Cech covering'' of $X$ formed by $\wh x$ and $X^\circ=X-\{x\}$,
   with ``intersection'' $\wh x^\circ$,  to calculate the $H^1$. If $X$ is a curve and $G$ is semi-simple,
    then $G$-bundles on $X^\circ$
   (and certainly on $\wh x$)
   are trivial, and we can write $\Bun_G(X) = G(\wh x)\backslash G(\wh x^\circ)/G(X^\circ)$ (stack-theoretic
   quotient on the left). We then  similarly represent $\Bunrig_G(X,x)$ as the coset space $ G(\wh x^\circ)/G(X^\circ)$,
   with  $G(\wh x^\circ) = G(K_x)$ being a group ind-scheme with Lie algebra $\gen_x$. 
   
   \vskip .2cm
   
   A generalization of the adelic formalism to varieties $X$ of dimension $n>1$ was proposed by 
   Parshin and Beilinson \cite{beilinson-adeles} \cite{huber} \cite{osipov}. 
   In this approach the completed local fields (analogs of $K_x$ for curves) are parametrized not by points, but by 
    flags $\{x\}=X_0\subset X_1\subset \cdots\subset X_{n-1}\subset X$ of irreducible subvarieties in   $X$. If all the $X_i$'s are smooth, then the completion is isomorphic to $\k((z_1))\cdots ((z_n))$, the iterated Laurent series field. 
   As before, it can be seen as a version of the \v Cech formalism for an idealized open covering formed 
  by certain
   formal neighborhoods. 
   However, the manipulations with iterated Laurent series fields are quite complicated:  in order to
   capture all the ``adic topologies'', 
    they should be considered
   as $n$-fold iterated ind-pro-objects ($n$-Tate spaces)  \cite{bgw-tate} and every step involves many levels  of technical work.

   \vskip .2cm
   
   In a sense, our approach can be seen as  a  ``simplified version" of the flag ad\`ele formalism, in which we keep
   track only of points $x\in X$ (just like for curves) and package all the combinatorial 
   data involving subvarieties of dimensions
   $\neq 0,n$, into  a ``black box" using the cohomological formalism. This allows us
    to avoid working with iterated ind-pro-objects and deal instead with classical  Tate spaces  (just like for curves)
    at the small  price of having to pass to the derived category of such spaces, i.e., to study {\em Tate  dg-spaces}
    (or {\em Tate complexes}),
    see \S \ref{subsec:tate-back} for details.   For example,  $\Aen_n^\bullet$ is a Tate complex for each $n$. 
    Our treatment is an adaptation and development of the approaches of \cite{drinfeld} 
    \cite{hennion-tate}.
        
    \paragraph{(0.4)} To relate our approach to the idea of $\Map(\Sigma,\gen)$, we can use a particular model
    $A_n^\bullet$ of the 
      ``abstract" commutative dg-algebra
     $\Aen_n^\bullet = R\Gamma(D_n^\circ,\Oc)$. This model 
     is formed by relative differential forms on the Jouanolou torsor,
    see \S  \ref{subsec:exp-mod}B.  Such torsors have been used in \cite{beilinson-drinfeld} as a general tool. 
    In our case, $A_n^\bullet$ provides a very precise algebraic analog of $\Omega^{0,\bullet}(\CC^n-\{0\})$, the
  $\dbar$-algebra of Dolbeault forms on $\CC^n-\{0\}$. In particular, such features of classical complex analysis
  as the Martinelli-Bochner form  or its ``multipole" derivatives, have direct incarnations in $A_n^\bullet$, see
    Proposition \ref{prop:MB-H}.  Our algebraic approach allows us to include these features in the formal setting
    of Tate (dg-)spaces. It also lends itself to a representation-theoretic analysis providing the analog of representing
    elements of $\k((z))$ as infinite sums of monomials (Theorem \ref{thm:irr}). 
    
    \vskip .2cm
    
    Restricting from $\CC^n-\{0\}$ to the unit sphere $S^{2n-1}$, we can see $A_n^\bullet$ as an algebraic analog of
    $\Omega^{0,\bullet}_b(S^{2n-1})$, 
    the tangential Cauchy-Riemann complex  (the $\dbar_b$-complex \cite{beals-greiner}\cite{dragomir}) of the sphere. The degree $0$ part
    of  $\Omega^{0,\bullet}_b(S^{2n-1})$ is $C^\infty(S^{2n-1})$, the algebra of smooth complex functions on $S^{2n-1}$. 
    This means that the degree $0$ part of $\gen_n^\bullet$ can be seen as an algebraic version of the
   current  Lie algebra $\Map(S^{2n-1},\gen)$, and the entire $\gen_n^\bullet$ as a natural derived thickening
   of this current algebra. 
   
   \vskip .2cm
   
   Usually,  considering  
  $\Map(\Sigma,\gen)$ with $\dim(\Sigma)>1$,  produces Lie algebras which, instead of interesting
  central extensions  (classes in $H^2$) have interesting  higher cohomology classes. These classes are
    typically   given by versions of the formula
   \begin{equation}\label{eq:integral}
   \gamma(f_0,\dots, f_n) \,\,=\,\,\int \tr(f_0 \, df_1 \cdots df_n). \tag{0.5}
   \end{equation}
   In our case (Theorem \ref{thm:gamma-P}),
    we still use a version of this formula  (with   integration over $S^{2n-1}$, done algebraically)
    but the classes we get have total degree $2$ and so  give central extensions,  regardless of $n$. 
    This happens because we take into account the grading on the dg-algebra, 
  In this sense our derived approach embeds $\Map(S^{2n-1},\gen)$ into an object whose properties
  are closer to those of $\Map(S^1,\gen)$. 
  
  All this suggests that our higher Kac-Moody algebras should have an interesting representation theory. 
  
  \paragraph {(0.6)} As in the 1-dimensional case, a key intermediate step in identifying the central extension
  acting on the determinant bundle, is given by
   a local analog of the Riemann-Roch
    theorem (Corollary \ref{cor:localRR}).
     It has the form of comparison of two central extensions of $\gen_n^\bullet$ for $\gen=\gl(r)$: one given
    by a version of \eqref{eq:integral}, the other induced from the ``Tate class" of the 
     the endomorphism dg-algebra
    of the Tate complex $(\Aen_n^\bullet)^{\oplus r}$.  For $r=1$ this statement can be seen as an analog, in our
    simplified adelic formalism, of the main result of Beilinson \cite {beilinson-adeles}. 
    Since we deal with the current algebras only, we detect
    only the Chern character; the Todd genus will naturally appear, as in 
    \cite{feigin-tsygan-RR},  after we include the dg-algebra  $R\Gamma(D_n^\circ, \TT)$,
     see (0.8) below.

  \paragraph {(0.7)} We use three main technical tools. The first one is the general theory of 
   derived stacks \cite{toen-vezzosi}. It is necessary for us to work freely with quite general derived stacks and even
   prestacks in order to study, for instance, 
   the group object corresponding to $\gen_n^\bullet$. This is an (infinite dimensional)
  derived group $G(D_n^\circ)$, see Proposition \ref{prop:Lie-alg-G-hat}.    In particular, for dealing with
  various infinitesimal constructions (even such seemingly simple ones as ``passing from a group to its Lie algebra")
  we  need to use Lurie's formalism of formal moduli problems \cite{lurie-dagx}. 
  
  \vskip .2cm
  
   The second technical tool is cyclic homology of dg-categories,  a concept of great flexibility and invariance.
   It includes, in particular cyclic (and de Rham) homology for schemes and at the same time, is related to the Lie algebra
   homology of endomorphism dg-algebras of objects in  a dg-category. 
   
   \vskip .2cm
   
   Another  important tool is   the $GL_n$-invariance of our cohomology classes. In exploiting this invariance, 
   the  Jouanolou model for $\Aen_n^\bullet$ is more convenient in that it allows an explicit $GL_n$-action which can be
   analyzed in detail at the level of complexes.  This level of detail is not available for more abstract models,
   e.g., for the flag-adelic one.

    \paragraph{(0.8)} This paper is  related to the  idea, mentioned already in \cite{beilinson-drinfeld}
    and developed in \cite{francis-gaitsgory}, 
    of generalizing the theory of chiral and holomorphic
     factorization algebras to higher dimensions.     In this approach
    we are dealing with finite collections $(x_i)$ of points moving in  an $n$-dimensional variety $X$, with
    ``singularities" developing when $x_i=x_j$ for some $i\neq j$. These singularities are of cohomological nature,
    representing classes in  $H^{n-1}(X^N-\{\on{diag}\}, \Oc)$. 
    
    \vskip .2cm 
    
    Note that the standard quantum field theory approach deals with collections of points in the Minkowski space
    with singularities developing when some $x_i-x_j$ lies on the light cone, see, e.g.,   \cite{tamarkin} for
    a discussion from the factorization algebra point of view. 
     However, our cohomological formulas in the Jouanolou model
     are, up to ``details" involving the cohomological grading,  algebraically similar
    to this. For example, the role of the standard ``propagator" $1/\|z\|^2$ is played by the Martinelli-Bochner form
    $\Omega(z,z^*)$. 
    
    \vskip .2cm
    
    The next natural step in this direction would be  to study the  dg-Lie algebra $R\Gamma(D_n^\circ, \TT)$,
    the analog of the Witt algebra of formal vector fields on the circle (as well as its central extensions).
     It should act on the derived
    moduli stack of $n$-dimensional rigidified complex manifolds $(X,x, (z_1, \dots, z_n))$ where $x\in X$ is a marked point and
    $(z_1,\dots, z_n)$ is a formal coordinate system near $x$. There is a natural combined version involving the derived Atiyah algebra
    ($R\Gamma$ of the semidirect product of matrix functions and vector fields). 
    We plan to address these issues in a future paper. 
    The additional technical difficulty here is the need
    to work with (quite general) derived stacks in the analytic context, as not all deformations of an algebraic variety
    are algebraic. 
    
    \paragraph{(0.9)} Here is a logical summary of the paper. 
    
    \vskip .2cm
    
    In Chapter \ref{sec:der-fun-series} we present our approach to derived analogs of Laurent polynomials,  formal 
    Laurent series and adeles.  The abstract derived adelic formalism is developed in \S \ref{subsec:der-ad-for}. 
    In \S \ref{subsec:exp-mod} we describe our explicit model $A_{[n]}^\bullet$ for the derived analog of
    the ring $\k[z^{\pm 1}]$ of Laurent polynomials (with $A_n^\bullet$, the analog of the field of 
    formal Laurent series, obtained as a completion). The richness of the classical theory of Kac-Moody,
    Virasoro and vertex algebras stems largely from the possibility to make purely algebraic computations
    with (operator valued) Laurent series, their residues etc., and we want our formalism to have similar
    capabilities. While, on one hand, our $A_{[n]}^\bullet$ is an algebraic analog of the Dolbeault complex
    (the precise connection is explained  in \S \ref{subsec:comp-dolb}), on the other hand, its elements
    allow explicit algebraic presentation similar to  writing $f(z) = \sum a_q z^q$ in the one-variable case. 
    This is done in \S \ref{subsec:repan}. The role of the integer $q$ (parametrizing, conceptually, irreducible
    representations of $GL_1$) is now played by certain irreducible representations of $GL_n$. 
    We determine, in Theorem \ref{thm:irr}, which representations do appear and show that they
    appear exactly once. So the analogs of individual Laurent monomials are now intrinsically defined
    ``monomial subspaces" in $A_{[n]}^\bullet$, invariant and irreducible under $GL_n$. 
    Another standard tool of the classical computations is the algebraic residue of a Laurent polynomial/series
    $\Res( f(z) dz) =   a_{-1}$. In \S \ref{subsec:resdual} we develop a similar algebraic residue formalism for
    our model. As mentioned above, the standard differential $dz/z$ is replaced by the Martinelli-Bochner
    form $\Omega$, which has an algebraic meaning in our model. 
    
    \vskip .2cm
    
    In Chapter  \ref  {sec:resclass} we study the higher-dimensional analog of the residue pairing
    $\Res(fdg)$ of the classical theory. This is a certain degree $1$ cyclic homology class $\rho$
    of $A_{[n]}^\bullet$. We give two definitions of $\rho$: an abstract one using Hodge decomposition
    of the cyclic homology of $\AAA^n-\{0\}$ in \S \ref {subsec:res-class} B, and a more explicit algebraic
    one, analogous to (0.5) and thus to the framework of  \S \ref{subsec:resdual}, in  \S \ref {subsec:res-class} A. 
    We show that the two definition agree up to a non-zero scalar, by using $GL_n$-invariance
    properties (Theorem \ref{thm:res-invariant}). The preliminary sections \ref{subsec:cyc-hom-alg-cat} and 
    \ref{subsec:cyc-hom-sch} recall, with some slight developments, 
     necessary techniques related to cyclic homology of dg-categories and
    schemes. In particular, we present a result (Proposition \ref{prop:derivedcat-qaffine}) describing the coherent
    derived category of a quasi-affine scheme in terms of the dg-algebra of derived functions, which
    may have other applications. 
    
    \vskip .2cm
    
    In Chapter \ref {sec:hi-curr} we introduce our main objects of study: the higher current algebras
    and their central extensions. While Chapter  \ref  {sec:resclass} was studying the higher analogs 
    of the pairing $\Res(fdg)$ for {\em scalar} functions $f,g$, the central extensions are governed by
    the analogs of the pairing $\Res(\tr(fdg))$ for {\em matrix} functions. As in the scalar case,
    we start by recalling, in \S \ref {subsec:cyc-Lie-hom}, the abstact formalism of Lie algebra and cyclic
    homology, and in \S \ref{subsec:hi-cur} we construct the cocycles explicitly, using the algebraic
    residue formalism of  \S \ref{subsec:resdual}. We prove (Theorem \ref{thm:gamma-P})
    that the space of invariant polynomials of degree $n+1$ on a reductive Lie algebra
    $\gen$ embeds injectively into the second cohomology of the current algebra. 
    
    \vskip .2cm
    
    Chapter \ref {sec:tate-RR} is devoted to identifying some particular central extensions of
    higher current algebras: those coming from  the ``trace anomaly" (failure of the identity $\tr[A,B]=0$
    for infinite-dimensional operators). The abstract properties of this  anomaly are encapsulated
    in the theory of Tate vector spaces and Tate complexes which we recall in \S \ref{subsec:tate-back}
    and perform some additional study of the homological properties of the category of Tate spaces.  In \S \ref{subsec:tate-ob-dgcat} we embed this into a more general formalism of Tate objects in
    dg-categories, where the trace anomaly is obtained as a delooping statement. The algebra $A_n^\bullet$,
    the higher analog of $\k((z))$, is naturally a Tate complex and we get the corresponding Tate extension
    of $A_n^\bullet\otimes\gl_r$ from the action on $A_n^\bullet\otimes\k^r$. Determining this extension
    can be seen as a local analog of the Riemann-Roch theorem for rank $r$ vector bundles.
    It is done in \S \ref{subsec:LRR}. More precisely, we prove in Theorem \ref{thm:compar-1} the the extension
    is given by a cocycle proportional to the matrix residue invariant $\Res(\tr(f_0df_1\cdots df_n))$
    with $\Res$ is understood in the sense of \S  \ref{subsec:resdual}. 
    
    \vskip .2cm
    
    In Chapter \ref{sec:act-der} we relate higher affine algebras with derived moduli spaces of $G$-bundles
    ($G$ a reductive group) on $n$-dimensional varieties $X$. In \S \ref{subsec:back-der} and \S \ref{subsec:kod-sp}
    we provide a brief backround on derived geometry in general and on the Kodaira-Spencer homomorphism
    in particular. In \S \ref{subsec:der-cur-mod-G} we prove (Theorem \ref{prop-groupaction}) that the higher
    affine algebra associated to $\gen=\Lie(G)$ acts on  the derived scheme $\RBunrig_G(X, x)$ 
    discussed above.  Further, if we want to extend to action
    to natural determinantal line bundles over  $\RBunrig_G(X, x)$, the obstruction to that  turns out to be
    the trace anomaly studied in Chapter \ref {sec:tate-RR}. This is shown in \S \ref{subsec:cent-ext-Tate}.
    Finally, in \S \ref{subsec:act-det-tor} we explicitly identify the cocycle that governs the anomaly.
    More precisely,  Theorem \ref{thm:action-det} together with the local Riemann-Roch theorem
    \ref{thm:compar-1} shows that
     that it corresponds to the cocycle proportional to 
    $\Res(\tr(\phi(f_0) d\phi(f_1)\cdots d\phi(f_n))$ where $\phi: G\to GL_r$ is the representation of $G$ used
    to define the determinantal bundle and $\Res$ is as in \S \ref{subsec:resdual}. 
    
    \vskip .2cm
    
    The brief Appendix collects background material on dg-algebras, dg-categories and model categories,
    used throughout the paper.

    \paragraph{(0.10)}  
    We would like to thank A. Beilinson, L. Hesselholt, P. Schapira and B. Vallette for useful correspondence. 
    G.F. is thankful to J.D. Stasheff, B.H. to P. Safronov and M.K. to A. Polishchuk
     for interesting discussions. We are also grateful to the referees for numerous remarks which helped us
     improve the paper. 
    
   The work of M.K. was supported by the World Premier International Research Center Initiative (WPI Initiative), MEXT, Japan.
   It was also partially supported by the EPSRC Programme Grant  EP/M024830 ``Symmetries and Correspondences".
   
   B.H. would like thank the Max-Planck institut f\"ur Mathematik in Bonn for hosting and supporting him during the redaction of this paper.
     
\section{Derived analogs of functions and series}\label{sec:der-fun-series}

\subsection{Derived adelic formalism}\label{subsec:der-ad-for}

\paragraph{A. Local part.} 
We fix a base field $\k$ of characteristic $0$. 
For $n\geq 1$ we have the $n$-dimensional formal disk $D_n=\Spec \k[[z_1,\dots, z_n]]$ and
the punctured formal disk $D_n^\circ= D_n-\{0\}$. We consider them as the completion of the affine space 
$\AAA^n=\Spec \k[z_1, \dots, z_n]$ and of  the punctured affine space $\pA^n=\AAA^n-\{0\}$. 

\vskip .2cm

Fundamental for us will be the commutative dg-algebras
\[
(\Aen_n^\bullet,\dbar)  = R\Gamma(D_n^\circ, \Oc), \quad (\Aen_{[n]}^\bullet,\dbar) = R\Gamma(\pA^n, \Oc). 
\]
defined uniquely up to quasi-isomorphism.
The cohomology of these dg-algebras 
  is well known and can be obtained using the covering of  $D^\circ_n$ by $n$ affine
open subsets $\{z_i\neq 0\}$.

\begin{prop}\label{prop:H^iO}
 For   $n=1$ the scheme $D^\circ_1$ is affine
 with ring of functions $\k((z))$ and $\pA^1$ is affine with ring of functions
 $\k[z, z^{-1}]$. For $n>1$ we have
  \[
 H^i(D^\circ_n, \Oc) \,\,=\,\,\begin{cases}
 \k[[z_1, \dots, z_n]], &  i=0;
 \\
 z_1^{-1}\cdots z_n^{-1}\k[z_1^{-1}, \dots, z_n^{-1}], & i=n-1;
 \\
 0, & \text{otherwise.}
 \end{cases}
 \]
 Here the notation $z_1^{-1}\cdots z_n^{-1}\k[z_1^{-1}, \dots, z_n^{-1}]$ can be seen as encoding the 
 action of the $n$-dimensional torus $\GG_m^n$
 on $H^{n-1}(D^\circ_n, \Oc)$. 
 
 The cohomology $H^i(\pA^n,\Oc)$ differs from the above by replacing $\k[[z_1, \dots, z_n]]$
 by $\k[z_1, \dots , z_n]$. 
  \end{prop}

  Thus, although for $n>1$, the scheme  $D^\circ_n$. resp. $\pA^n$, is not affine  
  and its global functions are the same as for  $D_n$, resp. $\AAA^n$, the missing ``polar parts"
  are recovered in the higher cohomology of the sheaf $\Oc$. The dg-algebras
  $\Aen_{[n]}$ and $\Aen_n$ are, therefore, correct $n$-dimensional generalizations of the 
  the rings of Laurent polynomials and Laurent series in one variable. 
  
  We will also use the doubly graded dg-algebras
 
 \[
(\Aen_n^{\bullet \bullet},\del, \dbar)  =  \bigoplus_p R\Gamma(D_n^\circ, \Omega^p), \quad
( \Aen_{[n]}^{\bullet\bullet},\del,\dbar)  =  \bigoplus_p R\Gamma(\pA^n,  \Omega^p),
\]
with $\del$ being induced by the de Rham differential on forms (and increasing the first grading)
and $\dbar$ being the differential on $R\Gamma$ (and increasing the second grading). 


\paragraph{B. Global part.} Let $X$ be a smooth $n$-dimensional   variety over $\k$
 and $x\in X$ a $\k$-point.
We then have the completed local ring $\wh\Oc_{X,x}$ which is isomorphic (non-canonically)
to $\k[[z_1, \dots, z_n]]$. We denote $\wh x = \Spec\wh\Oc_{X,x}$ the formal disk near $x$
and by $\wh x^\circ = \wh x-\{x\}$ the punctured formal disk near $x$. We then form
the commutative dg-algebras
\[
\Aen^\bullet_x = R\Gamma(\wh x^\circ, \Oc), \quad \Aen^{\bullet\bullet}_x = 
\bigoplus_p R\Gamma(\wh x^\circ, \Omega^p). 
\]
In particular, the Grothendieck duality defines a canonical linear functional
\be\label{eq:GR-res}
\Res_x: R\Gamma(\wh x^\circ, \Omega^n) \buildrel\delta\over\lra R\Gamma_{\{x\}}
(\wh x, \Omega^n)[1] \lra \k[1-n]. 
\ee
Let now  $\x = \{ x_1, \dots, x_m\} \subset X$ be a finite set of  disjoint $\k$-points.  We denote $X^\circ=X-\x$ the complement 
of $\x$, and write $\wh\x = \bigsqcup \wh x_i$
and $\wh\x^\circ = \bigsqcup \wh x_i^\circ$. 
In particular, we have the commutative dg-algebra
$R\Gamma(X^\circ, \Oc)$. Elements of this dg-algebra can be seen as $n$-dimensional
analogs of rational functions on a curve with poles in $x_1,\dots, x_m$.
Similarly, if $E$ is a vector bundle on $X$, we have the dg-modules
$\Aen_\x^\bullet(E) = R\Gamma(\wh \x^\circ, E)$
over $\Aen^\bullet_\x = R\Gamma(\wh\x^\circ, \Oc)$ and $R\Gamma(X^\circ,E)$
over  $ R\Gamma(X^\circ, \Oc)$. 
We have the canonical morphism
of complexes (commutative dg-algebras for $E=\Oc$) given by the restriction:
\be\label{eq:rest-delta}
\delta:   R\Gamma(X^\circ, E)
\oplus \Gamma(\wh \x, E)
\lra \Aen_\x(E)
\,= \,  \bigoplus_{i=0}^m \Aen^\bullet_{x_i}(E). 
\ee
This morphism can be seen as the dg-version of the adelic complex of Beilinson \cite{beilinson-adeles}
\cite{huber} ,
in which the dependence on schemes of dimensions $1, \dots, n-1$ has been ``integrated away"
and hidden in the cohomological formalism. The proof of the next proposition can be seen as an explicit
comparison. 
 
\begin{prop}\label{prop:adelic-cx}
The homotopy fiber of $\delta$ is identified with $R\Gamma(X, E)$.  
\end{prop}

\noindent{\sl Proof:} 
 For any Noetherian scheme $Y$ of dimension $n$ over $\k$ (proper or not) and any coherent sheaf
$\Fc$ on $Y$, the construction of  \cite{beilinson-adeles}\cite{huber} provides an explicit model
$C^\bullet(\Fc)$ for $R\Gamma(Y,\Fc)$.  By definition,
\[
C^p(\Fc) \,\, = \bigoplus_{0\leq i_0<\cdots < i_p\leq n} C_{i_0,\dots, i_p}(\Fc),
\]
where
\[
C_{i_0,\dots, i_p}(\Fc)\,\,= \operatorname*{\prod{}^{\displaystyle '}}\limits_{{Y_0\subset\cdots\subset Y_p\subset Y}
\atop
{\dim(Y_\nu)=i_\nu}} C_{Y_0,\dots, Y_p}(\Fc)
\]
is the appropriate restricted product, over all flags $Y_0\subset\cdots\subset Y_p\subset Y$ of irreducible
subschemes, of the completions $C_{Y_0,\dots, Y_p}(\Fc)$ as described in
 \cite{beilinson-adeles}\cite{huber}. In particular, for $p=0$ and $i_0=0$,
the summand $C_0(\Fc)$ is the usual product of $C_y(\Fc) = \Gamma(\wh y, \Fc)$ over all $0$-dimensional points $y\in Y$. 
We now  take $Y=X$,  $\Fc=E$ and represent 
\[
C^\bullet(\Fc) \,\,=\,\,\on{hofib}\bigl\{ C_1^\bullet\oplus C_2^\bullet \buildrel d\over \lra C_3^\bullet\bigr\},
\]
identifying the three summands with the corresponding summands in \eqref {eq:rest-delta} and $d$ with $\delta$.
Explicitly, we take
$C_2$ to be the direct sum of $C_y(\Fc)$ over $y\in \x$, we take $C_1^\bullet$ to be the direct sum of restricted
products of $C_{Y_0,\dots, Y_p}(\Fc)$ over all flags $Y_0\subset\cdots\subset Y_p$ such that $Y_0$
is any subvariety (of any dimension) other than some $y\in\x$. Then $C_2^\bullet$ is the adelic complex for the restriction of $E$
$Y=X-\x$. Similarly, we take $C_3$ to be the direct sum of restricted products of $C_{Y_0,\dots, Y_p}(\Fc)$ over all flags $Y_0\subset\cdots\subset Y_p$ such that $Y_0$
equals some $y\in\x$. Then $C_3^\bullet$ is the adelic complex for the restriction of $E$ to $\wh\x$. This proves
the statement. \qed

 \vskip .2cm

 In particular, for $E=\Omega_X^n$, we have morphisms 
 \[
 \Res_{x_i, X}: R\Gamma(X^\circ, \Omega^n)\to
 \k[1-n],\quad i=1,\dots, m, 
 \]
  which satisfy the {\em residue theorem}:

 \begin{prop}
$\sum_{i=1}^m \Res_{x_i, X} = 0$. \qed
 \end{prop}
 
 \noindent{\sl Proof:} This is a standard feature of Grothendieck duality, cf. \cite{conrad}.
 By degree considerations  it suffices to look at the behavior on the $(n-1)$-st cohomology only.  That is, we consider, 
  for each $i$, the map $H^{n-1}(\Res_i): H^{n-1}(\wh x_i^\circ, \Omega^n)\to\k$ induced by 
  $\Res_{x_i}$   on the $(n-1)$st cohomology,
   and  prove that the compositions of these maps with the $H^{n-1}(X^\circ,\Omega^n)\to H^{n-1}(\wh x_i^\circ,
 \Omega^n)$ sum to zero. 
 
  Indeed, $H^{n-1}(\Res_i)$ can be represented as the composition
 \[
 H^{n-1}(\wh x_i^\circ, \Omega^n) \simeq  H^n_{\{x_i\}}(\wh x_i, \Omega^n)  =  H^n_{\{x_i\}}(X, \Omega^n)
 \buildrel\iota_{x_i} \over \lra H^n(X, \Omega^n)  \buildrel\tr\over\lra
 \k,
 \]
 where $\tr$ is the global Serre duality isomorphism. Now the statement follows from the fact that in the long
 exact sequence relating cohomology with and without support in $\x$,  
  \[
\cdots \to  H^{n-1}(X^\circ, \Omega^n) \buildrel \delta \over\lra H^n_{{\x}}(X, \Omega^n)
\buildrel \iota_{\x} = \sum \iota_{x_i} \over  \lra H^n(X, \Omega^n)\to \cdots 
 \]
 the composition of  any two consecutive
 arrows is zero. \qed


\subsection{Explicit models}\label{subsec:exp-mod}

We start with the ``polynomial" dg-algebra $\Aen_{[n]}^\bullet$. By considering the fibration
$\pA^n\to \PP^{n-1}$, we can write
\be\label{eq:aen-[n]-P}
\Aen_{[n]}^\bullet \,\,\sim\,\,  \bigoplus _{i\in\ZZ} R\Gamma(\PP^{n-1}, \Oc(i)). 
\ee
From here, passing to the completion is easy:  it is similar to passing from Laurent polynomials in one variable to
Laurent series. More precisely, for a graded vector space $\bigoplus_{i\in\ZZ} V^i$
we denote
\[
\sum_{i\gg -\infty} V^i \,\, =\,\, \varinjlim_a\,\, \varprojlim_{b} \,\, \bigoplus_{i=-a}^b V^i 
\]
the vector space formed by Laurent series $\sum_{i\gg -\infty} v_i$  with $v_i\in V^i$. Then
\be\label{eq:aen-n-P}
\Aen_n^\bullet \,\,\sim \,\, \sum_{i\gg -\infty} R\Gamma(\PP^{n-1}, \Oc(i))
\ee
Applying various way of calculating the cohomology of $\PP^{n-1}$, we get various explicit
models for $\Aen_n$ and related dg-algebras and modules.

\paragraph{A. The \v Cech model.} Covering $\PP^{n-1}$ with open sets $\{z_i\neq 0\}$,
or, what is the same, covering $D_n^\circ$  with similar open sets right away, we get a model
for $\Aen_n^\bullet$ as the \v Cech complex of this covering. The Alexander-Whitney
multiplication makes this complex into an {\em associative but not commutative} dg-algebra. 

We can use Thom-Sullivan forms to replace this by a commutative dg-algebra model for
$\Aen_n^\bullet$. 


\paragraph{B. The Jouanolou model.} 
Introduce another set of variables $z_1^*, \dots, z_n^*$ which we think of as dual to the 
$z_\nu$, i.e., as the coordinates in the dual affine space $\wc\AAA^n$.
 We write
 \[
 zz^* \,\,=\,\,\sum z_\nu z^*_\nu, \quad z\in \AAA^n, \, z^* \in\wc\AAA^n.
 \]
 We
   form the corresponding ``dual" projective space  
 $\wc{\PP}^{n-1} = \on{Proj} \k[z_1^*, \dots, z_n^*]$ and consider the incidence quadric
 \[
 Q \,\,=\,\,\bigl\{ (z, z^*) \in \AAA^n\times \wc\PP^{n-1}\bigl|\, \, zz^*=0\bigr\}\,\,\subset \,\, \AAA^n\times \wc\PP^{n-1}. 
 \]
 We denote the complement $(\AAA^n\times\wc\PP^{n-1})-Q$ by $J$ and note that the projection to the
  first factor gives a  morphism
\[
\pi: J\lra \pA^n
\]
whose fibers are affine spaces of dimension $n-1$. We refer to $J$ as the {\em Jouanolou torsor}
for $\pA^n$. 
 For further reference let us point out that
\be\label{J=G/G}
J \,\,\simeq \,\,\bigl\{ (z_1, \dots, z_n, z^*_1, \dots, z^*_n) \in \AAA^n\times (\wc\AAA^n-\{0\}) \,| \,\, zz^*=1\bigr\},
\ee
the isomorphism given by the projection $\wc \AAA^n -\{0\} \to \wc\PP^{n-1}$ on the second factor. 

 For any quasi-coherent sheaf $E$
on $\pA^n$ we then have the  global  relative de Rham complex 
\[
A_{[n]}^\bullet (E) \,\,=\,\, \Gamma \bigr(J, \Omega^\bullet_{J/\pA^n}
\otimes  \pi^* E \bigr).
\]
The differential in  $A_{[n]}^\bullet (E)$ (given by the relative de Rham differential)
will be   denoted $\dbar$. 

\vskip .2cm

Let also
\[
\wh J \,\,=\,\, J\times_{\AAA^n} D_n \,\,=\,\, J\times_{\pA^n} D_n^\circ \buildrel\wh\pi
\over\lra D_n^\circ
\]
be the restriction of $J$ to the punctured formal disk. As before, $\wh J$ is an affine scheme
and an $\AAA^{n-1}$-torsor over $D_n^\circ$. For any quasi-coherent sheaf $E$ on $D_n^\circ$
we denote
 \[
A_{n}^\bullet (E) \,\,=\,\, \Gamma \bigr(\wh J, \Omega^\bullet_{\wh J/D_n^\circ}
\otimes  \wh\pi^* E \bigr).
\]

\begin{prop}\label{prop:RGammaJ}
(a)  $A_{[n]}^\bullet (E)$ is a model for $R\Gamma(\pA^n, E)$, and $A_n^\bullet(E)$
is a model for $R\Gamma(D_n^\circ, E)$. 

\vskip .2cm

(b) The functor $E\mapsto A_{[n]}^\bullet (E)$  (resp. $E\mapsto A_n^\bullet(E)$)
is a lax  symmetric monoidal functor from
the category of quasi-coherent sheaves on $\pA^n$ (resp. on $D_n^\circ$)
 to the category of complexes of $\k$-vector
spaces. 
In particular, if $E$ is a quasi-coherent  commutative 
$\Oc_{\pA^n}$-algebra (resp. $\Oc_{D_n^\circ}$-algebra), then  $A_{[n]}^\bullet (E)$
(resp. $A_n^\bullet(E)$)
is a commutative dg-algebra.
\end{prop}

\noindent{\sl Proof:} (a) This is a classical argument. We consider the  only
case of $A_{[n]}^\bullet(E)$. 
Because $J$ is affine, we have quasi-isomorphisms 
 \[
 A_{[n]}^\bullet (E)  \,\sim \, R\Gamma(J, \Omega^\bullet_{J/\PP^{n-1}} \otimes \pi^* E) \,\, \sim \,\, 
 R\Gamma(\pA^n , R\pi_*   (\Omega^\bullet_{J/\pA^n}\otimes \pi^* E) ).
 \]
Because $\pi$ is  a  Zariski locally trivial  fibration with fiber $\AAA^{n-1}$, 
the Poincar\'e lemma for  differential forms on $\AAA^{n-1}$ implies that the embeddings
 \[
 E \hookrightarrow \pi_*(\Omega^\bullet_{J/\pA^n}\otimes E) \hookrightarrow R\pi_*\
 (\Omega^\bullet_{J/\pA^n}\otimes E)
 \]
are quasi-isomorphisms of complexes of sheaves on $\pA^n$, whence the statement. 

\vskip .2cm

(b) Obvious by using the multiplication of differential forms. 

\qed

\vskip .2cm

By the above, the dg-algebras
\be\label{eq:defAn}
A_{[n]}^\bullet = A_{[n]}^\bullet(\Oc_{\pA^n}), \quad A_n^\bullet = A_n^\bullet(\Oc_{D_n^\circ})  
\ee
are  commutative dg-models for $\Aen_{[n]}^\bullet$ and $\Aen_n^\bullet$ respectively. 
Their grading is situated in degrees $[0, n-1]$. 
Let us reformulate their definition closer to \eqref{eq:aen-[n]-P} and
\eqref{eq:aen-n-P}. For this, let
\[
\ol J \,\,=\,\,\bigl\{ (z, z^*)\in \PP^{n-1}\times\wc\PP^{n-1} \,\bigl| \,\, zz^*\neq 0 \bigr\}
\buildrel \ol \pi\over\lra \PP^{n-1}
\]
be the classical Jouanolou torsor for $\PP^{n-1}$. For a quasi-coherent sheaf $F$ on $\PP^{n-1}$
we define
\[
R\Gamma^{(\ol J)}(\PP^{n-1}, F) \,\,=\,\, \Gamma \bigr(\ol J, \Omega^\bullet_{\ol J/\PP^{n-1}}
\otimes  \ol \pi^* F \bigr).
\]
As before, this is a model of $R\Gamma(\PP^{n-1}, F)$, depending on $F$ in a way
compatible with the symmetric monoidal structures. 

\begin{prop}
We have isomorphism of commutative dg-algebras
\[
A_{[n]}^\bullet \,\,=\,\, \bigoplus_i R\Gamma^{(\ol J)} (\PP^{n-1}, \Oc(i)), 
\quad A_n^\bullet \,\,=\,\, \sum_{i\gg\-\infty} R\Gamma^{(\ol J)} (\PP^{n-1}, \Oc(i)). 
\]
\end{prop}

\noindent {\sl Proof:} For a line bundle $\Lc$ on an algebraic variety $Z$ we denote by $\Lc^\circ$
the total space of $\Lc$ minus the zero section, so a $\GG_m$-torsor over $Z$. Now, 
 we have the Cartesian square
\[
\xymatrix{
J \ar[d]_r
\ar[r]^\pi& \pA^n
\ar[d]^q
\\
\ol J\ar[r]_{\ol\pi} &\PP^{n-1}
}
\]
 where $q$ realizes $\pA^n$ as $\Oc(-1)^\circ$. Therefore $r$ realizes $J$ as $(\ol\pi^*\Oc(-1))^\circ$.
 Further, $\Omega^\bullet_{J/\pA^n}= r^*\Omega^\bullet_{\ol J/\PP^{n-1}}$. Therefore
 \[
 \begin{gathered}
 A_{[n]}^\bullet \,=\, \Gamma (J, \Omega^\bullet_{J/\pA^n}) \, =\, 
 \Gamma \bigl( (\ol\pi^*\Oc(-1))^\circ, r^* \Omega_{\ol J/\PP^{n-1}}\bigr)\,=
 \\
 =\, \bigoplus_{i\in\ZZ} \Gamma(\ol J, \pi^*\Oc(i)\otimes \Omega_{\ol J/\PP^{n-1}}\bigr) \,=\,
 \bigoplus_{i\in \ZZ} R\Gamma^{(\ol J)} (\PP^{n-1}, \Oc(i)). 
 \end{gathered} 
 \]
 The proof for $A_n^\bullet$ is similar with direct products instead of direct sums. \qed

\paragraph{C. The Jouanolou model, explicitly.} 
Let
\[
\k[z, z^*] = \k[z_1, \dots, z_n, z^*_1, \dots, z_n^*], \quad \k[[z]][z^*] =
\k[[z_1, \dots, z_n]][z_1^*, \dots, z_n^*]
\]
 be the algebras of regular functions on $\AAA^n \times \wc\AAA^n$ and $D_n\times\wc\AAA^n$
 respectively. Recall the notation $z z^* = \sum z_\nu z_\nu^*$.

 \begin{prop}\label{prop:A-n-forms}
 Let $m=0, \dots, n-1$. 
  The $m$-th graded component $A_{[n]}^m$ (resp. $A_n^\bullet$)
  is identified with the vector space formed by differential forms
  \[
  \omega \,\,=\,\,\sum_{1\leq i_1< \cdots < i_m\leq n} f_{i_1, \dots, i_m}(z,z^*) dz^*_{i_1} \cdots dz^*_{i_m}
  \]
  where each $f_{i_1, \dots, i_m}$ is an element of the localized algebra $\k[z,z^*][(zz^*)^{-1}]$ 
  (resp. $\k[[z]][z^*][(zz^*)^{-1}]$) such that:
  \begin{enumerate}
  \item[(1)] $\omega$  is  homogeneous
  in the $z_\nu^*, dz^*_\nu$ of  total degree $0$, that is, each $f_{i_1, \dots, i_m}$ is homogeneous of degree $(-m)$.
  
  \item[(2)] The contraction  $\iota_\xi(\omega)$ of $\omega$ with the Euler vector field $\xi= \sum z^*_\nu \partial/\partial z^*_\nu$ 
  vanishes. 
  \end{enumerate}
      The differential $\dbar$ is given by
   \[
   \dbar \,\,=\,\,\sum_{\nu=1}^n dz^*_\nu {\partial \over \partial z^*_\nu}.
   \]
   \end{prop}
  
\noindent {\sl Proof:}  We prove the statement about $A_{[n]}^m$; the statement about $A_n^m$
is proved in the same way. 

Consider the product $\AAA^n\times\wc\AAA^n$ 
  and the incidence quadric $\wt Q$ inside it given by the
  same equation $zz^*=0$ as $Q$. Let $U= (\AAA^n\times\wc\AAA^n)-\wt Q$. 
  All forms $\omega$ as in the proposition (not necessarily satisfying the  conditions (1) and (2))
   form the space $\Gamma(U, \Omega^m_{U/\pA^n})$. 
   
      Now, we have a projection
   $p: U\to  J$ of $\pA^n$-schemes with the multiplicative group $\GG_m$ acting simply transitively on the fibers
  (i.e., $U$ is a $\GG_m$-torsor over $J$). The infinitesimal generator of this action is the Euler vector field
  $\xi$. Therefore relative forms from $\Omega^m_{J/\pA^n}$ are
  identified with  sections $\omega$  of  $\Omega^m_{U/\pA^n}$ which satisfy
  \[
  \partial_\xi(\omega) =0, \,\,\, \iota_\xi(\omega) =0, 
  \]
  where $\del_\xi$ is the Lie derivative, see, e.g., \cite{GKZ}. 
  These conditions translate precisely into the conditions
  (1) and (2) of the proposition. \qed

    \vskip .2cm
  
  \begin{cor}\label{cor:filt}
  The dg-algebra $A_{[n]}^\bullet$ carries a natural   filtration  ``by the order of poles"
  \[
  F_r A^m_n \,\,:=\,\,\bigl\{\omega ~ \bigl| \,\,  (zz^*)^{r+m} f_{i_1, \dots, i_m} \,\in \, \k[z, z^*],   \,\,\forall i_1, \dots, i_m \bigl\}. 
  \]
  compatible with differential and product: 
  \[
  \dbar (F_r A^\bullet)\subset F_r A^\bullet, \,\, \,\, (F_r A^\bullet) \cdot (F_{r'} A^\bullet)
  \subset F_{r+r'} A^\bullet. 
  \]
  \end{cor}

  \subsection{Comparison with the Dolbeault and $\dbar_b$-complexes.}\label{subsec:comp-dolb}

\paragraph{A. Comparison with the Dolbeault complex.} In this section we assume $\k=\CC$. 
The notation $\dbar$ for the differential in $A^\bullet_{[n]}$ is chosen to suggest the analogy with the Dolbeault differential in complex analysis. 
In fact, we have the following.

\begin{prop}
Let $\Omega^{0,\bullet}(\CC^n-\{0\})$ be the smooth  Dolbeault complex of  the complex manifold $\CC^n-\{0\}$. 
We have a (unique,  injective)  morphism of commutative dg-algebras
 $\varepsilon: A_{[n]}^\bullet \to \Omega^{0,\bullet}(\CC^n-\{0\})$ 
which sends $z^*_\nu$ to  $\ol z_\nu$ and $dz^*_\nu$ to $d\ol z_\nu$, i.e., 
 \[
 f(z, z^*) dz^*_{i_1} \cdots dz^*_{i_m} \,\,\,\mapsto \,\,\, f(z, \ol {z})|_{\CC^n-\{0\}} d\ol {z}_{i_1} \cdots d\ol {z}_{i_m}. 
 \]
\end{prop}
The proof is obvious once we notice that $zz^*$ is being sent to $z\ol{z} = |z|^2$ which does not vanish on $\CC^n-\{0\}$. \qed

\begin{rem}
 The morphism $\varepsilon$ is not a quasi-isomorphism: it 
 identifies $H^\bullet(A_n^\bullet)$ with the ``meromorphic part" of $H^\bullet (\Omega^{0,\bullet}(\CC^n-\{0\}))=
H^\bullet(\CC^n-\{0\}, \Oc_{\on{hol}})$. 
\end{rem}

 We also notice that $A^\bullet_{[n]}$ is concentrated in degrees $[0,n-1]$ while $\Omega^{0, \bullet}(\CC^n-\{0\})$
is situated in degrees $[0,n]$. To exhibit a better analytic fit for $A_{[n]}^\bullet$, we recall some constructions from
complex analysis. 

\paragraph{B. Reminder on the $\dbar_b$-complex.} Let $X$ be an $n$-dimensional complex manifold and $S\subset X$
a $C^\infty$ real hyper surface (of real dimension $2n-1$). The embedding $S\subset X$ induces on $S$ a differential geometric
structure known as the {\em CR-structure} \cite{beals-greiner}\cite{dragomir}. 

More precisely, let $x\in S$. The $2n-1$-dimensional real subspace $T_xS$ in the $n$-dimensional complex space $T_xX$
has the maximal complex subspace
\[
T_x^\com S \,\,=\,\, T_xS\, \cap \,  i(T_xS) \,\,\subset \,\, T_xX
\]
of complex dimension $(n-1)$. We get a complex vector bundle $T^\com_S$ on $S$ embedded into the real vector bundle $T_S$.
Its complexification splits, in the standard way, as
\[
(T^\com_S)\otimes \CC \,\,=\,\, T^{1,0}_S \,\oplus \, T^{0,1}_S \,\,\subset \,\,  T_S\otimes \CC.
\]
So $T^{1,0}_S$ and $T^{0,1}_S$ are complex vector sub-bundles in $T_S\otimes\CC$, of complex dimension $(n-1)$. 
Integrability of the complex structure on $X$ implies that these sub-bundles are integrable in the sense of Frobenius,
i.e., their sections are closed under the Lie bracket of the sections of $T_S\otimes \CC$. 

\vskip .2cm

Let $^b\Omega_S^{0,q}$ be the sheaf of $C^\infty$-sections of the complex vector bundle $\Lambda^q(T^{0,1}_S)^*$.
Integrability of $T^{0,1}_S$ gives, by the standard Cartan formulas, the exterior differentials ``along" $T^{0,1}_S$
\[
\dbar_b: {}^b\Omega_S^{0,q} \lra {}^b\Omega_S^{0,q+1}
\]
making $^b\Omega_S^{0,\bullet}$ into a sheaf of commutative dg-algebras known as the {\em tangential Cauchy-Riemann
complex} (or {\em $\dbar_b$-complex}) of $S$. It is concentrated in degrees $[0, n-1]$. The complex of global $C^\infty$-sections of   
 $^b\Omega_S^{0,\bullet}$ is traditionally denoted by $\Omega^{0,\bullet}_b(S)$ and is also called the $\dbar_b$-complex. 

\paragraph{C. Generalities on real forms.} Let $Y=\Spec R$ be a smooth irreducible affine variety over $\CC$ of dimension $m$.
A {\em real structure} on $Y$ is a $\CC$-antilinear involution $f\mapsto \ol f$ on $R$. Such a datum
defines an antiholomorphic involution $\sigma: Y(\CC)\to Y(\CC)$ on $\CC$-points. The fixed point locus of $\sigma$
is denoted $Y(\RR)$. If nonempty, $Y(\RR)$ has a structure of a $C^\infty$-manifold of dimension $m$. In this case
we have an embedding $\epsilon: R\to C^\infty(Y(\RR))$. Moreover, any  $f\in C^\infty(Y(\RR))$ can be approximated by
functions from $\epsilon(R)$ on compact subsets (Weierstrass' theorem). In particular, if $Y(\RR)$ is compact,
then $\epsilon(R)$ is dense in $C^\infty(Y(\RR))$ in any of the standard metrics of the functional analysis
(e.g., in the $L_2$-metric). 

\vskip .2cm 

If, further, $E$ is a  vector bundle on $Y$ (not necessarily equipped with a real structure), it gives a $C^\infty$-bundle
$E|_{Y(\RR)}$ on $Y(\RR)$ and an embedding $\Gamma(E)\to \Gamma_{C^\infty}(E|_{Y(\RR)})$ with approximation properties
similar to the above. A differential operator $D: E\to F$ between vector bundles on $Y$ gives rise to a $C^\infty$ differential
operator $D_{Y(\RR)}: E|_{Y(\RR)}\to F|_{Y(\RR)}$. 

\paragraph{D. The Jouanolou model and the $\dbar_b$-complex for $S^{2n-1}$.} 
We specialize part B to $X=\CC^n$ with the  standard coordinates $z_\nu$. We take $S$ to be
 the unit sphere $S^{2n-1}$ with equation $\|z\|^2=1$. We have therefore the $\dbar_b$-complex
 $(^b\Omega^{0,\bullet}_{S^{2n-1}}, \dbar_b)$. 
 
 \vskip .2cm
 
 At the same time, we introduce, on $\AAA^n\times \wc \AAA^n$ (with coordinates $z_\nu, z^*_\nu$),
 a real structure by putting $\ol z_\nu = z^*_\nu$ and $\ol{z^*_\nu}=z_\nu$.  The Jouanolou torsor $J$ is realized in
 $\AAA^n\times \wc \AAA^n$ as the hypersurface $zz^*=1$, so it inherits the real structure. 
 
 \begin{prop}
 (a) We have $J(\RR)=S^{2n-1}$. In other words, $A_{[n]}^0=\CC[J]$ is identified with the algebra of (real analytic)
 polynomial $\CC$-valued functions on $S^{2n-1}$. 
 
 \vskip .2cm
 
 (b) The $C^\infty$-vector bundle $\Omega^q_{J/\pA^n}|_{J(\RR)}$ on $J(\RR)=S^{2n-1}$ is identified with 
 $^b\Omega^{0,q}_{S^{2n-1}}$, and the differential induced by $d_{J/\pA^n}$, is identified with $\dbar_b$.
 In other words, $A_{[n]}^\bullet$ is identified with the dg-algebra of polynomial (in the same sense as in (a))
 sections of $^b\Omega^{0,\bullet}_{S^{2n-1}}$. 
 \end{prop}
 
 \noindent{\sl Proof:} Part (a) is obvious, as $zz^*=1$ translates to $z\ol z=1$. To prove (b), we notice
 that the sub-bundle
 \[
 T_{J/\pA^n}|_{S^{2n-1}} \,\,\subset \,\, T_J|_{S^{2n-1}} = T_{S^{2n-1}}\otimes \CC
 \]
 is equal to $T^{0,1}_{S^{2n-1}}$. \qed
 
 \begin{cor}\label{cor:L2-dense}
 $A_{[n]}^q$ is dense in the $L_2$-completion of $\Omega^{0,q}_b(S^{2n-1})$. In particular, $A_{[n]}^0$
 is dense in the $L_2(S^{2n-1})$. \qed
 \end{cor}

\subsection{Representation-theoretic analysis}\label{subsec:repan}

\paragraph{A. The $GL_n$-spectrum of $A_n^\bullet$.} 
The Jouanolou model   gives  commutative dg-algebras  with a natural
action of the algebraic group $GL_n$. We now study this action.

\vskip .2cm

Let $V$ be a   $\k$-vector space of  finite dimension $n$. Recall 
 \cite{fulton-harris} that irreducible representations of $GL(V)$ are labelled by their  highest weights
which are sequences of non-increasing integers $\alpha=(\alpha_1 \geq  \cdots \geq \alpha_n)$, possibly negative
(dominant weights). We will denote the underlying space of the irreducible representation
with highest weight $\alpha$ by $\Sigma^\alpha V$ and regard $\Sigma^\alpha$ as a functor (known as the {\em Schur functor}) from the category of $n$-dimensional
$\k$-vector spaces and their isomorphisms,  to $\Vect_\k$.  Here are some tie-ins with more familiar constructions.  

\begin{exas}
(a) For $d\geq 0$ we have $\Sigma^{d,0,\dots, 0} (V) = S^d(V)$ is the $d$th symmetric power of $V$. For $0\leq p \leq n$ let $1_p = (1, \dots, 1, 0, \dots, 0)$
(with $p$ occurrences of $1$). Then $\Sigma^{1_p}(V)=\Lambda^p(V)$ is the $p$th exterior power of $V$. 

\vskip .2cm

(b) We have canonical identifications  
\[
\Sigma^{\alpha_1, \dots, \alpha_n} (V)^* \,\, \simeq \,\, \Sigma^{\alpha_1, \dots, \alpha_n} (V^*) 
\,\, \simeq\,\,  \Sigma^{-\alpha_n, \dots, -\alpha_1}(V).
\]
In particular, $V^* = \Sigma^{0, \dots, 0, -1}(V)$. Further, 
\[
\Sigma^{1, 0, \dots, 0, -1}(V) \,\,\simeq \,\, \mathfrak{sl} (V)\,\,=\,\, \bigl\{A\in\End(V): \, \tr(A)=0\bigr\}. 
\]
 \end{exas}
 
 Recall  also the following fact from \cite{weyl}.

 \begin{prop}\label{prop:weyl} 
For any dominant  weight $\alpha$ for $GL_n$ the restriction of $\Sigma^\alpha(\k^n)$ to $GL_{n-1}$ has simple spectrum, explicitly given by:
\[
\Sigma^\alpha(\k^n)|_{GL_{n-1}} \,\,\simeq \,\,\bigoplus_{\alpha_1 \geq \beta_1 \geq\alpha_2\geq\beta_2\geq \cdots\geq \beta_{n-1}\geq \alpha_n}
\Sigma^{\beta_1, \dots, \beta_{n-1}}(\k^{n-1}).
\]
\end{prop}
 
  We now denote by $V=\k^n$ the space of linear combinations of the coordinate functions $z_i$ on $\AAA^n$, so $\AAA^n=\Spec \, S^\bullet(V)$,
 and let $GL_n=GL(V)$.  
 The Jouanolou torsor $\ol J\to \PP^{n-1}$ is acted upon by $GL_n$, and so the dg-algebra $A^\bullet_{[n]}$
 as well as its completion $A_n^\bullet$,  inherit the $GL_n$-action.
 The following fact can be seen as a higher-dimensional generalization of the
  representation of elements of $\k[z, z^{-1}]$ as linear
 combinations of Laurent monomials (irreducible representations of $GL_1$). 
 It will be used in Proposition \ref{prop:MB-H}.

  \begin{thm}\label{thm:irr}  (a) As a $GL_n$-module, each  $A_{[n]}^p$ has simple spectrum, that is, each  $\Sigma^\alpha(V)$  enters into the
 irreducible decomposition of 
 $A^p_{[n]}$ no more than once. 
 
 \vskip .2cm
 
 (b) More precisely, $\Sigma^\alpha(V)$ enters into $A_{[n]}^p$ if and only if
 \[
 \alpha_1 \geq 0 \geq \alpha_2 \geq 0 \geq \cdots \geq 0\geq \alpha_{n-p}\geq -1 \geq \alpha_{n-p+1}\geq -1 \geq \cdots \geq -1 \geq \alpha_n. 
 \]
 \end{thm}
 
  \begin{exas}\label{ex:A-n-irr}
 (a) For $n=1$ the only possible $p$ is $p=0$. In this case the condition on $\alpha=(\alpha_1)\in\ZZ$
 is vacuous and the theorem says that $A_{[1]}^0 = \bigoplus_{\alpha\in\ZZ} \Sigma^\alpha(\k)= \k[z, z^{-1}]$. 
 
 \vskip .2cm
 (b) Let $n=2$. In this case the theorem says that we have an identification of complexes of $GL_2$-modules 
 \[
 A_{[2]}^\bullet \,\, = \,\,\biggl\{  \bigoplus_{\substack{ \alpha_1\geq 0 \\ \alpha_2 \leq 0} } \Sigma^{\alpha_1, \alpha_2}(V) \buildrel \dbar\over\lra
  \bigoplus_{\substack{ \alpha_1\geq -1 \\ \alpha_2 \leq -1} } \Sigma^{\alpha_1, \alpha_2}(V)\biggr\}. 
 \]
From this we see the identifications
\[
\begin{gathered}
\Ker(\dbar) \,\,=\,\, \bigoplus_{\alpha_1\geq 0} \Sigma^{\alpha_1, 0}(V) = \k[z_1, z_2], \\
\Coker (\dbar) \,\,=\,\, \bigoplus_{\alpha_2 \leq -1} \Sigma^{-1, \alpha_2}(V) = z_1^{-1} z_2^{-1} \k[z_1^{-1}, z_2^{-1}],
\end{gathered}
\]
the other irreducible representations, common to $A_{[2]}^0$ and $A_{[2]}^1$,   are cancelled by the action of $\dbar$. 

\vskip .2cm

(c) For $n=3$ the theorem identifies $A_{[3]}^\bullet$, as a complex of $GL_3$-modules, with
\[
\bigoplus_{\substack {\alpha_1\geq 0 \\ \alpha_3\leq 0}} \Sigma^{\alpha_1, 0, \alpha_3} 
\buildrel \dbar \over\lra \bigoplus_{\alpha_1\geq 0\geq \alpha_2\geq -1\geq \alpha_3} \Sigma^{\alpha_1, \alpha_2, \alpha_3}
\buildrel \dbar \over\lra   
\bigoplus_{\substack {  \alpha_1\geq -1 \\ \alpha_3\leq -1    }} \Sigma^{\alpha_1, -1, \alpha_3}. 
\]
One can see, for example, the reason why the complex is exact in the middle: if $\alpha_2=0$, then $\Sigma^{\alpha_1, \alpha_2, \alpha_3}$
lies in $\Im(\dbar)$, while if $\alpha_2=-1$, then $\Sigma^{\alpha_1, \alpha_2, \alpha_3}$ is mapped by $\dbar$ isomorphically to its image. 
\end{exas}

The proof of Theorem \ref {thm:irr} is based on the following observation.  We use the definition
  of $A_{[n]}$ in terms of the torsor $J$.

\begin{prop}\label{prop:J=G/G}
As a variety with $GL_n$-action, $J\simeq GL_n/GL_{n-1}$. Further, 
$\Omega^p_{J/\pA^n}$  is identified with the homogeneous vector bundle
on $GL_n/GL_{n-1}$ associated to the representation $\Lambda^p((\k^{n-1})^*)$ of $GL_{n-1}$. 
\end{prop}

\noindent {\sl Proof:}  The identification \eqref{J=G/G}
  exhibits  $J$ as  a homogeneous space under $GL_n$. The stabilizer of the point $(z, z^*)$ where
$z= (1, 0, \dots, 0)$ and $z^* = (1, 0, \dots, 0)$, consists of matrices of the form
\[
\begin{pmatrix}
1 & 0\\
0 & A
\end{pmatrix}, \quad A\in GL_{n-1},
\]
whence the first  statement of the proposition. To see the second statement, look at the action of the stabilizer subgroup $GL_{n-1}$ on the
fiber of $\Omega^p_{J/\pA^n}$ over the chosen point $(z,z^*)$ above. By definition, this fiber is the $p$th exterior power of the relative
cotangent space at this point. It remains to notice that the relative {\em tangent} space is, as a $GL_{n-1}$-module, nothing but $\k^{n-1}$. \qed

\vskip .2cm

We now prove Theorem \ref {thm:irr}. By Proposition \ref{prop:J=G/G},
\[
A_{[n]}^p = \Gamma(J, \Omega^p_{J/\pA^n)}) \,\,\simeq \,\, \Ind_{GL_{n-1}}^{GL_n} \Lambda^p (\k^{n-1})^*
\]
  (induction in the sense of algebraic groups, i.e., 
   via regular sections of the homogeneous bundle). So we can apply Frobenius reciprocity and obtain:
\[
\mult \bigl( \Sigma^\alpha \k^n,  \Ind_{GL_{n-1}} ^{GL_n} \Lambda^p (\k^{n-1})^* \bigr) \,\,=\,\,\mult \bigl(
\Lambda^p(\k^{n-1})^* , \Sigma^\alpha \k^n |_{GL_{n-1}}
\bigr),
\]
where $\mult$ means the multiplicity of an irreducible representation. Now, 
$\Lambda^p(\k^{n-1})^* = \Sigma^{0, \dots, 0, -1, \dots, -1}(\k^{n-1})$
(with $p$ occurrences of $(-1)$). It remains to apply Proposition \ref{prop:weyl}. 
 Theorem \ref {thm:irr} is proved. 

\paragraph{B. The Martinelli-Bochner form and its multipoles.} 

 We use the analogy with the Dolbeault complex
 as a motivation for the following. 

\begin{ex}\label{ex:MB}
The {\em Martinelli-Bochner form}
\[
\Omega = \Omega(z,z^*)  \,\,=\,\, {  \sum_{\nu=1}^n
(-1)^ {\nu -1} z^*_\nu   dz^*_1 \wedge \cdots \wedge \wh{ dz^*_\nu}
\wedge \cdots \wedge dz^*_n 
\over 
(zz^*)^n }
\]
is an element of $A^{n-1}_{[n]}$. For $n=1$ it reduces to $1/z_1$. 
\end{ex}

The following will be used in Proposition \ref{prop:Hpq-An} and Theorem \ref{thm:res-invariant}. 

\begin{prop}\label{prop:MB-H}
 Let $n>1$. 

(a) The class of $\Omega$ in $H^{n-1}(A^\bullet_{[n]})= H^{n-1}(\pA^n, \Oc)$
is a generator of the 1-dimensional subspace $z_1^{-1} \cdots z_n^{-1}$ of weight $(-1, \dots, -1)$
under the coordinate torus, see Proposition \ref{prop:H^iO}. 

\vskip .2cm

(b) The top degree part $A_{[n]}^{n-1}$ contains precisely one 1-dimensional irreducible representation of $GL_n$,
which is $\Sigma^{-1, \dots, -1}(V) = \Lambda^n(V^*)$. This subspace is spanned by the Martinelli-Bochner form 
$\Omega(z,z^*)$. 

\vskip .2cm

(c)
Further, every element of $H^{n-1}(A^\bullet_{[n]})$
can be represented as the class of a ``multipole"
\[
P(\partial_{z_1}, \dots, \partial_{z_n}) \Omega(z, z^*)
\]
for a unique polynomial $P(y_1, \dots, y_n)$. 
\end{prop}

\noindent{\sl Proof:}  (b) follows from Theorem \ref{thm:irr}. To deduce (a) from (b),
note that the class $[z_1^{-1}\cdots z_n^{-1}]\in H^{n-1}(\pA^n, \Oc)$ spans a 1-dimensional representation
of $GL_n$, and so does $\Omega(z,z^*)$ (direct calculation).  At the same time 
$\Sigma^{-1,\dots, -1}(V)$ is not present in $A_{[n]}^{n-2}$, so $[\Omega(z,z^*)]$ is a nonzero
scalar multiple of $[z_1^{-1}\cdots z_n^{-1}]$.

 To prove (c), 
we notice the following fact which  complements Proposition 
\ref{prop:H^iO} and is proved using the same standard affine covering $\{z_i\neq 0\}$.
 Note that the ring $\k[\del_{z_1}, \dots, \del_{z_n}]$ of differential operators with constant coefficients
acts naturally on $\Oc_{\pA^n}$ and therefore on $H^{n-1}(\pA^n, \Oc)$. 

\begin{prop} As a $\k[\del_{z_1}, \dots, \del_{z_n}]$-module, the space $H^{n-1}(\pA^n, \Oc)$ is free of rank 1,
with generator
  $\delta = [z_1^{-1}\cdots z_n^{-1}]\in H^{n-1}(\pA^n, \Oc)$. \qed
 \end{prop}
 
  Proposition \ref {prop:MB-H} is proved.

\begin{rem}\label{rems:jou=sphere}

The fact that $A_{[n]}^0$ has simple spectrum, allows us  to define a canonical $GL_n$-equivariant  projection $S: A_{[n]}^0\to \k[z_1,\dots, z_n]$
along all the irreducible representations which do not enter into $\k[z_1,\dots, z_n]$. This is the algebraic analog of the classical 
{\em Szeg\"o  projection} from complex Hilbert space $L_2(S^{2n-1})$ to the Hardy space formed by the boundary values of functions
holomorphic in the ball $\|z\|^2<1$. See, e.g.,   \cite{BdMG}.

\end{rem}

\subsection{Residues and duality}\label{subsec:resdual}
  
\paragraph{A. Jouanolou model for forms.} 
We denote
  \be
  A_{[n]}^{p, q} \,\,=\,\, A_{[n]}^q(\Omega^p_{\pA^n}), \quad  A_{[n]}^{\bullet\bullet}\,\,=\,\, 
  \bigoplus_{p=0}^n \bigoplus_{q=0}^{n-1} A_{[n]}^{p,q}. 
  \ee
  
  Elements of $A_{[n]}^{p,q}$ can be viewed as differential forms on $\AAA^n\times \wc\PP^n$
  with poles on the quadric $Q$ given by $zz^*=0$. Let $\wt Q=\{zz^*=0\}$
  be the hypersurface in $\AAA^n \times\wc\AAA^n$
   lifting $Q$. 
   By pulling back from $\wc\PP^{n-1}$  to $\wc\AAA^n-\{0\}$, we can view
  elements of $A_{[n]}^{p,q}$ as  differential forms on $(\AAA^n \times\wc\AAA^n)-\wt Q$ of the form
 \be\label{eq:formApq}
  \omega \,\,= \sum f_{j_1, \dots, j_q}^{i_1, \dots, i_p} (z, z^*) dz_{i_1} \wedge \cdots \wedge dz_{i_p}
  \wedge dz^*_{j_1}\wedge \cdots \wedge dz^*_{j_q}, 
  \ee
  which have total degree 0 in the $z^*_\nu, dz^*_\nu$ and  are annihilated by contraction with the vector field $\sum z_\nu^* \del/\del z_n^*$.   
  \vskip .2cm

 It follows that 
  the bigraded vector space $A_{[n]}^{\bullet\bullet}$ is a graded commutative algebra, with  respect to
   multiplication
  of forms, with grading  situated in degrees $[0,n]\times[0,n-1]$. 
   It is equipped with two anticommuting
  differentials: $\dbar = \sum dz^*_\nu \del/\del z^*_\nu$ of degree $(0,1)$ 
    and $\del = \sum dz_\nu \del/\del z_\nu$ of  degree $(1,0)$ which correspond to  exterior differentiation along the
    two factors in $\AAA^n\times\wc\PP^{n-1}$. One can say that $\dbar$ is 
 induced by the relative de Rham differential in $\Omega_{J/\pA^n}\otimes \pi^*\Omega^p$ and
    $\del$ corresponds to the de Rham differential 
  $d: \Omega^p_{\pA^n}\to\Omega^{p+1}_{\pA^n}$.  Part (a) of Proposition \ref{prop:RGammaJ}
  implies:

  \begin{prop}
  The bigraded dg-algebra $A_{[n]}^{\bullet\bullet}$ is a commutative dg-model for
  $\Aen_{[n]}^{\bullet \bullet}$. \qed
  \end{prop}
  
    \begin{ex}
  Let $\k=\CC$. Then we have an embedding of commutative bigraded bidifferential algebras
  \[
  \eps: A_n^{\bullet\bullet} \hookrightarrow \Omega^{\bullet\bullet}(\CC^n-\{0\}),
  \]
  where on the right we have the algebra of $C^\infty$ Dolbeault forms on $\CC^n-\{0\}$ with its standard differentials $\del$
  and $\dbar$. The value of $\eps$ on a form $\omega$ given by \eqref{eq:formApq}, is
  \[
  \eps(\omega) \,\,= \sum  f_{j_1, \dots, j_q}^{i_1, \dots, i_p} (z, \ol z)|_{\CC^n-\{0\}}  dz_{i_1} \wedge \cdots \wedge dz_{i_p}
  \wedge d\ol z_{j_1}\wedge \cdots \wedge d\ol z_{j_q}
  \]
  \end{ex}
  
\paragraph{B. The residue map.} 
  Since $\Omega^n_{\pA^n}$ is identified, as a $GL_n$-equivariant coherent sheaf on $\pA^n$, with $\Oc_{\pA^n}\otimes_\k \Lambda^n(\k^n)$,
  we have a $GL_n$-invariant identification
  \[
  A^{n, \bullet}_{[n]} \,\,\simeq \,\, A_{[n]}^\bullet \otimes_\k \Lambda^n(\k^n). 
  \]
  We define the {\em residue map}
  \be\label{eq:def-res}
  \Res: A_{[n]}^{n, n-1} \lra \k
  \ee
  as the composition  
  \[
  A_{[n]}^{n, n-1} \,=  \, A_{[n]}^n \otimes_\k \Lambda^n(\k^n)\lra \Lambda^n(\k^n)^*\otimes \Lambda^n(\k^n) =\k,
  \]
  where $A_{[n]}^{n-1}\to\Lambda^n(\k^n)^*$ is the unique $GL_n$-invariant projection
 which takes the Martinelli-Bochner form $\Omega(z, z^*)$ to $dz_1^*\wedge \cdots\wedge dz_n^*$, see Proposition
  \ref{prop:MB-H}(b). 
 Thus, by definition,
  \be\label{eq:res-taut}
  \Res\bigl(\Omega(z,z^*) dz_1\wedge \cdots \wedge dz_n\bigr) \,\,=\,\, 1. 
  \ee
  Note that $A_{[n]}^{n, n-1}$ is the last graded component of $A_{[n]}^\bullet(\Omega^n_{\pA^n})$
  which is a dg-model for $R\Gamma(\pA^n, \Omega^n)$. 
  
   \begin{prop}
   For any $f(z)\in \k[z_1, \dots, z_n]$ we have 
  \[
  \Res\bigl(f(z) \Omega(z,z^*) dz_1\wedge \cdots \wedge dz_n\bigr)\,\,=\,\, f(0)
  \]
 (the algebraic Martinelli-Bochner formula).

  \end{prop}
  
  \noindent {\sl Proof:} Note that both sides of the proposed equality are $GL_n$-invariant functionals of 
  $f\in \k[z_1, \dots, z_n]
   =\bigoplus_{d\geq 0} S^d(V)$.  
  For   $d>0$ the space $S^d(V)$ does not admit any $GL_n$-equivariant
 functionals, so the LHS factors through the projection to $d=0$ which is nothing but the evaluation at $0$. 
 So our statement reduces to \eqref {eq:res-taut}. \qed. 
 
 \vskip .2cm
 
 The following proposition established compatibility of our algebraic residue formalism with the more
 analytic one based on integration of differential forms. It will not be used in the sequel. 
 
   \begin{prop}
  Let $\k=\CC$. Then for any $\omega\in A_{[n]}^{n, n-1}$ we have
  \[
\Res(\omega)  \,\,=\,\,{ (n-1)!\over (2\pi i)^n}   \oint_{S^{2n-1}} \eps(\omega) 
  \]
  where the integral is taken over any sphere $\|z\|=R$  in $\CC^{n}-\{0\}$. 
  \end{prop}
  
  \noindent {\sl Proof:} Each $\eps(\omega)$, $\omega\in A_{[n]}^{n, n-1}$, is a closed $(n, n-1)$-form on $\CC^n-\{0\}$,
  so its integrals over all spheres as above are equal. We see that the RHS of the proposed equality is
  a $GL_n(\CC)$-invariant functional on $A_{[n]}^{n, n-1}$, and any such functional, by Theorem \ref {thm:irr},
  should factor through the projection to $\Lambda^n(\k^n)^*\otimes\Lambda^n(\k^n)$, i.e., through the
  residue map. This means the statement holds up to a universal constant depending only on $n$.
  To see that this constant is 1, we invoke the classical Martinelli-Bochner formula \cite{griffiths-harris},
   which gives
  \[
  \oint_{S^{2n-1}} \Omega(z, \ol z) dz_1 \cdots dz_n \,\,=\,\,{ (2\pi i)^n\over (n-1)!}. 
  \]
  \qed

  \begin{prop}\label{prop:resdual1}
  Each irreducible representation of $GL_n$ enters into $A_{[n]}^{p,q}$ with at most finite multiplicity.
   \end{prop}
   
    \noindent{\sl Proof:}   As a $GL_n$-module, $A_{[n]}^{p,q} = \Lambda^p(V) \otimes A_{[n]}^q$.
 By the Pieri formula \cite{fulton-harris}  the irreducible components $\Sigma^\beta (V)$
 of $\Lambda^p(V)\otimes\Sigma^\alpha(V)$ all satisfy
 $\beta = \alpha + e_{i_1} + \cdots + e_{i_p}$ for some $1\leq i_1 < \cdots < i_p \leq n$. Here $e_i$
 is the $i$th basis vector. So the allowed $\beta$  situated in some fixed radius neighborhood of
 $\alpha$ in $\ZZ^n\subset \RR^n$. This means that tensoring a simple spectrum 
 representation with $\Lambda^p(V)$
 gives a representation with finite multiplicities. 
  So our statement follows from Theorem \ref{thm:irr}. \qed
  
  \vskip .2cm
  
 Let  $E \simeq  \bigoplus_\alpha C_\alpha\otimes\Sigma^\alpha(V)$ be a  representation  
  of $GL(V)$ with finite multiplicities (so $C_\alpha$ are finite-dimensional vector spaces). 
  We define the {\em restricted dual} of $E$ to be
 \[
 E^\bigstar \,\,:= \,\, \bigoplus_\alpha C_\alpha^* \otimes \Sigma^\alpha(V)^*. 
 \]
  
  \begin{prop}\label{prop:resdual2}
The ($GL_n$-invariant) residue pairing
  \[
  (\alpha, \beta) \mapsto \Res(\alpha\cdot\beta): A_{[n]}^{p,q} \otimes_\k A_{[n]}^{n-p, n-1-q} \lra \k
  \]
   gives an isomorphism $A_{[n]}^{n-p, n-1-q}\to (A_{[n]}^{p,q})^\bigstar$. 
  \end{prop}
  
 \noindent {\sl Proof:} We first prove a weaker statement: 
   $(A_{[n]}^{p,q})^\bigstar$ is isomorphic, as a $GL_n$-module, to $A_{[n]}^{n-p, n-1-q}$. 
 Because of the non-degenerate pairing 
 $\Lambda^p(V)\otimes\Lambda^{n-p}(V)\to\Lambda^n(V)$, the statement reduces to the isomorphism
 \[
 (A_{[n]}^q)^\bigstar \,\,\simeq \,\, A_{[n]}^{n-1-q} \otimes\Lambda^n(V)^*. 
 \]
 This isomorphism follows at once from inspecting  the irreducible components of $A_{[n]}^q$ and $A_{[n]}^{n-1-q}$
 given by Theorem \ref{thm:irr}. 
 They are in bijection $\Sigma^\alpha \leftrightarrow\Sigma^\beta$, so that 
 \be\label{eq:alpha-beta}
 \beta_i = -1-\alpha_{n-i}, \,\,\, i=1,\dots, n,  
 \ee
 which means that
 $\Sigma^\beta(V) \simeq \Lambda^n(V)^* \otimes (\Sigma^\alpha(V))^*$.
 \vskip .2cm
 
 We now prove that the residue pairing is actually an isomorphism as claimed.
 As before, we reduce to considering the pairing
 \[
 A_{[n]}^q\otimes A_{[n]}^{n-1-q} \lra  A_{[n]}^{n-1} \lra \Lambda^n(V)^*.
 \]
 Only the summands $\Sigma^\alpha(V)\subset A_{[n]}^q$ and $\Sigma^\beta(V)\subset A_{[n]}^{n-1-q}$ where
 $\alpha,\beta$ correspond to each other as in \eqref{eq:alpha-beta}, can pair in a non-trivial way. 
  It remains to show that they indeed pair non-trivially. If they pair trivially, then the subspace $\Sigma^\alpha(V)
 \subset A_{[n]}^q$ is orthogonal to the entire $A_{[n]}^{n-1-q}$. 
  The easiest way to see why this is impossible,  is
 to reduce (by the Lefschetz principle)  to $\k=\CC$. In this case we can  use the fact that 
 $A_{[n]}^{n-1-q}$ is $L_2$-dense in $\Omega_b^{0, n-1-q}(S^{2n-1})$, see  Corollary \ref{cor:L2-dense}, 
 and the 
non-degeneracy of the $L_2$-pairing on
 $\Omega^{0,\bullet}_b(S^{2n-1})$. 
 
 \qed
  
 \vskip .2cm

  \begin{prop}[Algebraic Stokes formula]\label{prop:res-d=0}
  The residue functional $\Res: A_{[n]}^{n, n-1}\to\k$ vanishes on
  $\del(A_{[n]}^{n-1,n-1}) + \dbar(A_{[n]}^{n, n-2})$. 
  \end{prop}
  
  \noindent {\sl Proof:}  
  As before, the easiest proof is to reduce to $\k=\CC$ and to embed
  $A_{[n]}^{\bullet\bullet}$ into $\Omega^{\bullet\bullet}(\CC^n-\{0\})$. Then we can  use the classical Stokes
  formula for $d=\del+\dbar$, noticing that  elements of  $A_{[n]}^{n-1,n-1}$ are annihilated by
  $\dbar$,  while elements of $A_{[n]}^{n, n-2}$ are annihilated by $\del$. 
  
  A purely algebraic proof,  by inspection of  the  possible relevant irreducible components,  is left to the reader.
 In this inspection we find that  $A_{[n]}^{n, n-2}$ does not contain the  trivial representation, while  $A_{[n]}^{n-1,n-1}$
  does contain it but the corresponding subspace is annihilated by $\del$ (it represents 
  $H^{n-1}(\PP^{n-1}, \Omega^{n-1})$). 
  \qed

\vfill\eject

\section{The residue class in cyclic cohomology}\label{sec:resclass}

\subsection{Cyclic homology of dg-algebras and categories}\label{subsec:cyc-hom-alg-cat}

In this subsection we compare various definitions of cyclic homology of associative dg-algebras
without any restriction on the grading. Care is needed, since existing treatments of some issues 
apply only to $\ZZ_{\leq 0}$-graded algebras  and rely on spectral sequences which may not converge
in the general case.

 For general background on cyclic homology (ungraded and $\ZZ_{\leq 0}$-graded cases), see \cite{loday}. 

\paragraph{A. General definitions.} Our
 basic approach is that of Keller \cite{keller-dgalg}. 
 That is, let  $(A, \dbar)$ be any associative dg-algebra over $\k$, possibly without unit.  
 The Hochschild and the cyclic complexes of $A$, denoted $C_\bullet^\Hoch (A)$ and $CC_\bullet(A)$, are defined 
 similarly  to  what is described in \cite{loday} 5.2.2 for the $\ZZ_{\leq 0}$-graded case. That is,  we form the total
  complex of the double or triple complex obtained when we take into account the grading and differential on $A$. 
  For example, 
 \[
 C_\bullet^\Hoch (A)\,\,=\, \, \Tot \, \bigl \{ \cdots  \buildrel b\over\lra A^{\otimes 3} \buildrel b\over\lra A\otimes A \buildrel  b\over\lra A\bigr\}. 
 \]
 The total complex here and  elsewhere is always understood in the sense of {\em direct sums}.
 The new phenomenon compared to the $\ZZ_{\leq 0}$-graded case is that direct sums  can be infinite.

 The definition of  $CC_\bullet(A)$ is similar. More precisely, let $V$ be a cochain complex of $\k$-vector spaces
 with a $\ZZ/(n+1)$-action, the generator of the action denoted by $t$.
 We denote
 \[
 V_C \,\,=\,\,\Tot \bigl\{ \cdots \buildrel 1-t\over\lra V \buildrel 1+\cdots + t^n\over\lra V\buildrel 1-t\over
 \lra V\bigr\},
 \] 
 (the horizontal grading situated in degrees $\leq 0$). Then
 \[
 CC_\bullet(A) \,\,=\,\,
  \Tot \, \bigl \{ \cdots  \buildrel b\over\lra A^{\otimes 3}_C \buildrel b\over\lra (A\otimes A)_C \buildrel  b\over\lra A_C\bigr\}.
 \]
 where the $\ZZ/(n+1)$-action on $A^{\otimes (n+1)}$ is given as in \cite{loday} (2.1.0), understood via the
 Koszul sign rule.

 The homology of $C_\bullet^\Hoch (A)$  and $CC_\bullet(A)$ will be   denoted $HH_\bullet(A)$ and $HC_\bullet(A)$.  
  Each of these  complexes has an exhaustive increasing filtration by the number of tensor factors. This
  gives a convergent spectral sequence
 \[
 E_2 \, =\, HC_\bullet (H^\bullet_{\dbar}(A)) \,\,\Rightarrow \,\, HC_\bullet(A),
 \]
 and similarly for $HH_\bullet(A)$. 
 It follows that the functors $HC_\bullet$ and $HH_\bullet$ take quasi-isomorphisms  of dg-algebras to isomorphisms
 and so descend to functors on
  the homotopy category of associative dg-algebras. 
  
  \begin{lem}\label{lem:CC-lambda}
   Let $V$ be a cochain complex of $\k$-vector spaces
 with a $\ZZ/(n+1)$-action. Then the morphism $V_C\to V_{\ZZ/(n+1)}$ from the last term to the cokernel of $(1-t)$,
 is a quasi-isomorphism.
  \end{lem} 
  
  \noindent {\sl Proof:} This is well known if $V$ in ungraded 
  (Tate resolution). When $V$ is $\ZZ_{\leq 0}$-graded,
  it follows from a spectral sequence argument, as $V_C$ has an increasing exhaustive filtration with quotients
  $V^i_C$. To prove the general case, it suffices to consider the case when $\k$ is algebraically closed.
  Assuming this, we consider the abelian category $\dgVect^{\ZZ/(n+1)}$ formed by
  cochain complexes with $\ZZ/(n+1)$-action.
  
  \begin{lem}\label{lem:com-indec}
  Each object of $\dgVect^{\ZZ/(n+1)}$ is isomorphic to a direct sum of (possibly infinitely many copies)
  of the following indecomposable objects:
  \begin{itemize}
  \item[(1)] A 1-term complex $\k$ (situated in some degree) on which $\ZZ/(n+1)$ acts via some character.
  
  \item[(2)] A 2-term complex $\k \buildrel \Id\over\to \k$ (situated in a pair of adjacent degrees) on which
   $\ZZ/(n+1)$ acts via the same character.
  \end{itemize}
  \end{lem}
  
  Lemma \ref{lem:com-indec} implies Lemma \ref{lem:CC-lambda} because the indecomposable objects
  are bounded and so \ref{lem:CC-lambda} holds for them.
  
  \vskip .2cm
  
  \noindent {\sl Proof of Lemma \ref{lem:com-indec}:} This is well known if $n=0$ (i.e., if we consider just cochain
  complexes with no group action). Now, given $V$ with action of $\ZZ/(n+1)$, we first split it (as a complex)
  into a  direct sum of eigencomplexes $V_\chi$ corresponding to the characters $\chi$  of $\ZZ/(n+1)$.
   Then we decompose each complex $V_\chi$ into indecomposables. \qed
  
  \begin{cor}\label{cor:HC-Hlambda}
  The complex $CC_\bullet(A)$ is quasi-isomorphic to Connes' complex
  \[
  C_\bullet^\lambda (A) 
  \,\,=\,\,
  \Tot \, \bigl \{ \cdots  \buildrel b\over\lra A^{\otimes 3}_{\ZZ/3} \buildrel b\over\lra (A\otimes A)_{\ZZ/2} \buildrel  b\over\lra A
  \bigr\}.
  \]
  \end{cor}

\paragraph{B. Morita invariance.} For unital dg-algebras, Keller  \cite{keller-dgalg} proved that cyclic homology is Morita-invariant.
 To formulate the results compactly, it is convenient to extend the definition of Hochschild and cyclic complexes and cohomology
  to small dg-categories
 $\Ac$
 (a unital dg-algebra is the same as a dg-category with one object). That is, we define $C_\bullet^\Hoch (A)$ to be the total
 complex of the double complex
 \[
 \cdots \buildrel b\over\to  \bigoplus_{a_0, \dots, a_n\in\Ob(\Ac)} \bigotimes_{i=0}^n
  \Hom^\bullet_\Ac(a_i, a_{i+1}) \buildrel b\over\to  \cdots \to \bigotimes_{a_0\in\Ob(\Ac)} \Hom^\bullet_\Ac(a_0, a_0), 
 \]
 (here we put $a_{n+1}:=a_0$) and similarly for $CC_\bullet(\Ac)$. 
 
 We recall  \cite {toen-morita}
 that a functor $f: \Ac\to\Bc$ is called a {\em Morita equivalence}, if $f_*: \Perf_\Ac\to\Perf_\Bc$ is a quasi-equivalence. 
 Recall further that $\Perf_\Ac$ is essentially small, if $\Ac$ is small, and the canonical  (Yoneda)
 embedding $\upsilon: \Ac\to\Perf_\Ac$ is a Morita equivalence.

 \begin{prop}[Keller, \cite{keller-cyex}] \label{prop:HC-morita}
 \begin{itemize}
 \item[(a)] If $f:\Ac\to\Bc$ is a Morita equivalence of small dg-categories, then $HC_\bullet(f): HC_\bullet(\Ac)\to HC_\bullet(\Bc)$ is
 an isomorphism.
 
\item[(b)] In particular, if $\Ac$ is an essentially small dg-category, then $HC_\bullet(\Ac')$ where $\Ac'\subset\Ac$ is an equivalent
 small dg-subcategory, are canonically identified, and denoted $HC_\bullet(\Ac)$.
 
 \item[(c)]  It follows that $\upsilon$ induces an isomorphism $HC_\bullet(\Ac)\to HC_\bullet(\Perf_\Ac)$.
 
 \item[(d)]  Therefore any dg-functor $\phi: \Perf_\Ac\to\Perf_\Bc$ induces a morphism $\phi_*: HC_\bullet(\Ac)\to HC_\bullet(\Bc)$
 which is an isomorphism, if $\phi$ is a quasi-equivalence. 
 \end{itemize}
  \end{prop}
  
  Let us note a more elementary instance of this proposition, cf. \cite{loday} 1.2.4 and 2.2.9. 
  
  \begin{prop}\label{prop:loday-trace}
  Let $A$ be a unital dg-algebra and $r\geq 1$. The collection of the trace maps
  \[
  \begin{gathered}
    \Mat_r(A)^{\otimes (n+1)} \simeq \bigl(\Mat_r(\k)\otimes_\k A\bigr)^{\otimes (n+1)}
    \buildrel\tr_{A,n}\over\lra A^{\otimes (n+1)},
    \\
    \tr_{A,n} \bigl( u_0a_0\otimes\cdots \otimes u_na_n\bigr) \,\,=\,\,
\tr(u_0 \cdots u_n) a_0\otimes\cdots \otimes a_n,   
    \end{gathered} 
    \]
   defines quasi-isomorphisms of complexes
  \[
  C_\bullet^\Hoch(\Mat_r(A)) \buildrel \tr\over\lra C_\bullet^\Hoch (A), \quad
  CC_\bullet(\Mat_r(A))  \buildrel \tr\over\lra  CC_\bullet(A). 
  \]
  The morphisms on $C_\bullet^\Hoch$ and $CC_\bullet$,
   induced by the embedding $A\hookrightarrow\Mat_r(A)$, are quasi-inverse to  $\tr$. \qed
  \end{prop}
  
\paragraph{C. Localization sequence.}

 \begin{Defi}
A localization sequence of perfect dg-categories is a homotopy cofiber sequence
\[
\Ac \lra \Bc \lra \Cc
\]
in the Morita model category of dg-categories (see Appendix), such that the functor $\Ac \to \Bc$ is (quasi-)fully faithful.
\end{Defi}

\begin{rem}
Given such a localization sequence, the homotopy category $[\Ac]$ is a thick subcategory of  the triangulated category
$[\Bc]$ and $[\Cc]$ is equivalent to the
 Verdier quotient $ [\Bc] / [\Ac]$.
\end{rem}

\begin{thm}[Keller, \cite{keller-cyex}]\label{thm:localizationHC}
A localization sequence of perfect dg-categories $\Ac \to \Bc \to \Cc$ induces a cofiber sequence
\[
CC(\Ac) \lra CC(\Bc) \lra CC(\Cc)
\]
in the category of complexes. In particular we get a long exact sequence
\[
\cdots \lra HC_p(\Ac) \lra HC_p(\Bc) \lra HC_p(\Cc) \lra HC_{p-1}(\Ac) \lra \cdots
\]

\end{thm}

\subsection{Cyclic homology of  schemes}\label{subsec:cyc-hom-sch}

\paragraph{A. Sheaf-theoretic definition and Hodge decomposition.} 
We recall the basic constructions and results of Weibel. 

\begin{Defi}[\cite{weibel-sch}]
Let $X$ be a  $\k$-scheme.   Denote by  $\Cc^\Hoch_{\bullet, X}$  
the complex of sheaves on the Zariski topology of $X$
obtained by sheafifying (term by term) the complex of presheaves
$U\mapsto C^\Hoch_\bullet(\Oc(U))$. 
 We define the Hochschild homology of $X$  as the hypercohomology of this complex:
  \[
\HHH(X) = \HH^k(X,\Cc^\mathrm{H}_\bullet(X)). 
\]
We also define the complex of sheaves $\Cc\Cc_{\bullet, X}$ as the
 sheafification of the complex of presheaves 
$U \mapsto CC _{\bullet}(\Oc(U))$, and the cyclic homology of $X$ as 
\[
\HC(X) = \mathbb{H}^k(X,\mathrm{Tot}(\Cc\Cc_{\bullet, X})). 
\]
\end{Defi}
 So $\HHH$ and $\HC$ fit into Mayer-Vietoris sequences by  definition. 
Weibel  proved that for an affine scheme $X=\Spec(A)$,  one has 
\[
\HHH(X)=HH(A), \quad \HC(X) = HC(A).
\]

Let us also recall here that the cyclic homology of a commutative algebra has
  a Hodge decomposition (also called $\lambda$-decomposition, see \cite{loday}):
\[
HH_\bullet(A) = \bigoplus_i  HH^{(i)}_\bullet(A)
\hspace{1cm} \text{and} \hspace{1cm}
HC_\bullet(A) = \bigoplus_i HC^{(i)}_\bullet(A)
\]
Weibel extends this decomposition to the case of schemes and describes it in term of
 HKR-isomorphism in the smooth case.
 
\begin{thm}[\cite{weibel-hod}]\label{thm:hodge-wei}
Let $X$ be a qcqs (quasi-compact and quasi-separated) $\k$-scheme. 
There are Hodge decompositions
\[
\HHH_\bullet(X) = \bigoplus_i \HHH^{(i)}_\bullet(X)
\hspace{1cm} \text{and} \hspace{1cm} 
\HC_\bullet(X) = \bigoplus_i \HC_\bullet^{(i)}(X). 
\]
Moreover, if $X$ is smooth, then we have for any $i$ and $k$ 
\[
\HHH_k^{(i)}(X) = {H}^{i -k}(X,\Omega_X^{i})
\hspace{1cm} \text{and} \hspace{1cm}
\HC_k^{(i)}(X) = \mathbb{H}^{2 i -k}(X,\Omega_X^{\leq i}). 
\]
\end{thm}

\paragraph{B. Relation to HC of dg-algebras and categories.} We first recall the following result of Keller.

\begin{thm}[\cite{keller-sch}]\label{thm:keller}  
For any qcqs scheme $X$ we have
\[
\HHH_\bullet(X) \simeq HH_\bullet(\Perf(X))
\hspace{1cm} \text{and} \hspace{1cm}
\HC_\bullet(X) \simeq HC_\bullet({\Perf}(X)), 
\]
 where $\Perf$ denotes the dg-category of perfect complexes.
\end{thm}

Next, we use this result to relate cyclic homlogy of schemes with that of
dg-algebras of derived functions. 
  
\begin{Defi}
A scheme $U$ is called quasi-affine if it is isomorphic to a  qcqs
 open subscheme of an affine scheme.
\end{Defi}

The following will be used in the setup leading to Theorem \ref{thm:res-invariant}.

\begin{thm}\label{cor:hc-qaffine}
For any quasi-affine scheme $U$, we have
\[
\HC_\bullet(U) \simeq HC_\bullet(R \Gamma(U,\Oc)). 
\]
\end{thm}

\noindent {\sl Proof:} By Theorem \ref{thm:keller} and Proposition \ref{prop:HC-morita}, 
it is enough to show that
$\Perf(X)$ is equivalent to $\Perf_{R \Gamma(U,\Oc)}$. Because perfect complexes
(resp. dg-modules) are intrinsically characterized as compact objects in the
derived category of all quasicoherent sheaves (resp. all dg-modules), we reduce to:

\begin{prop}\label{prop:derivedcat-qaffine}
Let $U$ be a quasi-affine scheme. The derived category of quasi-coherent sheaves on
 $U$ is equivalent to the derived category of dg-modules on the cdga $R \Gamma(U,\Oc)$. 
\end{prop}

Note that this statement is not true for more general schemes (e.g., for $U=\PP^n$). 

\vskip .2cm

\noindent {\sl Proof of the proposition:} 
Let us denote by $A$ any model for the cdga $R \Gamma(U,\Oc)$. Let $u \colon U \to X = \Spec B$ be an open immersion. In particular, we get a map of cdgas $p \colon B \to A$.
It induces an adjunction of $\infty$-categories
\[
p^* \colon \dgMod_B \leftrightarrows \dgMod_A \,:\, p_*
\]
We will prove the functor $p_*$ to be fully faithful and its essential image to coincide with that of $u_*$.
Recall that both $\dgMod_A$ and $\dgMod_B$ are stable $\infty$-categories. In particular, $p_*$ preserves finite colimits because it preserves finite limits.
Moreover, filtered colimits in both $\dgMod_A$ and $\dgMod_B$ are computed on the underlying $\k$-complexes. It follows that $p_*$ preserves all small colimits.

The element $A$ is a compact generator of $\dgMod_A$ and both $p_*$ and $p^*$ preserve small colimits. 
To show that $p_*$ is fully faithful, it suffices to show that the adjunction map $p^* p_* A \to A$ is an equivalence.
This map is equivalent to the multiplication map $A \otimes^L_B A \to A$.

\begin{lem}\label{lem:AABA}
 The multiplication map $A \otimes^L_B A \to A$ is a quasi-isomorphism. 
\end{lem}

\noindent {\sl Proof of the lemma:} 
Let $(f_i)_{i \in I}$ denote a finite family of elements of $B$ such that $U = \bigcup X_{f_i} \subset X$, where $X_{f_i} = \Spec(B[f_i^{-1}])$.  For any  non-empty subset $J \subset I$, 
let  $B_J$ denote the (derived) tensor product
\[
B_J = \bigotimes_{{ B}}^{i \in J}{}^{L} {} \, B[f_i^{-1}].
\]
  By definition, the cdga $A$ is equivalent to the homotopy limit $\holim_{\varnothing \neq J \subset I} B_J$. The derived tensor product $A \otimes^L_B A \to A$ is then equivalent to
\begin{multline*}
A \otimes ^L _B A \,\, \simeq\,\,  \holim_J B_J \otimes^L_B  \holim_{J'} B_{J'}  \,\,\simeq \,\, \holim_J \holim_{J'} (B_J \otimes_A B_{J'}) \,\, \stackrel{\mu}\simeq \\
\simeq \,\,  \holim_J \holim_{J'} B_{J \cup J'} \,\,
\simeq \,\,  \holim_{J} \holim_{J \subset J''} B_{J''}  \,\,\simeq\,\, \holim_{J} B_{J} \,\, \simeq\,\,  A; 
\end{multline*}
where $\mu$ is induced by the multiplication maps.
Lemma \ref{lem:AABA} is proved. 

\vskip .2cm

To continue with Proposition \ref{prop:derivedcat-qaffine}, 
we can now identify $\dgMod_A$ with the full stable and presentable $\infty$-subcategory of 
$\dgMod_B$ generated by $A$. We remark that the derived category of quasi-coherent sheaves on $U$ identifies, through the functor $u_*$, with a full stable and presentable $\infty$-subcategory of $\dgMod_B$ containing $u_* \Oc_U \simeq \RR \Gamma(U,\Oc) \simeq A$.
We hence get an adjunction 
\[
F := p^* u_* \colon \mathrm{D}_\mathrm{qcoh}(U) \rightleftarrows \dgMod_A \,:\, G = u^* p_*
\]
where the functor $G$ is fully faithful. It therefore suffices to prove that $F$ is conservative. Let then $E \in \mathrm{D}_\mathrm{qcoh}(U)$ such that $F(E) = 0$. For any $i$, if we denote by $v_i \colon \Spec(B_i) \to U$ the Zariski embedding, we get
\[
v_i^*(E) = F(E) \otimes_A B_i = 0. 
\]
As  $\{\Spec(B_i)\}_{i\in I}$ is a cover of $U$, we deduce that $E$ is acyclic and hence that the functor $F$ is conservative.
Proposition \ref{prop:derivedcat-qaffine} and therefore Theorem \ref{cor:hc-qaffine}, is proved. 

\subsection{The residue class $\rho\in HC^1(A_n^\bullet)$}\label{subsec:res-class}

\paragraph{A. Definition via a cyclic cocycle.} 
Let $(A_n^\bullet,\dbar) \simeq R\Gamma(\pA^n, \Oc)$ be the commutative dg-algebra
from \eqref{eq:defAn}. We recall that $A_n^\bullet$ is included into the doubly graded differential algebra
$(A_n^{\bullet\bullet}, \del, \dbar)\simeq R\Gamma(\pA^n, \Omega^\bullet)$: we have
$A_n^\bullet = A_n^{0,\bullet}$. 

Consider the  $(n+1)$-linear functional
\be
r: (A_n^\bullet)^{\otimes (n+1)}\lra \k, \quad r(f_0, \dots, f_n) = \Res(f_0\del f_1 \cdots \del f_n).
\ee
Here $\del f_i\in A_n^{1,\bullet}$, and $\Res: A_n^{n-1, n}\to \k$ is the residue
map from \eqref{eq:def-res}. By definition $r(f_0,\dots, f_n)$ is assumed to be equal to $0$
unless $f_0\del f_1\cdots \del f_n$ lies in $A_n^{n, n-1}$. 

\begin{prop}
The functional $r$ is  a degree 1 cocycle in the Connes cochain complex
$C^\bullet_\lambda(A_n^\bullet) = \Hom_\k(C_\bullet^\lambda(A_n^\bullet), \k)$. 
 \end{prop}
 
 \noindent {\sl Proof:} \underbar{\sl Degree 1:} The expression  $f_0\del f_1\cdots \del f_n$
 always lies in $A_n^{n, \bullet}$. For it to lie in $A_n^{n, n-1}$, we  must have
 $\sum \deg(f_i) = n-1$. The horizontal grading of $(A_n^\bullet)^{\otimes(n+1)}$ in the
 Hochschild  chain complex of $A$, is $(-n)$. So $r$, as an element of the dual complex,
 has degree $+1$. 
 
 \vskip .2cm
 
 \noindent \underbar{\sl Cyclic symmetry:} With the Koszul sign rule taken into account,
 the condition for $r$ to lie in $C^\bullet_\lambda(A_n^\bullet)$ is
 \[
 r(f_0, \dots, f_n) \,\,=\,\, (-1)^{n+ \deg(f_0)(\deg(f_1)+\cdots + \deg(f_n))} r(f_1, \dots, f_n, f_0).
 \]
This follows at once from the Leibniz rule for $\del$ and the fact that $\Res$ vanishes on the
image of $\del$.

\vskip .2cm

\noindent \underbar{\sl $r$ is a cocycle:} As $C^\bullet_\lambda(A_n^\bullet)$ is a subcomplex in
$C^\bullet_\Hoch(A_n^\bullet)$, we need to show that $r$ is closed under the 
differential in $C^\bullet_\Hoch$. This differential is the sum $\beta+\dbar^*$, where $\beta$ is
the Hochschild cochain differential, and $\dbar^*$ is induced by the differential   $\dbar$
in $A_n^\bullet$. We claim that both $\beta(r)$ and $\dbar^*(r)$ vanish. These statements follow
from the Leibniz rule for $\del$ and $\dbar$ respectively and the fact that $\Res$
vanishes on the images of both $\del$ and $\dbar$. \qed

\vskip .2cm

We denote by $\rho\in HC^1(A_n^\bullet)$ the class of $r$ and call it the
{\em residue class}. By construction, $\rho$ is $GL_n$-invariant. 

\paragraph{B. Definition via the Hodge decomposition.} 
 By Theorems \ref{cor:hc-qaffine} and \ref{thm:hodge-wei},
 we have the Hodge decomposition
 \be\label{eq:An-Hodge}
  HC_1(A_n^\bullet) \,\,\simeq\,\,  \HC_1(\pA^n) \,\,=\,\,\bigoplus_{i} \HC_1^{(i)}(\pA^n)
  \,\,\simeq \,\, \bigoplus_{i = 1}^n \HH^{2i-1}(\pA^n, \Omega^{\leq i}). 
 \ee
  We consider the projection to the summand
 \be\label{eq:An-dR}
 \HC_1^{(n)}(\pA^n) \,\,=\,\,\HH^{2n-1}(\pA^n, \Omega^{\leq n}) \,\,=\,\,
 \HH^{2n-1}(\pA^n, \Omega^\bullet) \,\,=\,\,\k,
 \ee
 which is the $(2n-1)$-st de Rham cohomology of $\pA^n$ (e.g., for $\k=\CC$, the  $(2n-1)$-st
 topological cohomology of $\CC^n\setminus\{0\} \sim S^{2n-1}$). 
 
\vskip .2cm

The following will be used in Theorem \ref{thm:compar-1}. 
 
 \begin{thm}\label{thm:res-invariant}
 The class in $HC^1(A_n^\bullet)$ given by the projection of \eqref{eq:An-Hodge}
 to \eqref{eq:An-dR}, is the unique, up to scalar, $GL_n$-invariant element in 
  $HC^1(A_n^\bullet)$. In particular, it is proportional 
 to the residue class $\rho$. 
 \end{thm}

\noindent {\sl Proof:} We analyze the $GL_n$-module structure of the $\HH^{2i-1}(\pA^n, \Omega^{\leq i})$.
Applying the spectral sequence starting from $H^p(\pA^n,  \Omega^q)$ and converging to
$\HH^{2i-1}(\pA^n, \Omega^{\leq i})$, we reduce to the following statement.

\begin{prop}\label{prop:Hpq-An}
The vector space $H^p(\pA^n, \Omega^q)$ has no $GL_n$-invariants unless
$p=q=0$ or $p=n-1$, $q=n$, in which case the space of invariants is 1-dimensional. 
\end{prop}

\noindent {\sl Proof:} We have, $GL_n$-equvariantly,
\[
H^p(\pA^n, \Omega^q) \,=\, \Lambda^q(V) \otimes_\k H^p(\pA^n, \Oc),
\]
which vanishes unless $p=0$ or $p=n-1$. 

If $p=0$, then 
\[
H^0(\pA^n,\Oc) = H^0(\AAA^n, \Oc) = \bigoplus_{d\geq 0} S^d(V),
\]
 and
$\Lambda^q(V)\otimes S^d(V)$ has no $GL_n$-invariants unless $q=d=0$.

If $p=n-1$, then Propositions \ref{prop:MB-H} (b) and (c) imply that
\[
\begin{gathered}
H^{n-1}(\pA^n, \Oc) \,=\,\Lambda^n(V^*) \otimes \bigoplus_{d\geq 0} S^d(V^*), \text{  and therefore}
\\
H^{n-1}(\pA^n, \Omega^q)\,=\, \bigoplus_{d\geq 0} \Lambda^q(V)\otimes\Lambda^n(V^*) \otimes S^d(V^*),
\end{gathered}
\] 
which has no invariants unless $q=n$ and $d=0$. \qed

\section{Higher current algebras and their central extensions}\label{sec:hi-curr}

\subsection{Cyclic homology and Lie algebra homology}\label{subsec:cyc-Lie-hom}

\paragraph{A. Reminder on the dg-Lie algebra (co)homology.}
    Let $\len^\bullet$ be a dg-Lie algebra with differential $\dbar$ and bracket
  $c: \Lambda^2\len\to\len$. The {\em Chevalley-Eilenberg cochain complex} of $\len^\bullet$
  is the symmetric algebra
  $\CE^\bullet(\len) = (S^\bullet(\len^\bullet[1]))^*$ equipped with the algebra differential $d_\Lie + D$,
  where:
  \begin{itemize}
  \item $D$ is the algebra differential on the symmetric algebra which extends
  $(\dbar[1])^*$ by the Leibniz rule.
  
 \item  $d_\Lie$ is the algebra differential given on the generators by the map
  $(s \circ c[2])^*: (\len^\bullet [1])^* \to (S^2 (\len^\bullet [1]))^*$
(i.e., extended from the generators by the Leibniz rule). Here we suppressed the
notation for $\dec_2$ from \eqref{eq:dec-n}. 
  \end{itemize}
 
 The cohomology of $\CE^\bullet(\len)$ (equipped with the total grading)
 will be denoted $\HH^\bullet_\Lie(\len)$.
 
In fact, $\CE^\bullet(\len)$ is dual to the {\em Chevalley-Eilenberg chain complex}
$\CE_\bullet (\len) = S^\bullet(\len^\bullet[1])$ whose differential is 
 defined by obvious ``undualizing" of  the formulas above. In other words, it can be
 described in terms of the coalgebra structure on $\CE_\bullet (\len)$.
 We have an increasing exhaustive filtration of $\CE_\bullet(\len)$ by the number of tensor factors
which gives a convergent spectral sequence $H^\Lie_\bullet(H^\bullet_{\dbar} (\len))\Rightarrow H_\bullet^\Lie(\len)$.
This implies that $H^\Lie_\bullet$, and therefore $H^\bullet_\Lie$, are quasi-isomorphism invariant,
in particular, descend to functors on the homotopy category $[\dgLie]$.
 The following is then standard. 
 
 \begin{prop}\label{prop:Lie-central-ext}
  We have a canonical identification
\[
\Hom_{[\dgLie]}(\len, \k[n]) \,\,=\,\, \HH^{n+1}_\Lie(\len). \qed
\]
 \end{prop}
 
 In particular, $\HH^2_\Lie(\len)$ classifies central extensions of $\len$. 
 
 \paragraph{B. Loday's homomorphism $\theta$.}
 
 Let $A$ be an associative dg-algebra (possibly without unit). The graded commutator
\[
[a,b] \,\,=\,\, ab - (-1)^{\deg(a)\cdot \deg(b)} ba
\]
makes $A$ into a dg-Lie algebra which we denote $A_\Lie$. If $A$ is ungraded, then
\[
 H_1^\Lie(A_\Lie) \,\,=\,\,           HC_0(A) \,\,=\,\,          A/[A,A],
\]
Extending \cite{loday} (10.2.3) straightforwardly  from the ungraded case, we include this in the
 following.

\begin{prop}\label{prop:theta}
For any associative dg-algebra $A$ there is a natural morphism of complexes
\[
\begin{gathered}
\theta^A: \CE_{\bullet+1}(A_\Lie) \lra C_\bullet^\lambda (A),
\\
a_0\wedge \cdots \wedge a_n \,\,\mapsto \,\, \sum_{\sigma\in S_{n}} \on{sgn}(\sigma)
[(Id,\sigma)^* (a_0\otimes\cdots\otimes a_n)].
\end{gathered} 
\]
Here, for $\sigma\in S_n$, we denote  by  $(Id,\sigma)\in S_{n+1}$ the permutation of $\{0,1,\dots, n\}$
fixing $0$ and acting on $1,2,\dots, n$ as $\sigma$. The notation $[x]$ means the class of $x$ in the coinvariant
space of $\ZZ/(n+1)$.  In particular, this gives natural maps 
 \[
\theta^A = \theta_i^A: \HH_{i+1}^\Lie (A_\Lie) \to HC_i(A).
\]
 \end{prop} 

\subsection{Higher current algebras and their central extensions}\label{subsec:hi-cur}

 Let $\gen$ be a finite-dimensional reductive Lie algebra over $\k$. We consider the
  dg-Lie algebra
  \[
  \gen_n^\bullet \,\,=\,\, \gen\otimes_\k A^\bullet_n \,\,\simeq \,\, R\Gamma(\pA^n,\gen \otimes_\k\Oc). 
  \]
  We call $\gen_n^\bullet$ the $n$th {\em derived current algebra} associated to $\gen$.
  For $n=1$ we get the Laurent polynomial algebra $\gen_1=\gen[z, z^{-1}]$.

  \vskip .2cm
  
  Let $P\in S^{n+1}(\gen^*)^\gen$ be an invariant polynomial on $\gen$, homogeneous of degree
  $(n+1)$. We will also consider $P$ as a symmetric $(n+1)$-linear form
  $(x_0,\dots, x_n)\mapsto P(x_0, \dots, x_n)$ on $\gen$. 
 
  \begin{thm}\label{thm:gamma-P}
  (a) Consider the functional $\gamma_P: (\gen_n^\bullet[1])^{\otimes (n+1)}\to \k$ given by 
  \[
  \gamma_P \bigl((x_0\otimes f_0) \otimes (x_2\otimes f_1)\otimes \cdots \otimes(x_n\otimes f_n)\bigr)  \,=\,
 P(x_0, \dots, x_n) \cdot  \Res\bigl( f_0 \cdot \del f_1 \wedge \cdots\wedge \del f_n\bigr). 
  \]
  Here $x_\nu\in\gen$, $f_\nu\in A_n^\bullet$, $\del$ is the degree $(1,0)$ differential in 
  $A_n^{\bullet\bullet}\supset A_n^\bullet$,
  and 
   we consider $\Res$ as a functional on the entire $A_n^{\bullet\bullet}$ vanishing on all $A_n^{p,q}$
  with $(p,q)\neq (n, n-1)$. 
  
  This functional is of total degree 2, is symmetric and is annihilated by both differentials $d_\Lie$ and $D$.
  Therefore it is a  cocycle in $\CE^2(\gen_n^\bullet)$ and 
   defines a class $[\gamma_P]\in\HH^2_\Lie(\gen_n^\bullet)$. 
   
   \vskip .2cm
   (b) Assume $\gen$ semisimple. 
   The correspondence $P\mapsto [\gamma_P]$ given an embedding
   $S^{n+1}(\gen^*)^\gen \subset \HH^2_\Lie(\gen^\bullet_n)$. 
   \end{thm}
   
   The cocyles constructed in  theorem \ref{thm:gamma-P}, will be used in the local Riemann-Roch theorem 
   (Corollary \ref{cor:localRR}). 
   
   \begin{ex}
   let $\gen = \gen\len_r$ and $P_{\tr} (x)=\tr(x^{n+1})$ or, in the polarized form,
   \[
   P_{\tr} (x_0,\dots, x_n) \,\,=\,\,  {1\over (n+1)!} \sum_{s\in S_{n+1}} \tr(x_{s(0)}\cdots x_{s(n)}).
   \]
   In this case $\gamma_{P_{\tr}}$ is the image of the residue cocycle $\rho\in C^1_\lambda(A_n^\bullet)$ 
   under the
   composite map
   \[
   C^1_\lambda (A_n^\bullet) \buildrel \tr^*\over\lra C^1_\lambda(\gen\len_r(A_n^\bullet))
   \buildrel \theta_{A_n^\bullet}\over\lra \CE^2 (\gen\len_r(A_n^\bullet)),
   \]
   where $\tr^*$ is dual to the trace map  $\tr_{A_n^\bullet, n}$ from
   Proposition \ref{prop:loday-trace}, and $\theta_{A_n^\bullet}$ is dual to the
   map $\theta^{A_n^\bullet}$ from Proposition \ref{prop:theta}. In particular, we see
   that $\gamma_{P_{\tr}}$ is a Chevalley-Eilenberg cocycle satisfying all the conditions
   of part (a) of  Theorem \ref{thm:gamma-P}. 
    \end{ex} 
    
    \noindent{\sl Proof of Theorem \ref{thm:gamma-P}:} 
    (a) Let $\phi: \gen\to\gen\len_r$ be a representation
    of $\gen$. Take $P_{\phi, n} (x) =\tr(\phi(x)^{n+1})$. In this case $\gamma_{P_{\phi, n}}$ is induced from
    $\gamma_{P_{\tr}}\in \CE^2 (\gen\len_r(A_n^\bullet))$ by $\phi$ and therefore 
    satisfies the conditions of (a).  Further, notice that $\gamma_P$ depends on   $P$
    in a linear way. Now, the statement follows from the next lemma, to be proved further below. 
    
     \begin{lem}\label{lem:repres-1}
     Fix  $m\geq 0$. 
  Any  $P\in S^{m}(\gen^*)^\gen$ is a linear combination of polynomials
     of the form $P_{\phi,m}: x\mapsto \tr(\phi(x)^m)$. 
    \end{lem}

(b) Note    that  $\CE^\bullet(\gen_n^\bullet)$ is the total complex of a bicomplex $\CE^{\bullet\bullet}(\gen_n^\bullet)$
   with
  \[
   \CE^{p q}(\gen_n^\bullet) \,=\, \Hom^p_\k(\Lambda^q \gen_n^\bullet, \k),
  \]
  and with differentials $d_\Lie$ and $D$. We consider the corresponding spectral sequence 
   \be\label{eq:SS}
  E_1^{pq} \,=\, H^{p,q}_D \bigl(  \CE^{\bullet \bullet }(\gen_n^\bullet )\bigr) \,=\, \Hom^p_\k
  \bigl(\Lambda^q H^\bullet_{\dbar}(\gen^\bullet_n), \k\bigr) 
  \, \Rightarrow \, \HH^{p+q}_\Lie(\gen_n^\bullet). 
  \ee
  Since $\gen_n^\bullet$ has $\dbar$-cohomology only in degrees $0$ and $(n-1)$, 
  the spectral sequence is supported in the fourth quadrant,  
  on  the horizontal lines
  \[
  q\geq 0, \,\, p=l(1-n), \, l \geq 0. 
  \]
  Since $\gamma_P$ is annihilated by both differentials, it gives a permanent cycle in $E_1^{-n+1, n+1}$, denote it
  $(\gamma_P)$. For the class $[\gamma_P]$ in $\HH^2_\Lie(\gen_n^\bullet)$ to be zero,
  $(\gamma_P)$ must be killed by some differential of the spectral sequence. From the shape of it,
  the only possible such differential is
  $
  d_n: E_n^{0,1} \to E_n^{-n+1, n+1}
  $.  
  But, denoting $\gen[z]=\gen[z_1, \dots, z_n]$, we have
  \[
  E_2^{0,1} \,=\, H^1_\Lie (H^0_{\dbar}(\gen_n^\bullet)) \,=\, H^1_\Lie (\gen[z])\, = \, \bigl( \gen[z] \bigl/ \bigl[ \gen[z],\gen[z]\bigr]
  \bigr)^*  
  \]
  and this vanishes for a semi-simple $\gen$. Therefore $E_n^{0,1}=0$ and $(\gamma_P)$ cannot be killed. 
  Theorem \ref {thm:gamma-P} is proved. 
  
  \vskip .2cm

       \noindent {\sl Proof of the lemma \ref{lem:repres-1}:} Consider the completion $\wh S^\bullet(\gen^*)$, i.e.,
    the ring of formal power series on $\gen$ near $0$, with its natural adic topology. 
    To any representation $\phi$ of $\gen$ we associate the invariant series
    \[
    Q_\phi(x) \,\, = \,\, \tr(e^{\phi(x)}) \,\,=\,\,\sum_{m=0}^\infty {1\over m!} P_{\phi, m}(x)
    \,\,\in \,\, \wh S^\bullet(\gen^*)^\gen. 
    \]
    By separating the series into homogeneous components, the lemma is equivalent to
    the following statement.
    
\begin{lem}\label{lem:repres-2}
The $\k$-linear space spanned by series $Q_\phi(x)$ for finite dimensional representations $\phi$ of $\gen$, is dense in $\wh S^\bullet(\gen^*)^\gen$. 
\end{lem}
    
    \noindent {\sl Proof of Lemma \ref{lem:repres-2}:} 
     Let $G$ be a  reductive algebraic
    group with Lie algebra $\gen$. 
    The exponential map
    $\exp: \gen\to G$ identifies  $\wh S^\bullet(\gen^*)$
     with $\wh\Oc_{G,1}$, the completed local
    ring of $G$ at $1$, in a way compatible with the adjoint $\gen$-action on both spaces.
    Now, $\k[G]$, the coordinate algebra of $G$, is dense in $\wh\Oc_{G,1}$.  
     Let $\Phi: G\to GL_r$ be an algebraic representation of $G$. Then  the character $\tr(\Phi)$
     is a $\gen$-invariant element in $\k[G]^\gen$, and  such elements span $\k[G]^\gen$.
     Therefore the space spanned by the $\tr(\Phi)$, is dense in     $\wh\Oc_{G,1}^\gen$.
     If now $\phi$ is the representation of $\gen$ tangent to $\Phi$, 
     then the image of $\tr(\Phi)$ under the above identification
     $\exp^*: \wh\Oc_{G,1}^\gen\to \wh S^\bullet(\gen^*)^\gen$, is precisely the series $Q_\phi(x)$.
     So the space of such series is dense in $\wh S^\bullet(\gen^*)$, as
     claimed. \qed
     
     \begin{ex}[(Heisenberg dg-Lie algebra)]
     Let $\gen = \gl_1$ (abelian) and $P(x) = x^{n+1}$. In this case the (2-)cocycle
     \[
     \gamma(f_0, \dots, f_n) \,\,=\,\, \Res\bigl( f_0 \,\del f_1 \cdots \del f_n)
     \]
     defines a central extension $\Hc_n$ of the abelian dg-Lie algebra $\gen_n^\bullet = A_n^\bullet$.
     The dg-Lie algebra $\Hc_n$ is the $n$-dimensional analog of the Heisenberg Lie algebra
     associated to the vector space $\k((z))$ equipped with the skew-symmetric form $[f_0, f_1]=\Res(f_0 df_1)$. 
     \end{ex}
     
\section{The Tate extension and local Riemann-Roch}\label{sec:tate-RR}

\subsection{Background on Tate complexes}\label{subsec:tate-back}

The concept of a Tate space being elementary, in this section we give an equally
elementary treatment of the corresponding derived category. Some of the statementswe list,
are particular cases of more general results, which we indicate. 

\paragraph{A. The quasi-abelian category of  linearly topological spaces.} 
We will use the concept of a {\em quasi-abelian category}  \cite{schneiders},
 a weakening of that of an abelian category.
 In particular, in a quasi-abelian category $\Ac$ any morphism $a: V\to W$ has 
 categorical 
kernel $\Ker(u)$, cokernel, image $\Im(a) = \Ker\{W\to\Coker(a)\}$  
 and the coimage ${\Coim}(a) = \Coker \{\Ker(a)\to V\}$ but  the canonical morphism ${\Coim}(a)\to\Im(a)$
  need not be an isomorphism. If it is, $a$ is called {\em strict}. 
  
   As pointed out in \cite{schneiders},  any quasi-abelian category $\Ac$
  has an intrinsic structure of exact category in the sense of 
   Quillen. 
 
 \begin{ex}
 Let  $\TV_\k$ be  the category of Hausdorff linearly topological $\k$-vector spaces  
with a countable base of neighborhoods of $0$. Here $\k$ is considered with discrete topology.
See \cite{lefschetz}, Ch. 2 for background; in this paper we  additionally impose the countable base assumption.

 For a  morphism $a: V\to W$, 
the kernel $\Ker(a)\subset V$ is the usual kernel with induced topology, $\Im(a)$ is the closure of
the set-theoretic image,  and $\Coim(a)$ is the set-theoretical quotient of $V$ by $\Ker(a)$ with the quotient topology. 

The category $\TV_\k$ is analogous to several  categories treated in the literature   \cite{prosmans}  \cite{pro-schneiders}    \cite{prosmans-2}. 
 In particular, arguments  similar to \cite{prosmans-2} Cor. 3.1.5 and  Prop. 3.1.8,  give:
\end{ex} 

\begin{prop}
(a) The category $\TV_\k$ is quasi-abelian. 

(b) A morphism $a: V\to W$ in $\TV_\k$ is strict, if its set-theoretical image is closed and $a$ is quasi-open.
That is, for any open subspace $U\subset V$ there is an open subspace $U'\subset W$
such that $a(U) \supset a(V)\cap U'$. 
\qed
\end{prop}

\vskip .2cm

Every quasi-abelian category $\Ac$ gives rise to 
the (bounded) derived category $D^b(\Ac)$
equipped with a canonical t-structure whose heart $^\heartsuit\Ac$ (called {\em left heart} in \cite{schneiders}),
is a natural abelian envelope of $\Ac$. It is equipped with a fully faithful left exact embedding $h: \Ac\to {}^\heart\Ac$
  (so $h$ preserves kernels but not cokernels). 
  
  Objects of ${}^\heart\Ac$  can be  thought of as formal  ``true cokernels"  of monomorphisms
$a$ in $\Ac$ and in fact have the form $\Coker_{{}^\heart\Ac}(a)$  (actual cokernels in 
${}^\heart\Ac$).
See \cite{schneiders} Cor. 1.2.21 or, in the  more general framework of exact categories,
\cite{laumon}, Def. 1.5.7. 

\begin{ex}
The formal quotient $\k[[t]]/\k[t]$ represents an object of the abelian category $^\heart\TV_\k$. The 
short exact sequence
\[
0\to \k[t] \lra \k[[t]] \lra \k[[t]]/\k[t]\to 0
\]
represents a nontrivial extension in $^\heart\TV_\k$. 
\end{ex}
  
  One can identify $D^b(\Ac)$ with  the localization 
  \be\label{eq:DbA-loc}
D^b(\Ac) \,\,\simeq \,\,  
  K^b(\Ac)[\on{qis}^{-1}]
  \ee
 of $K^b(\Ac)$, the 
homotopy category  of bounded complexes over $\Ac$,  with respect to quasi-isomorphisms
(understood in the sense of complexes over $^\heart\Ac$), see \cite{laumon} \cite{schneiders}. 
Further, the natural functor $D^b(\Ac)\to D^b(^\heart\Ac)$
into the usual bounded derived category of $\Ac$, is an equivalence (\cite{schneiders} Prop. 1.2.32).

  \begin{rems}
  (a) We note that $K^b(\Ac)$ has a natural dg-enhancement: it is   the $H^0$-category of the  dg-category of bounded complexes over
$\Ac$.  Therefore   \eqref{eq:DbA-loc} can be used to represent $D^b(\Ac)$ as 
 the $H^0$-category of a dg-category, by using  the Gabriel-Zisman localization for dg-categories
(\cite{toen-dg}, \S 2.1) which is the dg-analog of the Dwyer-Kan simplicial localization for categories
\cite{dwyer-kan}\cite{dwyer-kan-simplicial}. 

\vskip .2cm

 (b)  The concept of a quasi-abelian category is self-dual. Therefore there exist another abelian envelope
  $\Ac^\heart  =( ^\heart(\Ac^\op))^\op$ (the right heart) whose objects can be thought as formal ``true kernels" of epimorphisms in $\Ac$,
  and a right exact embedding $\Ac\to {}^\heart\Ac$. In our example $\Ac=\TV_\k$,  the categorical
  kernels coincide with set-theoretical ones so it is natural to use the left heart to keep the kernels unchanged. 
\end{rems}

\begin{prop}\label{prop:relopen}
Let $\Ac$ be quasi-abelian. 
 A bounded complex $V^\bullet$ over $\Ac$ has all $H^i(V^\bullet)\in \Ac\subset {}^\heart\Ac$,
  if and only if
all the differentials are strict. In particular, for a monomorphism $a: V\to W$ in $\Ac$, the cokernel
$\Coker_{^\heart\Ac}(a)$ lies in $\Ac$ if and only if $a$ is strict.  
\end{prop}

\noindent{\sl Proof:}  See \cite{prosmans}, Cor. 1.13. \qed

\paragraph{B. Tate spaces and Tate complexes.} 
 For $V\in\TV_\k$ we have the {\em topological dual} $V^\vee = \ul\Hom_\k (V, \k)$
(continuous linear functionals, with weak topology). The functor $V\mapsto V^\vee$ is
not a perfect duality on $\TV_\k$; however, the canonical morphism $V\to V^{\vee\vee}$
is an isomorphism on the following  full subcategories in $\TV_\k$:
\begin{itemize}
\item[(1)] The  category $\Vect_\k$ of  discrete (at most countably dimensional) vector spaces
 $V\simeq \bigoplus_{i\in I} \k.$

\item[(2)] The category $\LC_\k$ of linearly compact spaces $V\simeq\prod_{i\in I}\k$. 

\item[(3)] The category $\Ta_\k$ of locally linearly compact  spaces (\cite{lefschetz}, Ch. 2, \S 6)
which we will call {\em Tate spaces}. Each Tate space $V$ can be represented as $V\simeq V^d\oplus V^c$
with $V^d\in\Vect_\k$ and $V^c\in\LC_\k$. 
\end{itemize}
Thus, the topological dual identifies
\[
\Vect_k^\op \simeq \LC_\k, \quad \LC_\k^\op \simeq\Vect_\k, \quad \Ta_\k^\op\simeq\Ta_\k. 
\]
In particular, since $\Vect_\k$ is abelian, so is $\LC_\k$, while $\Ta_\k$ is a  self-dual quasi-abelian
(in particular, exact) 
 category, cf. \cite{beilinson-perv} \cite{bgw-tate}. 
 Let us add two more examples  to the above list:
 \begin{itemize}
 
 \item[(4)]  The quasi-abelian category  $\ILC_\k$ formed by  {\em  inductive limits of linearly compact spaces}. 
 
 \item[(5)] The quasi-abelian category $\PVect_\k$ formed by {\em projective limits of discrete vector spaces}
 or, what is the same,  objects in $\TV$ which, considered  as topological vector spaces,
 are {\em complete}. 
 \end{itemize}
   \begin{ex}
 The space of Laurent series $\k((z)) = \k[[z]][z^{-1}]$ is an object of $\Ta_\k$. The localized ring
 $\k[[z_1, z_2]][(z_1z_2)^{-1}]$ is an object of $\ILC_\k$ but not of $\Ta_\k$. 
 Similarly, the ring $\k[[z]][z^*][(zz^*)^{-1}]$, see \S \ref{subsec:exp-mod} C, is an object of $\ILC_\k$. 
 \end{ex}
 
 \begin{prop}
 We have
 \[
 \ILC_k^\op\simeq \PVect_\k, \quad \PVect_\k^\op \simeq \ILC_\k, \quad \Ta_\k = \ILC_\k \cap \PVect_\k, 
 \]
 the first two identifications given by forming the topological dual. 
 \end{prop}

 \noindent{\sl Proof:} A more general statement appears in \cite{BGHW} Prop. 2.1. For convenience of the reader we give here a direct
 proof. 
 
 The first two statements are obvious. 
Let us prove the third statement.
 It is clear that $ \Ta_\k \subset  \ILC_\k \cap \PVect_\k$. Let us prove the inverse inclusion. 
Suppose
 \[
 V \,\, = \,\, \varprojlim \,\,  \bigl\{ \cdots \buildrel q_3\over \lra V_2\buildrel q_2\over \lra V_1\buildrel q_1\over\lra V_0
 \bigr\}\,\,\in \PVect_\k
 \]
 is represented as the projective limit of a diagram $(V_i)$ of discrete vector spaces and surjections
 $q_i$. Then $V$ is Tate, if and only if $\Ker(q_i)$ is finite-dimensional for all but finitely many $i$. 
 Suppose this is not so. Then we can, without loss of generality, assume that all $\Ker(q_i)$
 are inifinite-dimensional, by composing finite strings of the arrows in $(V_i)$ and getting a diagram
 with the same projective limit.

  With this assumption, suppose that $V=\varinjlim_j \, L_j$ where $(L_j)_{j\geq 0}$ is
  an increasing chain of linearly compact subspaces. Then, for each $j$, 
  \[
  L_j  \,\, = \,\, \varprojlim \,\,  \bigl\{ \cdots  \lra V_2^{(j)}  \lra V_1^{(j)} \lra V_0^{(j)}
 \bigr\}
  \]
  where $V_i^{(j)}$ is the image of $L_j$ in $V_i$, a finite-dimensional subspace in $V_i$. 
  We now construct an element $v\in V$, i.e., a compatible system $(v_i\in V_i)$,
  by a version of the Cantor diagonal process. That is, we take $v_0=0$, then
  take $v_1$ from $\Ker(q_1)$  (an infinite-dimensional space)
  not lying in $V_1^{(1)}$ (a finite-dimensional space).  Then take $v_2\in V_2$
  with $q_2(v_2)=v_1$ in such a way that $v_2\notin V_2^{(2)}$ (this is possible
  since $\Ker(q_2)$ is infinite-dimensional), and so on. We get an element $(v_i)$
  of the projective limit with $v_i\notin V_i^{(i)}$ for all $i$. Such an element cannot lie in the
  union of the $L_i$. \qed

  \begin{Defi}\label{def:tate-complex}
(a) A {\em Tate complex} over $\k$ is a bounded  complex  $V^\bullet$ over $\ILC_\k$ whose cohomology 
groups $H^i(V^\bullet)\in {}^\heart\ILC_\k$  belong to $^\heart\Ta_\k$. 
We denote by $\Tate_\k$ the full  dg-subcategory in $C^b(\ILC_\k)$ formed by Tate complexes. 

(b) A Tate complex $V^\bullet$ is called {\em strict}, if  all the $H^i(V^\bullet)$ belong to $\Ta_\k\subset {}^\heart\Ta_\k$
or, what is the same, if its differentials are strict (Prop. \ref{prop:relopen}). 
\end{Defi}

Note that $\Tate_\k$ is a perfect dg-subcategory, i.e., $[\Tate_\k]$, the corresponding
cohomology category, is triangulated and closed under direct summands, see Appendix A.C. 

\begin{ex}\label{ex: A_n-tate-complex} 
The Jouanolou complex $A^\bullet_n$ is a strict Tate complex. More precisely,  the topology  on each $A_n^p$
 is given by the convergence of series. An explicit representation of $A_n^p$ as an inductive limit of linearly
  compact spaces is given by the filtration
of Corollary \ref {cor:filt}. Thus $A_n^\bullet$ is a complex over $\ILC_\k$. As we have seen, 
  its cohomology groups are Tate vector spaces. For $n>0$, we have only two cohomology spaces:
   $H^0$, linearly compact and  $H^{n-1}$, discrete. 

More generally, for any finite dimensional vector bundle $E$ on $D^\circ_n$, the complex $A^\bullet_n(E)$ is naturally made into a strict Tate complex.
\end{ex}
 
 \paragraph{C. Tate complexes, algebraically.} 
  For a category $\Cc$ we denote by $\Ind(\Cc)$
 and $\Pro(\Cc)$ the category of countable ind- and pro-objects in $\Cc$,  see \cite{artin-mazur},
 \cite{KS-categories} for general background. In particular, we will use the notation
 $\ilim_{i\in I} C_i$ for an object of $\Ind(\Cc)$ represented by a filtering inductive system $(C_i)_{i\in I}$
 over $\Cc$. Similarly for $\plim_{i\in I} C_i$, an object of $\Pro(\Cc)$. 
   
 Assume $\Cc$ is abelian.  Then so are $\Ind(\Cc)$
 and $\Pro(\Cc)$. In this case  we denote by $\Ind^s(\Cc), \Pro^s(\Cc)$  the full subcategories
 formed by  ind- and pro-objects which are  {\em essentially strict} 
  i.e.,  isomorphic to objects $\ilim_{i\in I} C_i$, resp.  $\plim_{i\in I} C_i$
  where $(C_i)$ is an
 filtering inductive (resp. projective)
 system formed by monomorphisms (resp. epimorphisms). 
 These are  quasi-abelian but not,  in general,  abelian
 categories.

  Let $\Vect_\k^f$ be the category of finite-dimensional $\k$-vector spaces. 
 
 \begin{prop}\label{prop:ind-fin}
(a) We have
  \[
 \Ind^s(\Vect_\k^f)= \Ind(\Vect_\k^f)\simeq \Vect_\k, \quad \Pro^s(\Vect_\k^f)=\Pro(\Vect_\k^f)\simeq \LC_\k.
 \]  
 
 (b) Further, the taking of inductive or projective limits in $\TV_\k$ gives identifications
 \[
 \Ind^s(\LC_\k) \simeq \ILC_\k, \quad \Pro^s(\Vect_\k) \simeq \PVect_\k. 
 \]

 (c) The abelian envelopes of the semi-abelian categories in (b) are identified as
 \[
^\heart \ILC_\k  \simeq \Ind(\LC_\k), \quad ^\heart\PVect_\k  \simeq \Pro(\Vect_\k)
 \]
(the categories of all, not necessarily strict, ind- and pro-objects). 
 \end{prop}
 
 \noindent {\sl Proof:} (a) This is well known. For instance,  because of the Noetherian and Artinian property of
 $\Vect_\k^f$, and ind- or pro-object in this category is essentially strict.   
 Part (b) is also
 clear. 

(c) Let us prove the first statement,  the second one is dual.
 Consider the abelian category $\Ac$ of arbitrary chains of morphisms (inductive systems)
 \be\label{eq:k[t]}
V=\bigl\{  V_0\lra V_1\lra V_2\lra \cdots\bigr\}
 \ee
 in $\LC_\k$, i.e.,  of graded $\k[t]$-modules in $\LC_\k$. Let
 $\Ac^s\subset \Ac$ be  the semi-abelian subcategory
  formed
 by chains of monomorphisms, i.e., torsion-free $\k[t]$-mod-ules.
  Every object $V$ of $\Ac$ can be represented as the cokernel of
 a monomorphism $i$ in $\Ac^s$. More precisely, we have a short exact sequence
 \be\label{eq:2-term-res}
 0\to K\buildrel i\over\lra  V\otimes_\k \k[t] \buildrel c\over \lra V\to 0.
 \ee
 Here $V\otimes_\k \k[t]$ is the free $\k[t]$-module generated by $V$ as a graded
 vector space, 
$c$ is the canonical map given by the $\k[t]$-module structure and $K=\Ker(c)$. 
 This implies that any object in $\Ind(\LC_\k)$ is the cokernel of a monomorphism
 in $\Ind^s(\LC_\k)$. 
  \qed
 
 \begin{prop}\label{prop:homdim1}
(a)  The objects of $\ILC_k\subset {} ^\heart\ILC_k  = \Ind(\LC_k)$ are projective. 
In particular, Tate spaces are projective 
objects in the abelian category $^\heart\Ta_\k$. 

(b)  The abelian categories $\Ind(\LC_\k)$, $\Pro(\Vect_\k)$ have homological dimension 1
(i.e.,  the $\Ext^{\geq 2}$ in these categories vanish). 
 \end{prop}
 
  \noindent {\sl Proof:}  We start with three lemmas.
  As in the proof of Proposition \ref {prop:ind-fin}(c), let $\Ac$ be the category
  of diagrams as in \eqref{eq:k[t]}. We have a functor
  \[
  \ilim \, : \Ac\lra \Ind(\LC_\k), \quad V\mapsto \ilim V_i. 
  \]
 
 \begin{lem}\label{lem: lim-exact}
 $\ilim\,\,$ is an exact, essentially surjective functor. Furthermore,  any epimorphism in $\Ind(\LC_\k)$
 is isomorphic (in the category of arrows) to an image of an epimorphism in $\Ac$. 
  \end{lem}

 \noindent{\sl Proof of the lemma:} Exactness of $\ilim$ and the 
statement about epimorphisms are
   general properties of ind-categories
 of abelian categories, see \cite{KS-categories}, Lemmas 8.6.4 and  8.6.7. Next,  any countable filtering category
  admits a cofinal map from
 the poset $\ZZ_+ = \{0,1,2,\dots\}$. This means that any countably indexed ind-object is isomorphic
 to the object of the image of $\ilim$ as above. 
    \qed
 
 \begin{lem}\label{lem:projective}
 Objects of $\Ac^s$ are projective in $\Ac$. 
  \end{lem}
  
  \noindent {\sl Proof of the lemma:} A diagram $V$ as in \eqref{eq:k[t]} formed by
  embeddings, is a free 
  graded $\k[t]$-module in $\LC_\k$. Indeed, since $\LC_\k=\Vect_\k^\op$
  is a semisimple abelian category, the embedding $V_{i-1}\to V_{i}$ admits a direct
  sum complement $W_i$. Considering $W=(W_i)$ as a $\ZZ_+$-graded object in
  $\LC_\k$, we see that $V\simeq W\otimes_\k \k[t]$ is free. This, and
  semisimplicity of $\LC_\k$,  implies projectivity of $V$. 
  
  \qed
  
  \vskip .2cm
  
  Let us call an object $V\in\Ac$ essentially strict, if $\ilim (V)$ is an essentially strict
  object of $\Ind(\LC_\k)$, i.e., is isomorphic to $\ilim (M)$ where $M\in\Ac^s$. 
  
  \begin{lem}\label{lem:essentially}
  If $V\in\Ac$ is essentially strict, then there is $M\in\Ac^s$ and an epimorphism $q: V\to M$
  in $\Ac$ such that $\ilim (q)$ is an isomorphism in $\Ind(\LC_\k)$. 
  \end{lem}
  
  \noindent {\sl Proof of the lemma:} For each $i$ consider the diagram of epimorphisms
  \[
  V_i \lra \Im\{V_i\to V_{i+1}\} \lra \Im\{V_i\to V_{i+2}\}\lra \cdots. 
  \]
   If $V$ is essentially strict, these epimorphisms eventually become isomorphisms,
   so the terms of the diagram stabilize to some $M_i\in\LC_\k$, and we get the diagram
   of monomorphisms
   \[
   M\,\,=\,\,\{ M_0 \lra M_1\lra \cdots \}.
   \]
 We see that $M\in\Ac^s$ equipped with a natural epimorphism $q: V\to M$ in $\Ac$
 which induces an isomorphism on the $\ilim$. 
 \qed
 
 \vskip .2cm
 
 We now prove part (a) of Proposition \ref{prop:homdim1}. Let $f: A\to B$ be an epimorphism
 in $\Ind(\LC_\k)$ with $B$ essentially strict. We prove that $f$ splits.
 By Lemma \ref{lem: lim-exact}, $f$ is isomorphic to $f' = \ilim (g)$ where $g: N\to V$
 is a surjection in $\Ac$. It is enough to prove that $f'$ splits. Now, $V\in\Ac$ is essentially
 strict, so by Lemma \ref{lem:essentially}, there is a surjection $q: V\to M$ in $\Ac$
 with $M\in\Ac^s$ and such that $\ilim(q)$ is an isomorphism. It is enough
 therefore to prove that the epimorphism $f'' = \ilim(q g)$ splits in $\Ind(\LC_\k)$.
 But $qg: N\to M$ is a surjection in $\Ac$ with $M$ projective. So $qg$ splits in $\Ac$
 and therefore $f''$ splits in $\Ind(\LC_\k)$. This proves part (a).

 Part (b), follows by considering the 2-term resolution \eqref{eq:2-term-res}. 
 Proposition \ref{prop:homdim1} is proved.

 \begin{cor}\label{cor:tate-triang-explicit}
 (a) Every object of $D^b(\ILC_\k)\simeq D^b(^\heart\ILC_\k)$ is quasi-isomorphic to
  its graded object of cohomology 
  equipped with zero differential. 
  
  (b) The triangulated category $[\Tate_\k]$ is equivalent to the $D^b(^\heart\Ta_\k)$.

  (c) Any strict Tate complex $V^\bullet$ can be split (in $\ILC_\k$) into a direct sum of complexes $V^\bullet = H^\bullet\oplus E^\bullet$, where $E^\bullet$ is exact and $H^\bullet$ has zero 
  differential (and is thus a graded Tate space isomorphic to $H^\bullet(V^\bullet)$.)
 
 \end{cor}
 
 \noindent{\sl Proof:} Part (a) follows from Proposition \ref{prop:homdim1}(b) and the following well known fact
 that can be proved by induction on the length of the complex.

 \begin{lem}
 Let $\Bc$ be an abelian category of homological dimension $1$. Any object $V^\bullet \in D^b(\Bc)$ is
 quasi-isomorpic to $H^\bullet(V^\bullet)$ with zero differential. \qed
 \end{lem}
 
 Let us prove part (b) of  Corollary \ref{cor:tate-triang-explicit}. For a  strictly full abelian subcategory $\Bc$ in
 an abelian categpry $\Cc$ let $D^b_\Bc(\Cc)\subset D^b(\Cc)$ be the full subcategory formed by complexes
 over $\Cc$ with cohomology in $\Bc$. We use the following.
 
 \begin{lem}\label{lem:D-B-C}
 If both $\Bc$ and $\Cc$ have enough projective, then the natural functor $D^b(\Bc)\to D^b_\Bc(\Cc)$
 is an equivalence. \qed
 \end{lem}

 To continue with the proof of part (b) of  Corollary \ref{cor:tate-triang-explicit}, we notice that Proposition
  \ref{prop:homdim1}(a) means that $[\Tate_\k]$ is the full category in the {\em homotopy category}
  of complexes over $^\heart\ILC_\k$ formed by complexes of   strict (hence projective)
  objects with cohomology lying in $^\heart\Ta_\k$. For two complexes of projective objects, their Hom in the homotopy category is the
  same as their Hom in the derived category. Because of this and of the $2$-term projective resolution
  of every object of $\ILC_\k$, we conclude that $[\Tate_\k]$ is identified with $D^b_{^\heart\Ta_\k} (\ILC_\k)$. 
  Now, $^\heart\Ta_\k$ has enough projectives (the actual Tate spaces), so our statement follows from
  Lemma \ref{lem:D-B-C}. 
  
  \vskip .2cm
  
  Finally, let us prove part (c) of Corollary   \ref{cor:tate-triang-explicit}. A strict Tate complex is a complex of
  projective objects in $^\heart\ILC_\k$ whose cohomology objects are projective in $^\heart\Ta_\k$ and
  in $^\heart\ILC_\k$. Therefore it has a splitting as claimed. Corollary  \ref{cor:tate-triang-explicit} is proved. 
  
  \vskip .2cm

 Inside $\Tate_\k$ we have the full dg-subcategories:
$\Perf_\k$,   complexes with finite-dimensional cohomology; 
$D_\k$, complexes with
discrete cohomology; 
$C_\k$, complexes with linearly compact cohomology. Thus $\Perf_\k=C_\k\cap D_\k$
and the associated cohomology categories are identified as
\[
[\Perf_\k]\simeq D^b(\Vect_\k^f), \quad [D_\k]\simeq D^b(\Vect_k), \quad 
[C_k] \simeq D^b(\LC_\k). 
\]

\begin{prop}\label{prop:tate-exp-smallest}
$[\Tate_\k]$is the smallest strictly full triangulated subcategory in $D^b(\ILC_\k)$ which
contains $[D_\k], [C_\k]$ and is closed under direct summands.
\end{prop}

 \noindent {\sl Proof:} Let $\Tc$ be the smallest subcategory in question. Then 
 $[\Tate_\k] \subset\Tc$, because any object of the quasi-abelian category
 $\Ta_\k$ is a direct sum of an object
 from $C_\k$ and an object from $D_\k$. Conversely, 
 $^\heart\Ta_\k \subset {}^\heart\ILC_\k$
 is closed under extensions. Therefore forming cones and direct summands,
 starting with $\Ob(C_\k)\cup\Ob(D_\k)$ will always give complexes whose
 cohomology objects lie in $^\heart\Ta_\k$. \qed
 
\subsection{Tate objects in dg-categories} \label{subsec:tate-ob-dgcat}

\paragraph{A. Ind- and pro-objects in dg-categories.} 
For background on ind- and pro-objects in  a (stable) $\infty$-category $\Cc$ we refer to \cite{lurie-htt}, \S 5.3. 
  In this paper we consider only ind- and pro-objects represented by countable filtering diagrams.
  Such objects form new $\infty$-categories $\Ind(\Cc)$, $\Pro(\Cc)$. 
  Thus,   objects of  $\Ind(\Cc)$ can be represented by symbols $\hoind_{\hskip -.2cm i\in I}\,\,  x_i$ where $I$ is a countable
  filtering $\infty$-category and $(x_i)_{i\in I} $ is an $\infty$-functor $I\to \Cc$. Similarly, objects of $\Pro(\Cc)$
  can be represented by symbol $\hopro_{\hskip -.2cm i\in I} \,\, y_i$, where $I$ is as before and $(y_i)_{i\in I}$
  is an $\infty$-functor $I^\op\to\Cc$. 
  
  \vskip .2cm 
  
We apply these concepts to dg-categories by converting them to $\k$-linear $\infty$-categories by 
using the dg-nerve construction \cite{cohn-dgcat}\cite{faonte}. The resulting (countable) ind- and
 pro-categories  associated to a dg-category $\Ac$ will be denoted $\Ind(\Ac)$ and $\Pro(\Ac)$. 
They are still $\k$-linear $\infty$-categories that we can and will see as dg-categories. 
We note that, with this understanding, the dg-categories
 $\Ind(\Ac)$ and $\Pro(\Ac)$ are perfect, whenever $\Ac$ is perfect (see Appendic A.C for the meaning of
 ``perfect''). 
Let us point out the following  more explicit description.

\begin{prop}\label{prop:ind-explicit}
Let $\Ac$ be a dg-category. Then: 
 \begin{itemize}
\item[(I0)]  $\Ind(\Ac)$ is quasi-equivalent to its full dg-subcategory
whose  objects are  $\hoind    x_i$, where
\[
(x_i) \,\,=\,\,\bigl\{ x_0 \buildrel f_{01}\over\lra x_1 \buildrel  f_{12}\over\lra x_2 \buildrel f_{23}\over\lra \cdots\bigr\}
\]
is a diagram consisting of objects $x_i\in\Ac$, $i\geq 0$ and closed degree $0$ morphisms $f_{i, i+1}: x_i\to x_{i+1}$.

\item[(I1)]  We have quasi-isomorphisms  (``$\epsilon$-$\delta$ formula"): 
\[
\Hom^\bullet_{\Ind(\Ac)}  \bigl( \hoind x_i, \hoind y_j\bigr) \,\,\simeq \,\, \holim_i \,\, \hocolim _j\Hom^\bullet_\Ac(x_i, y_j),
\]
where the homotopy limits on the right are taken in the model category $\dgVect$ of complexes. 
\end{itemize} 
\end{prop}

\begin{rems}
(a) The  $\hocolim_j$ in (I1) is the same as the naive inductive 
 limit in $\dgVect$, see Proposition \ref{prop:dgvect-holim}(a).

  \vskip .2cm

(b) By duality we get a description of $\Pro(\Ac) =\Ind(\Ac^\op)^\op$ in terms of symbols $\hopro  x_i$
where $(x_i)$ is a diagram of objects and closed degree $0$ morphisms in the order  opposite to  that of (I0). 

\end{rems}
  
  \noindent{\sl Proof of Proposition \ref{prop:ind-explicit}:} 
By definition, objects of $\Ind(\Ac)$ are represented by $\infty$-functors $I\to\Ac$
where $I$ is a filtering $\infty$-category. In our setting we assume that $\Ob(I)$
is at most countable. As in the classical case, any filtering $I$ admits a cofinal $\infty$-functor
from a filtering poset (\cite{lurie-htt}, Prop. 5.3.1.16), and in the countable case
we can take this poset to be $\ZZ_+$. Now, the category corresponding to the
poset $\ZZ_+$, is freely generated by the elementary arrows $i\to i+1$.
Therefore any $\infty$-functor $x: \ZZ_+\to\Ac$ can be replaced by an honest
functor obtained by extending $x$ from these elementary arrows. Such functors
are precisely the data in (I0).  Finally, the appearance of the ``$\epsilon$-$\delta$ formula''
in (I1) from conceptual properties of $\Ind(\Ac)$ is explained in  
\cite{lurie-htt}, p.378. 
\qed

  \paragraph{B. Tate $A$-modules.}
 
 Let $A$ be a $\ZZ_{\leq 0}$-graded commutative dg-algebra. For any $m\in\ZZ$
 we then have the full dg-category
 $\Perf^{\leq m}_A\subset\Perf_A$ formed by those perfect dg-modules over $A$ which, as
 complexes, are situated in degrees $\leq m$. Recall that the duality functor
 $M\mapsto M^\vee = \RHom_A(M, A)$ identified $\Perf_A$ with its dual. We define
 the full subcategories in $\Perf_A$: 
 \[
 \Perf^{\geq l}_A \,=\, (\Perf^{\leq -l}_A)^\vee, \quad \Perf^{[l,m]}_A \,=\,\Perf^{\geq l}_A
 \cap\Perf^{\leq m}_A, \quad l\leq m. 
 \]
 We define then the  $\infty$ (or dg-) categories
 \[
 \Db_A =\bigcup_{l,m} \Ind(\Perf^{[l.m]}_A) \subset \Ind(\Perf_A),\quad 
 \Cb_A   = \bigcup_{l,m}\Pro(\Perf^{[l,m]}_A) \subset 
 \Pro(\Perf_A)
  \]
  formed by ind- or pro-diagrams of which all terms belong to $\Perf_A^{[l,m]}$
 for some $l,m$ (depending on the diagram).  These categories are dual to
 each other.
 
 The category $\Db_A$
 is tensored over the category $\Vect_\k$ of all (possibly infinite-dimensional)
 $\k$-vector spaces. In particular, with each object $M$ it contains $\k[z]\otimes M= \bigoplus_{i=0}^\infty M$,
 the direct sum of infinitely many copies of $M$. The category $\Cb_A$
 is tensored over the category of linearly compact topological $\k$-vector spaces.
 In particular, with each object $M$ it contains $\k[[z]]\wh \otimes M  = \prod_{i=0}^\infty M$,
 the direct product of infinitely many copies of $M$.

\vskip .3cm

 We define the dg-category $\Tate_A$  of {\em Tate $A$-modules}
 as the perfect envelope of the full dg-subcategory in $\Ind(\Pro(\Perf_A))$
 whose class of objects is $\Ob(\Cb_A) \cup \Ob(\Db_A)$. Since
 $\Ind(\Pro(\Perf_A))$ is a perfect dg-category, we can view $\Tate_A$
 as the minimal dg-subcategory in $\Ind(\Pro(\Perf_A))$
 containing $\Cb_A, \Db_A$ and closed under forming shifts,  cones and homotopy direct summands.
    
 \begin{rem}
 This is a slight modification of the definition in \cite{hennion-tate} in that, besides restricting to
 countable ind- or pro-diagrams,  we force all objects to be
``bounded complexes". In particular, an infinite resolution of a non-perfect $A$-module $M$
 (and thus such an $M$  itself)
 would be an object in $\Ind(\Perf_A)$ but not in $\Db_A$. In particular, it is, a priori,
 not an object of $\Tate_A$. 
 \end{rem}
 
\begin{prop}\label{prop:tate-comparison}
Let $A=\k$. Then, in comparison with the constructions from \S \ref{subsec:tate-back}:

(a) $\Db_\k$ is quasi-equivalent to $D_\k 
\simeq C^b(\Vect_\k)$  and $\Cb_\k$ to $C_\k \simeq C^b(\LC_\k)$;

(b)  $\Tate_A$ (for $A = \k$) is quasi-equivalent to the category of Tate complexes from \S \ref{subsec:tate-back}. 
\end{prop}

\noindent {\sl Proof:}  (a)
Let us show the first identification, the second one follows
by duality.  It is known that $\Ind(\Perf_\k)=\dgVect$. (Both homotopy limits in the RHS of (I1) can be replaced by the ordinary projective
resp. inductive limits because of Proposition \ref{prop:dgvect-holim}(b1).) Then $D_\k$ and $C^b(\Vect_\k)$ are
full subcategories of $\dgVect$ and have the same objects.

\vskip .2cm

We now prove part (b). Let us denote temporarily by $\Tate_{A=\k}$ the dg-category
obtained by specializing the above definition of $\Tate_A$ to $A=\k$, as distinguished
from the dg-category $\Tate_\k$ defined in \S \ref{subsec:tate-back}. Let
\[
IP_\k \,\,=\,\,\bigcup_{l,m} \Ind(\Pro(\Perf_\k^{[l,m]}))\,\,\subset \,\,\Ind(\Cb_\k). 
\]
Then $\Tate_{A=\k}$ is the perfect envelope of 
$\Cb\Db_\k$, the full
dg-subcategory in $IP_\k$ on the class of objects $\Ob(\Cb_\k)\cup\Ob(\Db_\k)$.
Similarly, since $C^b(\ILC_\k)$ is a perfect dg-category, Proposition 
\ref{prop:tate-exp-smallest} implies that $\Tate_\k$ is identified with the
perfect envelope of $CD_\k$, the full subcategory in $C^b(\ILC_\k)$
on the class of objects $\Ob(C_\k)\cup\Ob(D_\k)$. Let us denote both identifications
$\Db_\k\to D_\k$ and $\Cb_\k\to C_\k$ from part (a) by the same letter $\lambda$
(taking the limit). Then it is enough to prove the following,

\begin{lem}
Let $V= \hoind \, V_i^\bullet$ be an object of $\Db_\k$ and $W=\hopro \, W_i^\bullet$
be an object of $\Cb_\k$, with $V_i^\bullet$, resp. $W_i^\bullet$ being
an inductive resp. projective system  over $\Perf^{[l,m]}_\k$. Then the natural 
morphisms of complexes
\[
\begin{gathered}
\Hom^\bullet_{IP_\k}(V, W) \lra \Hom^\bullet_{C^b(\ILC_\k)} (\lambda(V), \lambda(W)), 
\\
\Hom^\bullet_{IP_\k}(W, V) \lra \Hom^\bullet_{C^b(\ILC_\k)} (\lambda(W), \lambda(V))
\end{gathered}
\]
are quasi-isomorphisms. 
\end{lem}

\noindent{\sl Proof of the lemma:} We can assume that $(V_i^\bullet)$ consists of injective
morphisms and $(W_i^\bullet)$ consists of surjective morphisms of perfect complexes. 
Then, applying  the formula (I1) from Proposition \ref{prop:ind-explicit} twice, we realize
$\Hom^\bullet_{IP_\k}(V,W)$ as double $\holim$ of a diagram of perfect complexes and surjective
maps. Applying Proposition \ref{prop:dgvect-holim}(b) once, we replace the first
$\holim$ by $\varprojlim$ and get a $\holim$ of a diagram of complexes and surjective
maps. Applying  Proposition \ref{prop:dgvect-holim}(b) once again, we replace the
second $\holim$ by $\varprojlim$. After this the result reduces to the set of
continuous morphisms of complexes $\lambda(V)\to\lambda(W)$. 

Similarly, we realize $\Hom^\bullet_{IP_\k}(W,V)$ as double $\hocolim$ of a diagram of perfect complexes and  injective maps. Applying Proposition \ref{prop:dgvect-holim}(a),
we reduce it to double $\varinjlim$ which again gives the space of continuous
morphisms of complexes $\lambda(W)\to\lambda(V)$.
This proves the lemma and Proposition \ref{prop:tate-comparison}. 

\qed

 \vskip .2cm

 \begin{prop}\label{prop:RGammak}
The construction $E \mapsto A_n^\bullet(E) = \Gamma(\wh J,\Omega^\bullet_{\wh J/D^\circ_n} \otimes \wh \pi^* E)$ defines an exact functor 
\[
R\Gamma \colon \Perf_{D^\circ_n} \lra \Tate_k.
\]
\end{prop}

This statement,  as well as the stronger Proposition \ref{prop:buildRGamma} later, can be seen as analogs, in our setting, of
Theorem 7.2 of \cite{drinfeld}.

\begin{proof}
The functor $\Gamma(\wh J,\Omega^\bullet_{J/D^\circ_n} \otimes -)$ is naturally made into a (strict) functor
\[
\on{Mod}_B \lra \Ind \Pro \Perf_k
\]
where $B$ is the ring of function of the affine scheme $\wh J$.
Deriving this functor, we get an exact dg-functor
\[
\on{D}_{\mathrm{qcoh}}(\wh J) \simeq \dgMod_B \lra \Ind \Pro \Perf_\k
\]
which restricts along $\wh \pi^*$ to the announced functor
\[
R\Gamma \colon \Perf_{D^\circ_n} \lra \Tate_\k.
\]
\end{proof}

\subsection{The Tate class, the residue class and the local Riemann-Roch}\label{subsec:LRR}

\paragraph{A. The Tate class in cyclic cohomology.}
We start with a delooping  result.

\begin{thm}\label{thm:Hc-dloop}
There is a canonical isomorphism
\[
HC_{*}(\Tate_\k) \simeq HC_{*-1}(\Perf_\k) \simeq HC_{*-1}(\k).
\]
\end{thm}

\begin{proof}
This is essentially a corollary of the main result of Saito \cite{saito} except  in our definition of $\Tate_\k$
we ``first derive'' and then topologize while Saito deals with exact categories such as $\Ta_\k$ and with
K-theory rather than cyclic homology. For convenience of the reader we give a direct prove which uses
the same approach as \cite{saito}.

Let us consider the morphism of localization sequences (of  perfect dg-categories):
\[
\xymatrix{
\Perf_\k\ar[d]  \ar[r] & \Cb_\k \ar[d] \ar[r]& \Cb_\k/\Perf_\k \ar[d]^\alpha
\\
\Db_\k \ar[r]& \Tate_\k \ar[r]& \Tate_\k/\Db_\k
}
\]
It follows from \cite[Proposition 4.2]{hennion-tate} that the functor $\alpha$ is an equivalence.
We will deduce the result from the following lemma.

\begin{lem}\label{lem:vanishHC}
If $\Ac$ is a  perfect dg-category with infinite direct sums, then the cyclic complex $CC(\Ac)$ is acyclic. In particular $HC_*(\Ac) \simeq 0$.
\end{lem}
Let us postpone the proof of the lemma for now. The categories $\Cb_\k$ and $\Db_\k$ both admit infinite sums, and therefore have vanishing cyclic homology.
Using the localization invariance of $HC$ (see Theorem \ref{thm:localizationHC}), we get quasi-isomorphisms
\[
CC(\Tate_\k) \buildrel \sim \over \lra  CC(\Tate_\k/\Db_\k) \buildrel \sim \over \lla CC(\Cb_\k/\Perf_\k) \buildrel \sim \over \lra CC(\Perf_\k)[1].
\]
This concludes the proof of Theorem \ref{thm:Hc-dloop}.
\end{proof}

\noindent{\sl Proof (of Lemma \ref{lem:vanishHC}):} We use the following fact (see  \cite{keller-cyex}): if $f,g$ are two functors from $\Ac$ to $\Bc$ and $\Bc$ is perfect (in particular, has direct sums), then the action of $f \oplus g$ on $HC$ is the sum of the action of $f$ and of $g$.
Specifying to the case $f=\bigoplus_{i=0}^\infty \Id$ and $g=\Id$, then $f\oplus g \simeq f$ and we get $HC(g)=HC(\Id)=0$. 
\qed

\vskip .2cm

\begin{Defi}\label{def:tateclassHC}
Let us denote by $\tau$ the class $\tau \in HC^1(\Tate_\k)$ given by the image of the trace class through the isomorphism $HC^0(\k) \simeq HC^1(\Tate_\k)$.\\
For any object $V\in \Tate_\k$ we have a class $\tau_V\in HC^1 (\End(V))$ induced by $\tau$.
\end{Defi}

\begin{rem}
The class $\tau_V$ vanishes (by construction) as soon as $V$ belongs to either $\Cb_\k$ or $\Db_\k$.
\end{rem}

 \paragraph{B.  Comparison with the residue class and local Riemann-Roch. }
The  Tate complex $A_n^\bullet$ is, as the same time, a  commutative dg-algebra. We note that (left) action of $A_n^\bullet$
on itself gives rise to  a  morphism of associative dg-algebras which we call the {\em regular representation}
\[
l: A_n^\bullet\lra \End_{\Tate_\k}(A_n^\bullet). 
\]
 In other words,  for each $p$ and each $a\in A_n^p$, the  multiplication operator  $l(a): A^\bullet_n\to A^\bullet_n[p] $,  
 is a continuous morphism of  graded objects of 
$\ILC_\k$. This follows from the fact that the filtration on $A_n^\bullet$ from  Corollary \ref {cor:filt} which,
by Example \ref{ex: A_n-tate-complex}, gives the inductive limit representation of $A_n^\bullet$,
is compatible with multiplication. 

\begin{thm}\label{thm:compar-1}
There exists a non-zero constant $\lambda\in\k$ with the following property.
  The pullback $l^*\tau_{\End(A_n^\bullet)}$
  of the Tate class $\tau_{\End(A_n^\bullet)}$ with respect to  $l$
  is a class in $HC^1(A_n^\bullet)$
equal to  $\lambda\cdot \rho$, where $\rho$ is the residue class, see \S \ref{subsec:res-class}.  
\end{thm}

This theorem is close to the main result of Braunling \cite{braunling} Thm. 2.6 which deals with $n$-fold
local fields (and $n$-fold Tate objects) . For $n=1$, the theorem reduces to the cited result which also implies
that $\lambda=1$ in that case. For $n>1$, Theorem \ref{thm:compar-1} highlights the essence of our
derived adelic approach: we get higher-dimensional residues from usual ($1$-fold) Tate objects
but work in the derived category. 

\vskip .2cm 

\noindent{\sl Proof of Theorem \ref{thm:compar-1}:} For any associative dg-algebra $R^\bullet$ let us denote by $BR^\bullet$
the corresponding dg-category with one object. 
The global section functor 
\[
R\Gamma \colon \Perf_{A_n^\bullet} \simeq \Perf_{D_n^\circ} \lra   \Tate_\k
\]
 from Proposition \ref{prop:RGammak} is compatible with the map $l$.  In other words, we have a commutative diagram of dg-categories
 \[
\xymatrix{
BA^\bullet _n \ar[r] \ar[d]^l & \Perf_{A_n^\bullet} \ar[d]^{R\Gamma} \\
B(\End(A_n^\bullet)) \ar[r] & \Tate_\k
}
\]
where the horizontal functors map the one object onto $A_n^\bullet$.
In particular we have 
\[
l^*\tau_{\End(A_n^\bullet)} = R\Gamma^* \tau \,\,\in \,\,  HC^1(\Perf(A_n^\bullet)) \simeq HC^1(A_n^\bullet).
\]

Recall that the dg-category $\Perf_{A_n^\bullet} \simeq \Perf_{D_n^\circ}$ is the dg-quotient of $\Perf_{D_n}$ by the full subcategory $\Perf_{D_n,\{0\}}$ spanned by perfect complexes supported at $0$ (see \cite[\S5]{TT}). The global section functor $\Gamma_{D_n} \colon \Perf_{D_n} \to \Cb_\k$ hence induces a functor $G$ and a commutative diagram
\[
\xymatrix{
\Perf_{\{0\}}(D) \ar[r] \ar[d]^{\Gamma_{\{0\}}} & \Perf(D) \ar[r]^-P \ar[d]^{\Gamma_{D_n}} & \Perf(D^\circ) \ar[d]^{G} \\
\Perf_k \ar[r] \ar[d] & \Cb_\k \ar[r] \ar[d]^i & \Cb_\k/\Perf_\k \ar[d]^\alpha_\simeq \\
\Db_\k \ar[r] & \Tate_k \ar[r]^-Q & \Tate_k/\Db_\k.
}
\]
\begin{lem}
The composite functor $Q \circ R\Gamma$ is equivalent to the composite $\alpha \circ G$.
\end{lem}
\noindent{\sl Proof:} 
From the universal property of $\Perf(D_n^\circ)$ as a quotient of $\Perf(D_n)$, it suffices to compare $Q \circ R\Gamma \circ P$ and $\alpha \circ G \circ P \simeq Q \circ i \circ \Gamma_{D_n}$. 
The inclusion $H^0(A_n^\bullet) \to A_n^\bullet$ induces a canonical natural transformation $\alpha \circ \Gamma_{D_n} \to R\Gamma \circ P$. In turn, it induces the required equivalence (as the pointwise kernel of the natural transformation is killed by the projection $Q$). \qed

\vskip 2mm

It follows from this lemma that $R\Gamma^* \tau$ equals the composite
\[
HC_1(\Perf(D_n^\circ)) \overset{\delta}{\lra} HC_0(\Perf_{D_n,\{0\}}) \overset{\Gamma_{\{0\}}}{\lra} HC_0(\Perf_\k).
\]
Let us write $z_\bullet$ for the coordinate system $z_1,\dots, z_n$ on $D_n$ and denote, 
as in \S \ref{subsec:repan},  by  $V = \bigoplus \k  z_i$
the space spanned by the $z_i$.  In what follows we will pay attention to the $GL_n$-action on various spaces arising.

\vskip .2cm

The dg-category $\Perf_{D_n,\{0\}}$ is generated by $\k$ 
(considered as a trivial $\k[[z_\bullet]]$-module), and
we have a $GL_n$-equivariant identification
\be\label{eq:s-lambda}
R\End_{\k[[z_\bullet]]}(\k)\,\,  \simeq \,\,  \Ext^\bullet_{\k[[z_\bullet]]} (\k, \k) \,\,\simeq \,\, S^\bullet(V^*[-1])
\ee
with the exterior algebra of $V$ graded by its degree. 
Using \cite{schwede-shipley}, we get a Morita equivalence between $\Perf_{D_n,\{0\}}$ and $R\End_{\k[[z_\bullet]]}(\k)$
 (a version of the classical $S-\Lambda$ duality \cite{gelfand-manin}). 
  Under this Morita equivalence, the functor $\Gamma_{\{0\}}$ amounts to the augmentation morphism $S^\bullet(V^*[-1]) \to \k$. The induced map $HC_0(\Perf_{D_n,\{0\}}) \to HC_0(\Perf_\k)$ is thus non-trivial.
It is also $GL_n$-invariant.

\begin{lem}
$HC_0(\Perf_{D_n,\{0\}})$ admits a unique $GL_n$-coinvariant class.
\end{lem}
\noindent{\sl Proof:} 
Consider the Hodge decomposition of $HC_\bullet(S^\bullet V^*)$.
The part of weight $p$ is computed by the complex $S^\bullet (V^*[-1]) \otimes \Lambda^{\leq p} (V^*[-1]) [2p]$ with the de Rham differential. The de Rham differential is $GL_n$-equivariant. Therefore, to compute the $GL_n$-coinvariant elements in the cyclic homology at hand, it suffices to understand the action on each degree of the graded vector space
\[
S^\bullet (V^*[-1]) \otimes \Lambda^{\leq p} (V^*[-1]) [2p] \simeq \bigoplus_{j=0}^n \bigoplus_{i=0}^{p} \Lambda^j V^* \otimes S^i V^* [2p - (2i - j)]
\]
The $GL_n$-representation $\Lambda^j V^* \otimes S^i V^*$ has simple spectrum and admits coinvariants if and only if $i = j = 0$. We get
\[
HC_m^{(p)}(S^\bullet (V^*[-1]))_{GL_n} = \begin{cases} \k & \text{if } m = 2p \\ 0 & \text{else.} \end{cases}
\]
In particular $HC_0(S^\bullet(V^*[-1]))$ has exactly one invariant $1$-dimensional subspace.

\qed

 \vskip .2cm 

We now finish the proof of Theorem \ref{thm:compar-1}. The residue $HC_1(A_n) \to HC_0(\k)$ is $GL_n$-invariant and vanishes when restricted to $HC_1(\k[[z_\bullet]])$. 
By Theorem \ref{thm:res-invariant}, $HC_1(A_n^\bullet)$ containes a unique $GL_n$-invariant line. 
The long exact sequence
\[
\cdots \lra HC_1(\k[[z_\bullet]]) \lra HC_1(A_n) \overset{\delta}{\lra} HC_0(S^\bullet(V^*[-1])) \lra \cdots
\]
implies that $\delta$ maps isomorphically the unique invariant line of $HC_1(A_n)$ onto the unique invariant line of $HC_1(S^\bullet(V^*[-1]))$. This concludes the proof of Theorem \ref{thm:compar-1}.

\begin{rem}
Since $\k$ is allowed to be an arbitrary field of characteristic $0$,
we have $\lambda\in\QQ^*$. We expect $\lambda=\pm 1$. This can
be possibly proved either by direct calculation or by upgrading some of
the considerations of this paper
to fields of arbitrary characteristic. 
\end{rem}

Recall (Proposition \ref{prop:theta}) that the class $\tau_V\in HC^1(\End(V))$,
$V\in \Tate_\k$ gives a Lie algebra cohomology class $\theta^*\tau_V$
in $\HH^2_\Lie(\End(V))$. We will also call $\theta^*\tau_V$ the Tate class. 
Let us consider the following particular case.

\vskip .2cm

Let $r\geq 1$ and $V= (A_n^\bullet)^{\oplus r}$. As before, we have then the regular representation
(morphism of dg-algebras) $l_r: \Mat_r(A_n^\bullet) \to \End_{\Tate_\k}((A_n^\bullet)^{\oplus r})$
which we can also see as a morphism of dg-Lie algebras.
 
\begin{cor}[Local Riemann-Roch]\label{cor:localRR}
The pullback $l_r^*(\theta^*\tau_{(A_n^\bullet)^{\oplus r}})$ is
equal to the class of the cocycle 
$\lambda \cdot \gamma_{P_r}$, where $\lambda$ is the same as in Theorem 
\ref{thm:compar-1} and  $\gamma_{P_r}$ is the special case, for
$P(x) = P_r(x) = \tr(x^{n+1})/(n+1)!$, of the cocycle defined in Theorem 
\ref{thm:gamma-P}. \qed
\end{cor}
 
\begin{rems}
(a) We note that $P_r(x)$ is  the degree $(n+1)$ component of 
$\tr(e^x)$, the ``Chern character" of $x$.  Corollary \ref{cor:localRR} can be therefore seen
as a simplified version of a local Riemann-Roch-type  theorem where we take into account
the infinitesimal symmetries of a vector bundle but not of the underlying manifold. 

\vskip .2cm

(b) Combined  with Theorem \ref{thm:action-det} below, this result determined explicitly the nature of the higher
Kac-Moody algebra acting on the determinantal bundles. 

\end{rems}

\section{Action on derived moduli spaces}\label{sec:act-der}

  In the rest of the paper we will relate the dg-Lie algebra ${\gen}_n^\bullet$
  and its central extensions,
    with derived moduli spaces of principal bundles on $n$-dimensional manifolds.
    We start with recalling the general setup.

\subsection{Background on derived geometry}\label{subsec:back-der}

\paragraph{A. General conventions.}
   We will work in the framework of derived algebraic geometry. For general results on the subject, we refer to \cite{toen-vezzosi}. For a comprehensive survey, the reader may  look at \cite{toen-ems}.
    
    Derived algebraic geometry can be though as algebraic geometry, where rings are being replaced by 
    ``homological rings''.  Namely, the category of $\k$-algebras will be replaced the category $\cdga$ formed by $\ZZ_{\leq 0}$-graded commutative dg-algebras over $\k$ (up to quasi-isomorphisms).
    It is naturally made into a model category. Moreover, the usual notions of Zariski open or closed immersions, flat, smooth or \'etale morphisms extend to morphisms in $\cdga$. In particular, one can form an \'etale Grothendieck topology.
    
    Given any commutative algebra, one can consider it as an object in $\cdga$ concentrated in degree $0$. On the other hand, for any object $A \in \cdga$, the cohomology space $H^0 A$ is a commutative algebra. We also get a canonical morphism $A\to H^0A$.
    
    Let also $\sSet$ be the category of simplicial sets. Given two objects $A, B \in \cdga$, we get a simplicial set of morphisms $\mathrm{Map}(A,B)$. In particular, we get a Yoneda functor mapping $A$ to a covariant functor $\Spec A \colon \cdga \to \sSet$.
    
    A derived prestack is a covariant functor $\cdga \to \sSet$. A derived stack is a derived prestack
     satisfying the natural homotopy \'etale descent condition.
    A derived stack representable by a cdga is called a derived affine scheme. We will denote by $\dAff$ the category of derived affine schemes. It is equivalent to the opposite category of $\cdga$.

\paragraph{B. Derived stacks and derived categories.}
The category of derived stacks will be denoted by $\St$. It will be considered either as a model category, or as an $\infty$-category. 
Given $A \in \cdga$, the category of $A$-dg-modules is endowed with a standard model structure. We denote by $\dgMod_A$ the ($\k$-linear) dg-category of fibrant-cofibrant $A$-dg-modules. 
Given a derived stack $Y$, we define, following \cite{toen-dgazumaya}
\[
\mathrm{D}_\mathrm{qcoh}(Y) = \holim_{\Spec A \to Y} \dgMod_A
\]
(homotopy limit in the model category of dg-categories). 
An object in $\mathrm{D}_\mathrm{qcoh}(Y)$ can be informally described as the data of a $A$-dg-module
$M_{A, \phi}$ for any $A \in \cdga$ and any map $\phi:  \Spec A \to Y$, together with natural homotopy glueing data.
We also define the $\ZZ_{\leq 0}$-graded derived category
$\mathrm{D}_\mathrm{qcoh}^{\leq 0}(Y) \subset \mathrm{D}_\mathrm{qcoh}(Y)$ 
formed by $(M_{A, \phi})$ consisting of $\ZZ_{\leq 0}$-graded dg-modules. 

\begin{ex}
Any $M\in\mathrm{D}_\mathrm{qcoh}^{\leq 0}(Y)$ gives rise to the {\em dual number stack} $Y[M]$ defined 
by gluing $\Spec (A \oplus M_{A, \phi})$ (the trivial square zero extension). 
\end{ex}

Note that for any map $f \colon Y \to Z$, one gets an adjunction
\[
Lf^* \colon \mathrm{D}_\mathrm{qcoh}(Z) \rightleftarrows \mathrm{D}_\mathrm{qcoh}(Y) \,:\, Rf_*
\]
For example, for $\phi:\Spec(A)\to Y$ and $M\in \mathrm{D}_\mathrm{qcoh}(Y)$,
the object $L\phi^* M$ is just the structure datum $M_{A,\phi}$ (we identify 
$\mathrm{D}_\mathrm{qcoh}(\Spec A)$ with $\dgMod_A$).

Another example: if $Z = \Spec \k$ 
and $f$ is the canonical projection, then  $Rf_* \colon \mathrm{D}_\mathrm{qcoh}(Y) \to \dgMod_k$
computes the cohomology of $Y$ with values in a given object in $\mathrm{D}_\mathrm{qcoh}(Y)$.

\vskip 2mm

We will also need not necessarily quasi-coherent sheaves. The functor $A \mapsto \dgMod_A$ lands in dg-categories, hence in ($\k$-linear) $\infty$-categories. Let $\xi \colon \int \dgMod \to \dAff$ denote the associated Cartesian Grothendieck construction (see \cite[Chap. 3]{lurie-htt}).
\begin{Defi}
Let $Y$ be a derived stack. We define its derived category of $\Oc_Y$-complexes $\on{D}(Y)$ as the $\infty$-category of sections $\dAff/Y \to \int \dgMod$ of $\xi$ over $Y$.
\end{Defi}
Note that by \cite[3.3.3.2]{lurie-htt}, the $\infty$-category $\on{D}_{\mathrm{qcoh}}(Y)$ is the full subcategory of $\on{D}(Y)$ spanned by Cartesian sections.
Informally, an object in $M \in \on{D}(Y)$ is the data of $A$-dg-modules $M_{A,\phi}$ for any map $\phi \colon \Spec A \to Y$, together with coherence maps $\zeta_{f} \colon M_{A,\phi} \otimes_A^L B \to M_{B,\phi \circ f}$ for any map $f \colon \Spec B \to \Spec A$ and higher coherence data.
The module $M$ is then quasi-coherent if and only if all the maps $\zeta_f$ are quasi-isomorphisms.

The category $\on{D}(Y)$ admits internal homs that we will denote by $R\Hom_{\Oc_Y}$.

\paragraph{C. Geometric objects and tangent complexes. }
For the definition of  geometric derived stacks (or, what is the same, derived Artin stacks)
 we refer to \cite{toen-vezzosi}. 
 
This class includes, first  all derived schemes, that is, 
 derived stacks that are Zariski locally equivalent to derived affine schemes. Following \cite{lurie-dagv},
  one can represent derived schemes in terms of "homotopically" ring spaces. 
   Namely, a derived scheme $X$ is a topological space together with a sheaf (up to homotopy) of 
   $\ZZ_{\leq 0}$-graded cdga's $\Oc_X$ such that $(X,H^0(\Oc_X))$ is a scheme.
   
   In fact,   a derived Artin stack is a derived stack that can be obtained from derived affine schemes by 
   a finite number of smooth quotients.  
   
   \vskip .2cm
   
  The {\em cotangent complex} $\LL_Y$ of a derived stack $Y$ is
  an object of $\mathrm{D}_\mathrm{qcoh}(Y)$
   defined (when it exists) by the universal
  property   
  \[
  \Map_{\mathrm{D}_\mathrm{qcoh}(Y)} (\LL_Y, M) \,\,\simeq \,\,
   \Map_{{}^{Y/}\St}
   (Y[M], Y), \quad M\in  \mathrm{D}_\mathrm{qcoh}^{\leq 0}(Y). 
  \]
  Here  ${}^{Y/} \St$ is the comma category of derived stacks under $Y$. 
  The  object $\LL_Y$ is known to exist \cite{toen-vezzosi} when $Y$ is geometric
  (no smoothness assumption). 
  
The {\em tangent complex} $\TT_Y$ is defined as the dual
\[
\TT_Y \,\,=\,\, R\Hom_{\Oc_Y}(\LL_Y, \Oc_Y) \in \on{D}(Y).
\]
If $Y$ is locally of finite presentation \cite{toen-vezzosi}, then $\LL_Y$ is a perfect complex and hence so is $\TT_Y$. In particular $\TT_Y$ is an object of $\mathrm{D}_\mathrm{qcoh}(Y)$. 
  
For a $\k$-point $i_y: y\hookrightarrow  Y$ we will write 
\[
\TT_{Y,y} = Li_y^*( \TT_Y) =  R\Hom_{\Oc_Y}(\LL_Y, \Oc_y)
\]
  for the tangent complex of $Y$ at $y$. This is a complex of $\k$-vector spaces.  
  
  \paragraph{D. Derived intersection:}
 Given a diagram $X\to Z\leftarrow Y$, we have  the {\em derived (or homotopy)
  fiber product} $X\times_Z^h Y$. 
  If $X,Y,Z$ are affine, so our diagram is represented by a diagram $A\leftarrow C\to B$ in $\cdga$, then
\[
X\times_Z^h Y \,\,=\,\,\Spec\bigl( A\otimes_C^L B\bigr). 
\]
We will be particularly interested in the following situation. 
Let $f: X \to Y$ be a morphism of derived stacks, and $y \in Y$ be a $\k$-point. 
Then we have the derived stack (a derived (affine) scheme, if both $X$ and $Y$ are 
derived (affine) schemes) 
\[
Rf^{-1}(y) \,\,=\,\, X\times^h_Y \{y\}. 
\]
It will be called the  {\em derived preimage of } $y$. It is the analog of the homotopy fiber of a map
between spaces in topology.

\subsection{The Kodaira-Spencer homomorphism}\label{subsec:kod-sp}

\paragraph{A. Group objects and actions.} 
By a {\em group stack}
 we will mean
a stack  $G$ together with simplicial stack $G_\bullet$ such that $G_0\simeq \Spec \k$,
$G_1\simeq G$ and which satisfies the {\em Kan condition}: the morphisms corresponding to
the inclusions of horns are equivalences. 
Intuitively, $G_\bullet$ is the nerve of the group structure on $G$, 
see   \cite[ \S 4.2.2]{lurie-halg} for more details. 

\vskip .2cm

Similarly, an {\em action} of a group stack $G$ (given by $G_\bullet$) on a stack $Y$
is a simplicial stack $Y_\bullet$ together with
a morphism $q: Y_\bullet\to G_\bullet$   with an identification $Y_0\simeq Y$
such that, for any $m$, the morphism
\[
(q_m, \del_{\{m\}\hookrightarrow \{0,1,\dots, m\}}): Y_m\lra G_m\times Y_0
\]
is an equivalence.
In this case $Y_\bullet$ satisfies the Kan condition. Intuitively, $Y_\bullet$
is the nerve of the ``action groupoid".   The ``realization" of $Y_\bullet$, i.e.,
the derived stack associated to the prestack $A\mapsto |Y_\bullet(A)|$,
is the quotient derived stack $[Y/G]$. In particular, we have the stack $BG= [*/G]$,
the {\em classifying stack} of $G$. 

\begin{exas}
(a) Let $Y$ be a derived stack and $y\in Y$ be a $\k$-point. 
The {\em pointed loop stack} 
\[
\Omega_y Y = \{y\} \times^\mathrm{h}_Y \{y\}: A \mapsto \Omega(Y(A), y)
\]
 is a group stack. The corresponding simplicial stack 
 $(\ul\Omega_yY)_\bullet$ is  the (homotopy)
 nerve of the morphism $\{y\}  \to Y$, i.e., 
 \[
(\ul \Omega_yY)_m \,\,=\,\,  \{y\} \times^\mathrm{h}_Y \{y\}  \times^\mathrm{h}_Y\cdots  \times^\mathrm{h}_Y
\{y\} 
\,\,\simeq \,\, (\Omega_yY)^m
\]
($(m+1)$-fold product). 

\vskip .2cm

(b) Let $Y$ be any derived stack. Its {\em automorphism stack} is the group stack 
\[
\RAut(Y): A\mapsto 
\Map^\mathrm{eq}_{\St/\Spec A}(Y \times \Spec A, Y \times \Spec A)
\]
 Here the superscript ``eq" means the union of connected components of the mapping space
 formed by vertices which are equivalences. Alternatively, we can describe it
 as the functor
 \[
 A\mapsto \Omega(\St/\Spec A, Y\times \Spec A),
 \]
 the based loop space of the nerve of the category of derived stacks over $\Spec A$ 
 with the base point being the object $Y\times \Spec A$. 
 By construction, we have an action of $\RAut(Y)$ on $Y$; an action of a group
 stack $G$ on $Y$ gives a morphism of group stacks $G\to\RAut(Y)$.

 \end{exas}

\begin{prop}\label{prop:actionofloops}
Let $f \colon X \to Y$ be a map of derived stacks and $y \in Y$ be a $\k$-point.
Then the group stack $\Omega_y Y$ has a natural action on the derived preimage
$Rf^{-1}(y)$. 
 \end{prop}
 
 \noindent{\sl Proof:} We define the simplicial stack $\underline Rf^{-1}(y)_\bullet$  
 as the nerve of the morphism $Rf^{-1}(y)\to X$, i.e.,
 \[
\begin{gathered}
\underline Rf^{-1}(y)_m \,\,=\,\ Rf^{-1}(y)  \times^\mathrm{h}_X Rf^{-1}(y)
\times^\mathrm{h}_X\cdots \times^\mathrm{h}_X Rf^{-1}(y) \,\,\simeq \,\, 
\\
\simeq \,\, 
\{y\}  \times^\mathrm{h}_Y \{y\}  \times^\mathrm{h}_Y\cdots  \times^\mathrm{h}_Y X \,\,\simeq\,\,
(\Omega_yY)^m\times Rf^{-1}(y). 
\end{gathered} 
\]
All the required data and properties come from contemplating the commutative diagram, obtained as the nerve of the RHS horizontal morphisms
\[
\xymatrix{
\cdots ~\ar[r] \ar@<3pt>[r] \ar@<-3pt>[r] &
\underline Rf^{-1}(y)_1 \ar@<2pt>[r]\ar@<-2pt>[r] \ar[d] & Rf^{-1}(y) \ar[r] \ar[d] & X \ar[d] \\ 
\cdots ~\ar[r] \ar@<3pt>[r] \ar@<-3pt>[r] &
\Omega_y Y \ar@<2pt>[r]\ar@<-2pt>[r] &\relax \{y\} \ar[r] & Y. 
} 
\]
 \qed
 
 \begin{ex}[(Eilenberg-MacLane stacks)]
 Let $\Pi$ be a commutative algebraic group (in our applications $\Pi=\GG_m$). For each $r\geq 0$ we then
 have group stack $\EM(\Pi,n)$, known as the {\em Eilenberg-MacLane stack}. It is defined in the standard
 way using the Eilenberg-MacLane spaces for abelian groups $\Pi(A)$ for commutative $\k$-algebras $A$.

 Thus $\EM(\Pi,0) = \Pi$ as a group stack, i.e., the corresponding simplicial stack $\EM(\Pi,0)_\bullet = \Pi_\bullet$
 is the simplicial classifying space of $\Pi$. Similarly (the underlying stack of the group stack) $\EM(\Pi, 1)$ 
 is identified with  $B\Pi$. 
 In general, if we denote $\EM(\Pi, n)_\bullet$ the simplicial
 stack describing the group structure on $\EM(\Pi, n)$, then $|\EM(\Pi,n)|= \EM(\Pi, n+1)$. 
 \end{ex}
 
 \begin{Defi}
 Let $G$ be a group stack and $\Pi$ a commutative algebraic group. A {\em central extension} of $G$ by $\Pi$
 is a morphism of group stacks $\phi: G\to B\Pi$ or, what is the same, a morphism of stacks $BG\to\EM(\Pi, 2)$. 
 \end{Defi} 
 
 A central extension $\phi$  gives, in a standard way, a fiber and cofiber sequence of group stacks
 \[
 1\to \Pi\lra\wt G\lra G\to 1, 
 \]
 where $\wt G$ is the fiber of $\phi$. 
 
 \paragraph{B. Formal moduli problems. } We recall Lurie's work
 \cite{lurie-dagx} on formal moduli problems which serve as infinitesimal analogs of derived stacks. 

\begin{Defi}
A cdga  $A\in  \cdga$ is called  Artinian, if:
\begin{itemize}
\item The cohomology of $A$ is finite dimensional (over $\k$);
\item The ring $H^0 A$ is local and the unit induces an isomorphism between $\k$ and the residue field of $H^0 A$.
\end{itemize}
\end{Defi}

 In particular, any Artinian cdga admits a canonical augmentation (the unique point of
 $\Spec A$). 
 Artinian cdga's form an $\infty$-category which we will denote by $\mathbf{dgArt}_\k$. 

\begin{Defi}\label{def:fmp}
A formal moduli problem is a functor (of $\infty$-categories)
\[
F \colon \mathbf{dgArt}_\k \to \sSet
\]
such that:
\begin{itemize}
\item[(1)] 
 $F(\k) \simeq *$ is contractible. 
 
 \item[(2)]  (Schlessinger condition): 
For any diagram $A \to B \leftarrow C$ in $\mathbf{dgArt}_\k$ with both  maps  surjective on $H^0$, the canonical map $F(A \times^\mathrm{h}_B C) \to F(A) \times^\mathrm{h}_{F(B)} F(C)$ is an equivalence.

\end{itemize}
\end{Defi}

\noindent We denote by $\Fun_*(\dgart, \sSet)$ the ($\infty$-)category of functors from $\dgart$ to simplicial sets satisfying the
condition (1) of Definition \ref{def:fmp}, and by $\FMP$ the full subcategory of formal moduli problems. 
General criteria for representability of functors imply (see \cite{lurie-dagx}, 1.1.17) that
we have a left adjoint  (the ``formal moduli envelope")
\be\label{eq:FM-envelope}
\Lc: \Fun_*(\dgart, \sSet) \lra \FMP
\ee
to the embedding functor. 

\vskip .2cm

For a formal moduli problem $F$ we define its {\em  tangent complex} $\mathbb T_{F}$ (at the only point 
$*$ of $F$).
This is a complex of $\k$-vector spaces defined as follows. First, we define
\[
\Tc^{(p)}_F  = F(\k[\epsilon_p]/\epsilon_p^2), \quad \deg(\epsilon_p)=-p, \quad p\geq 0. 
\]
 These simplicial sets are actually  simplicial vector spaces, forming a spectrum in the sense of
 homotopy theory, that is, connected by morphisms (in fact, by equivalences)
  $\Tc^{(p)}_F\to\Omega \Tc^{(p+1)}_F$. We define the complex $\TT_F$
  to correspond to the spectrum $(\Tc_F^{(p)})$  by the Dold-Kan equivalence.

\begin{thm}[Lurie]\label{thm:formalmod}
For any formal moduli problem $F$, the complex  $\mathbb T_{F}[-1]$ has a homotopy Lie structure.
Moreover, the data of $F$ is equivalent (up to homotopy) to the data of a complex
$\TT = \TT_F[-1]$ and of a dg-Lie  structure on  this complex.
 \end{thm}
 
 For future reference we recall   from  \cite{lurie-dagx} \S 2,  the characterization of the Lie algebra
 structure on $\TT_F[-1]$. For a graded $\k$-vector space $M^\bullet$ let $\on{FL}(M^\bullet)$
 be the free graded Lie algebra (with zero differential) generated by $M^\bullet$. 
 Denote by $\on{FLie}^{\geq 1}_\k$ be the full ($\infty$-) subcategory in
 $\dgLie$ whose objects are $\on{FL}(M^\bullet)$ where $M^\bullet$ is finite-dimensional
 and situated in degrees $\geq 1$. For any $L\in \on{FLie}^{\geq 1}_\k$
 the Chevalley-Eilenberg cochain algebra $\CE^\bullet(L)$ is
 an object of $\dgart$. In fact, $\CE^\bullet(L)$ is quasi-isomorphic to
 the dual number algebra $\k \oplus (M^\bullet)^*[-1]$, see
   \cite{lurie-dagx} (2.2.15). 
We then have 
 \be
 \Map_{\dgLie}(L, \TT_F[-1]) \,\,=\,\, F(\CE^\bullet(L)), \quad L\in \on{FLie}^{\geq 1}_\k
 \ee
 (identification of functors on $\on{FLie}^{\geq 1}_\k$).
 Note that this defines the dg-Lie algebra structure uniquely.

 \vskip .2cm

We will also need the following global analog of the Schlessinger condition,
see \cite[Def. 1.4.2.1]{toen-vezzosi}.
  
\begin{Defi}\label{def:infcart}
A derived stack $Y$ is {\em infinitesimally cartesian}
 if the following condition holds. Let 
   $A\in\cdga, M\in \dgMod^{\leq -1}_A$ and 
   \[
   M\lra B\lra A \buildrel\delta\over\lra M[1]
   \]
   be a square zero extension of $A$ by $M$. Then the square
   \[
   \xymatrix{
   F(B)
   \ar[d] \ar[r] & F(A)
   \ar[d]^{F(1+\delta)}
   \\
   F(A) \ar[r]_-{F(1)}& F(A\oplus M[1])
   }
   \]
   is cartesian. 
\end{Defi}

It is known that any derived Artin stack is infinitesimally cartesian \cite[Prop. 1.4.2.5]{toen-vezzosi}. 

\begin{ex}\label{ex-tangentislie}
Any infinitesimally cartesian stack $Y$ and any $\k$-point $y\in Y$ defines a formal moduli problem
\[
\wh{Y}_y \colon A \mapsto Y(A) \times^\mathrm{h}_{Y(\k)} \{y\}. 
\]
Its tangent complex $\TT_{\wh Y_y}$ is identified with  $\TT_{Y,y} = \Hom(\LL_Y, \k_y)$,
if $\LL_Y$ exists (e.g., if $Y$ is geometric). In other cases it can be considered
as the definition of $\TT_{Y,y}$.
By  Theorem \ref{thm:formalmod},
the shifted complex $\mathbb T_{Y,y}[-1]$   carries a homotopy Lie structure. 
It is an analog of the fundamental group $\pi_1(Y,y)$
of a topological space $Y$ at a point $y$. 
\end{ex}

For any derived stacks $Y$ and $X$ we define the mapping stack
\[
\RMap(Y, X): A\mapsto 
\Map_{\St/\Spec A}(Y \times \Spec A, X \times \Spec A). 
\]
Notice that $\RAut(Y)$ is an open substack in $\RMap(Y,Y)$, that is, the pullback of it under any morphism
$U\to \RMap(Y,Y)$ from an affine derived scheme $U$, is an open subscheme in $U$.

\vskip .2cm
  
  \begin{prop}\label{prop:tangent-aut}
(a) If $Y$ is an infinitesimally cartesian derived stack, then so is $\RMap(X,Y)$ for any $X$.
Moreover, suppose  $Y$ is geometric and $f: X\to Y$ is any morphism,
representing a $\k$-point $[f]$ of $\RMap(X,Y)$. Then 
 the tangent complex $\TT_{\RMap(X,Y), [f]}$ 
(defined as the tangent complex of the associated
formal moduli problem)
  is identified with $R\Gamma(X, f^*\TT_Y)$.

(b) Further,  if $Y$ is infinitesimally cartesian, then so is $\RAut(Y)$. In this case
  $\TT_{\RAut(X), \Id}$ 
  is  $R\Gamma(X, \TT_X)$. 
\end{prop}

\begin{proof} (a) 
 Consider the full subcategory $\Cc$ of $\St$ spanned by those $X \in \St$ such that  $\RMap(X,Y)$ is 
 infinitesimally cartesian.
Since $Y$ is infinitesimally cartesian, $\Cc$ contains all derived affine schemes. 
The category $\Cc$ is moreover stable by (homotopy) colimits. We get $\Cc = \St$.
To identify the tangent complex, we consider the universal properties of both sides. Let us fix a perfect complex $E$ over $\k$. We have
\begin{align*}
 \Map_\k(&E, R\Gamma(X, f^*\TT_Y)) \simeq \Map_{\Oc_X}(E \otimes_\k \Oc_X, f^* \TT_Y) \\
 &\simeq \Map_{X/\St}(X \times \Spec(\k \oplus E^\vee), Y) \\ & \simeq \Map_{\Spec(\k)/\St}(\Spec(\k \oplus E^\vee), \RMap(X,Y)) \\ & \simeq \Map_\k(E, \TT_{\RMap(X,Y), [f]}).
\end{align*}
This equivalence being functorial in $E$ (since every individual step is), we find the announced equivalence $\TT_{\RMap(X,Y), [f]} \simeq R\Gamma(X, f^*\TT_Y)$.

\vskip .2cm

(b) Follows from (a) because $\RAut(X)$ is a  open   in $\RMap(X,X)$ and so their completions
are identified. 
\end{proof}

\paragraph{C. The Lie dg-algebra of a group stack.}
Let $G$ be a group stack (represented by a simplicial stack $G_\bullet$, as in part A). 
Assume that the completion $\wh G_1$ is a formal moduli problem (this is true, for example,
if $G$ is infinitesimally cartesian). In this case we have the tangent complex
$\Lie(G) = \TT_{G,1}$, naturally made into a dg-Lie algebra. Explicitly, we form the 
``classifying formal moduli problem"  $B(\wh G_1)$ as the  formal moduli   envelope
\eqref{eq:FM-envelope}
of the functor 
\[
A\mapsto \hocolim\,  (G_n)^{\wh{\hskip .3cm}}_1 (A), \quad \dgart\lra\sSet. 
\]
Then $\Lie(G)$ is, as a dg-Lie algebra, identified with
$\TT_{B(\wh G_1)}[-1]$. 

\begin{exas}

(a) This construction extends the classical correspondence between group schemes and Lie algebras.
Further, if $\Pi$ is a commutative algebraic group, then $\Lie (\EM(\Pi, n)) \simeq  \Lie(\Pi)[n]$, an abelian
Lie algebra $\Lie(\Pi)$ put in degree $n$. 

\vskip .2cm

(b) For the group stack $\Omega_{Y,y}$ of loop spaces, we have
$\Lie (\Omega_yY) = \TT_{Y,y}[-1]$, because $B((\Omega_{Y,y})^{\wh{\hskip .2cm}}_1)$
is identified with $\wh Y_y$. 

\vskip .2cm

(c)  Let $G$ be a group stack such that $\wh G_1$ is a formal moduli problem, and $\gen=\Lie(G)$.
A central extension of $G$ by $\GG_m$, i.e., a morphism of group stacks $\phi: G\to B(\GG_m)$
gives, after passing to Lie dg-algebras, a morphism $\Lie(\phi): \gen \to \k[1]$, since $\Lie(\GG_m)=\k$.
 By Proposition \ref{prop:Lie-central-ext}, 
 $\Lie(\phi)$ gives rise to  
  a cohomology class
$\gamma_\phi \in\HH^2_\Lie(\gen, \pi)$. 
\end{exas}

 \paragraph{D. The Kodaira-Spencer morphism.}  Let $f: X\to Y$ be a morphism of
 derived stacks and $y\in Y$ be a $\k$-point.  Assume that $\wh Y_y$ is a formal moduli problem
 and the homotopy fiber $Rf^{-1}(y)$ is geometric.   In particular, we have
  the  cotangent complex $\LL_{Rf^{-1}(y)}$. 
  
  \begin{Defi}
  The Kodaira-Spencer morphism of $f$ is the morphism of homotopy Lie algebras
  \[
  \kappa: \TT_{Y,y}[-1]\lra R\Gamma(Rf^{-1}(y), \TT_{Rf^{-1}(y)})
  \]
  obtained as the differential, at the identity element, of the action
  \be\label{eq:action-a}
 a: \Omega_yY\lra \RAut(Rf^{-1}(y))
  \ee
   from  Proposition \ref{prop:actionofloops}. Here we
  use Proposition \ref {prop:tangent-aut}(c) to identify the tangent to $\RAut$. 
  \end{Defi}
  
  \begin{prop}\label{prop:KS-class}
  Assume $X$ is geometric, so we have the normal fiber sequence in
  $\on{D}(Rf^{-1}(y))$
  \[
  \TT_{Rf^{-1}(y)} \lra \TT_X\otimes^L_{\Oc_X}  \Oc_{Rf^{-1}(y)} \lra \TT_{Y,y}\otimes_\k \Oc_{Rf^{-1}(y)}.
  \]
  Let $\delta$ be the coboundary map of this triangle.
  Then $\kappa$ is identified, as a morphism of complexes,
   with the composite of $\delta[-1]$ with the adjunction map
  $\TT_{Y,y}\to R\Gamma(Rf^{-1}(y), \TT_{Y,y}\otimes\Oc_{Rf^{-1}(y)})$. 
  \end{prop}
  
  \noindent{\sl Proof:} We deduce the claim from a more general local statement. 
   Let $i: F\to X$ be any morphism of geometric stacks. We then have the relative tangent
  complex $\TT_{F/X}$ on $X$ fitting into the exact triangle 
  \be\label{eq:jacobi-tr}
  \TT_{F/X}
\buildrel c\over\lra \TT_F \lra i^* \TT_X \lra \TT_{F/X}[1].
  \ee
 Let us form the ``groupoid stack" $Z_\bullet$ with 
  $Z_n = F\times_X^h\cdots \times_X^h F$ ($(n+1)$ factors). Thus $Z_0=F$ (the objects),
  $Z_1 = Z:= F\times_X^h F$ (the morphisms) and so on. 
  The  second projection $p_2: Z\to F$ defines
   an action
  of the groupoid $Z_\bullet$ on $F$ which we write as
  \[
  \ul a: Z\times_F F \lra F. 
  \]
   The tangent,  at the identities of the groupoid (which are represented by the diagonal embedding
   $e: F\to Z$) of $\ul a$ is a map
  \[
  d\ul a: e^* (\TT_{Z/F})\lra \TT_F
  \]
  of which the source is identified, via the projection $p_1: Z\to F$, with $\TT_{F/X}$. 
  The following is then obvious. 
    
  \begin{lem}
 In the above situation,  $d\ul a$ is identified with the
  canonical morphism $c: \TT_{F/X}\to  \TT_F$.
  \end{lem}
  \begin{proof}
   The pullback $Z := F \times_X F$ gives on tangents at $e$ a pullback square
   \[
    \xymatrix{
     e^* \TT_Z \ar[r]^-{d\ul a} \ar[d] & \TT_F \ar[d] \\ \TT_F \ar[r] & i^*\TT_X.
    }
   \]
  In particular, it identifies the fibers of the vertical morphisms $e^*(\TT_{Z/F}) \simeq \TT_{F/X}$. It also identifies the morphisms $d\ul a \colon e^*(\TT_{Z/F}) \to \TT_F$ and $c \colon \TT_{F/X} \to \TT_F$.
  \end{proof}

  We now deduce Proposition \ref{prop:KS-class} from the lemma. Let $F=Rf^{-1}(y)\buildrel i\over\to X$
  be the inclusion of the homotopy fiber. By definition, the action $a$ from
  \eqref{eq:action-a} comes from the action $\ul a$, while $\delta[-1]$ class from the normal sequence,
  as in the proposition. We now note that the normal sequence is canonically identified with
  the shift of the triangle  \eqref{eq:jacobi-tr}, that is, $\TT_{Y,y}\otimes\Oc_F$
  (the ``normal bundle" to $F$)  is the same
  as $\TT_{F/X}[1]$. This finishes the proof.

\subsection{Derived current groups and moduli of $G$-bundles}\label{subsec:der-cur-mod-G}

\paragraph{A. Derived moduli spaces.}
Let $G$ be a reductive algebraic group over $\k$ with Lie algebra $\gen$.
Let  $X$ be a smooth irreducible algebraic variety over  $\k$  and  $x\in X$ be a $\k$-point.
Recall the notations $\wh x=\Spec \Oc_{X,x}\simeq D_n$ and $\wh x^\circ = \wh x-\{x\}
\simeq D_n^\circ$ for the (punctured) formal neighborhood of $x$ in $X$.

\vskip .2cm

Recall that principal $G$-bundles are classified by the $1$-stack $BG = [*/G]$: for any scheme $Y$,
 the (nerve of the) groupoid of principal $G$-bundles on $Y$ is equivalent to the simplicial set of maps
  $Y \to BG$.
This allows us to define the notion of principal $G$-bundles over a derived scheme $Y$: 
first, we denote by $\spRBun_G(Y)$ the simplicial set of morphisms $Y \to BG$. 
A vertex in that simplicial set is called a principal $G$-bundle on $Y$.
 We then define the derived moduli stack of principal $G$-bundles over $Y$ as the following functor from
  $\cdga$ to simplicial sets: 
\[
\RBun_G(Y) = \RMap(Y, BG) \colon\,\,  A \mapsto \spRBun_G(Y \times \Spec A).
\]

Fundamental for us will be the derived stack  $\RBun_G(X)$. 

 \begin{prop}
 Let $P$ be a principal $G$-bundle on $X$ and $[P]\in \RBun_G(X)$ be the corresponding
 $\k$-point. Then 
 \[
 \TT_{[P]} \RBun_G(X) \,\,\simeq \,\, R\Gamma(X, \Ad(P))[1]. 
 \]
 \end{prop}
 
 \noindent {\sl Proof:} Follows from Proposition 
 \ref{prop:tangent-aut}  with $Y=BG$.
 
\qed
\vskip .2cm

We then define $\RBun_G(\wh x)$ as the functor  
\[
A \,\,\mapsto \RBun_G(\Spec(A \mathbin{\wh\otimes}_\k \wh\Oc_{X,x})), \quad
A \mathbin{\wh\otimes}_\k \wh\Oc_{X,x} = \holim (A \otimes_\k \Oc_X/\Ic^n).<
\]
Here $\Ic$ is the ideal of $x$. 

\begin{rem}
 Note that  the definition of $\RBun_G(\wh x)$ is not  the result of applying the construction $\RBun_G(Y)$
 to the scheme $Y=\wh x$. More precisely, for
 any integer $n$,  denote by
$x^{(n)}=\Spec(\Oc_X/\mathcal{I}^n)$,
 the $n^\mathrm{th}$
 infinitesimal neighborhood of $x$ in $X$.    Then
  \[
\RBun_G(\wh {x})  \simeq \holim_n  \RBun_G(x^{(n)}) \,\,= \,\, \RBun_G(\on{Spf} \wh\Oc_{X,x}) 
\]
is a particular case of $\RBun$ construction but for a formal scheme. 
\end{rem}
 
We have the restriction map
\[
\lambda: \RBun_G(X)\to\RBun_G(\wh x).
\]
 
 \begin{Defi}
The derived stack of $G$-bundles on $X$ rigidified at $x$ is defined to be
\[
\RBunrig_G(X, x) = \RBun_G(X) \times^h_{\RBun_G(\hat x)} \{\mathrm{Triv}\}
\,\,=\,\,
R\lambda^{-1}(\on{Triv}),
\]
where $\mathrm{Triv} = X\times G$ denotes the trivial bundle. For any cdga $A$, the simplicial set of $A$-points 
of $\RBunrig_G(X,x)$ is 
\[
\spRBun_G(X \times \Spec A ) \times^h_{\spRBun_G(\Spec(A\mathbin {\wh\otimes}  \wh\Oc_{X,x} ) )} \{\mathrm{Triv}\} . 
\]
In other words, this is the groupoid formed by  $G$-bundles on $X \times \Spec A$ endowed with a trivialization on $\Spec(A\mathbin {\wh\otimes}  \wh\Oc_{X,x} )$.
\end{Defi}

\begin{prop}
(a) The derived stack $\RBunrig_G(X, x)$ is represented by a derived Artin stack.

(b) If $X$ is projective, then $\RBunrig_G(X, x)$ is represented by a derived scheme of amplitude $[0, n-1]$.
\end{prop}

\noindent{\sl Proof:}  Using the smooth atlas $\Spec(\k) \to BG$, we deduce from \cite[Lemma 3.2.4]{hennion-floops} that the inclusion of the trivial bundle $\{\mathrm{Triv} \} \to \RBun_G(\wh x)$ is a smooth atlas. In particular, we get that the derived stack $\RBunrig_G(X,x)$ is representable by a derived Artin stack.

Using Lurie's representability criterion for derived schemes (see \cite[Theorem 3.1.1]{lurie-dagxiv}, or \cite[Appendix C]{toen-vezzosi}), we can hence reduce to the non-derived moduli problem. It is known that $\Bun_G(X)$ is a geometric stack of amplitude $[-1, n-1]$,  which locally 
is represented as the quotient of a scheme of finite type by an algebraic group. When we introduce the rigidification, we kill the stack structure. \qed

\vskip .2cm

We next define the derived stack $\RBun_G(\wh x^\circ)$ as the functor
  \[
   A \mapsto \spRBun_G \left(\Spec\bigl (A \mathbin{\wh\otimes}_\k \wh\Oc_{X,x}\bigr) - 
   \bigl(\Spec(A) \times \{x\}\bigr) \right). 
   \]

Let us fix the notations $X^\circ = X-\{x\}$.

   \begin{prop}\label{prop:BLaszlo}
   The natural morphisms
   \[
 \RBunrig_G(X, x) \to \RBun_G(X^\circ) \times^h_{\RBun_G(
    \wh x^\circ   )} \{\mathrm{Triv}\}
\]
is an equivalence
   \end{prop}
   
   This is proved in \cite[Theorem 6.20]{HennionPortaVezzosi}.
   
   \paragraph{B. Derived current groups and their action.}
   We define $G(\wh x^\circ)$ as the functor
   \[
   A \mapsto \on{Map}\biggl( \Spec (A\wh\otimes\wh\Oc_{X,x}) - (\Spec(A) \times \{x\}), \, G\biggr). 
   \]
   This is group object in derived stacks.     
   
   \begin{prop}\label{prop:Lie-alg-G-hat}
   (a)  $G(\wh x^\circ)$  is a derived affine ind-scheme, identified with the group of automorphisms
    of the point $\on{Triv}$ in the derived stack $\RBun_G( \wh x^\circ)$. 
    
    \vskip .2cm
    
    (b) The  Lie algebra of $G(\wh x^\circ)$ is identified with 
   \[
   \Lie(G(\wh x^\circ)) \,\,=\,\, \TT_e G(\wh x^\circ) \,\,=\,\,
   \gen\otimes R\Gamma(\wh x^\circ, \Oc) \,\,\simeq  \,\,  \gen_x^\bullet 
   \,\,\simeq \,\,\gen_n^\bullet
   \]
   studied earlier. 

    \end{prop}
    
    \noindent {\sl Proof:}
    
    (a) The first claim follows from Lemma 1.4.5 and the proof of Lemma 3.3.6 in \cite{hennion-floops}. The second is obvious.

    \vskip .2cm
    
  (b) We use (a) and compare   $\RBun_G(\wh x^\circ)$ with the mapping stack
\[
\RMap\left(\wh x^\circ, BG\right) \colon A \mapsto \Map\left(\wh x^\circ\times \Spec A, BG\right). 
\]
The group of automorphisms of $\on{Triv}$ in $\RMap\left(\wh x^\circ, BG\right)$ has Lie algebra
identified with $\gen\otimes R\Gamma(\wh x^\circ, \Oc)$ by Proposition
\ref {prop:tangent-aut}(a).

\vskip .2cm
 
We now consider the morphism of  derived stacks
\[
f:  \RMap\left(\wh x^\circ, BG\right) \lra \RBun_G(\wh x^\circ)
\]
induced by the canonical maps
\be\label{eq:can-maps-comp}
A \otimes_\k \wh \Oc_{X,x} \to A \mathbin{\wh \otimes}_\k \wh \Oc_{X,x}. 
\ee
We note that \eqref{eq:can-maps-comp} is a quasi-isomorphism whenever $A$ has finite dimensional
cohomology. This implies that $f$ is  formally \'etale at the point $\on{Triv}$ and therefore
the induced morphism of  tangent Lie algebras at $\on{Triv}$ is a quasi-isomorphism. 

\qed

    \vskip .2cm
    
     Consider the projection of derived stacks
\[
g: \mathbb{R}\mathbf{B}{\mathbf{un}}_G(X^\circ) \lra \mathbb{R}\mathbf{B}{\mathbf{un}}_G
(\wh x^\circ).  
\]
Proposition \ref{prop:BLaszlo} identifies $\RBunrig_G(X,x)$ with the homotopy fiber
$Rg^{-1}(\on{Triv})$. Therefore 
Propositions \ref{prop:Lie-alg-G-hat} and \ref{prop:actionofloops} imply: 
    
\begin{thm}\label{prop-groupaction}
(a) The derived group stack $G(\wh x^\circ)$ acts on $\RBunrig_G(X,x)$ by changing the trivialization.

\vskip .2cm

(b)   The Kodaira-Spencer map of $g$
gives rise to a morphism of dg-Lie algebras
\[
\beta: \gen_x^\bullet\lra R\Gamma(\RBunrig_G(X,x), \TT). 
\]
\qed
\end{thm}

\begin{rem} \label {rem:def-moduli}
 In particular, $\beta$ induces a 
 morphism  on the $(n-1)$-st cohomology of
the dg-Lie algebra $\gen_x^\bullet$:
\[
H^{n-1}_\dbar(\gen_x^\bullet) \lra \HH^{n-1}\bigl( \RBunrig_G(X,x), \TT\bigr). 
\]
Consider the first non-classical case $n=2$, when $X$ is a surface. In this case we are dealing with $\HH^1$ of the tangent complex
which has the meaning of the space of deformation of $\RBunrig_G(X,x)$. Theorem \ref{prop-groupaction}
produces therefore a class of deformations of the rigidified  derived moduli space labelled by the space  of ``polar parts"
$H^{1}_\dbar(\gen_x^\bullet) = H^2_{\{x\}}(\gen\otimes \Oc_X)$. 

\vskip .2cm

A natural way of deforming the moduli space
 would be to ``twist"  the cocycle condition $g_{ik}^{-1}g_{jk}g_{ij}=1$ defining  $G$-bundles by replacing it with 
$g_{ik}^{-1}g_{jk}g_{ij}=\lambda_{ijk}$ with some ``curvature data" $\lambda=(\lambda_{ijk})$.
 Considering such twisted bundles is standard when $G=GL_r$ and 
 $\lambda$ consists of scalar  functions (we then get modules over an Azumaya algebra).
   In our case,  elements of $H^2_{\{x\}}(\gen\otimes \Oc_X)$
 can be seen as providing infinitesimal germs of more general  (non-abelian) twistings and thus deformations of the 
 rigidified moduli space.

 \end{rem}

\subsection{Central extensions associated to Tate complexes}\label{subsec:cent-ext-Tate}

\paragraph{A. The group stack $GL(V)$ and its Lie algebra.}
The categories of Tate modules assemble into two prestacks
\[
\bTate: A \mapsto \Tate_A, \quad\ul\bTate: A \mapsto \Tate_A^{\on {equiv}}
\]
where ``equiv" means the maximal $\infty$-groupoid in the $\infty$-categorical envelope (nerve) of the dg-category.   

In particular, for any object $V$ of $\Tate_\k$ we have a group prestack of automorphisms
$GL(V) = \Omega_V\ul\bTate$. 

\begin{ex}
If $V$ is the Tate space $\k((z))$ in degree $0$, then $GL(V)$ is the
non-derived group ind-scheme $GL(\infty)$ studied in \cite{kashiwara}.
For a general $V$, we get a derived analog of $GL(\infty)$. 
\end{ex}

\begin{prop}\label{prop:GL-end}
(a) Each prestack $GL(V)$ is an infinitesimally cartesian stack.

(b) The Lie algebra $\Lie (GL(V))$  is identified
with $\End(V)$, the algebra of endomorphism of $V$ in $\Tate_\k$ made into a Lie algebra. 

\end{prop}

\noindent {\sl Proof:}  
(a) let $A\to B$ be a homotopy \'etale cover in $\cdga$. Denote $B_n = B^{\otimes n}_A$,
so that $B_\bullet = (B_n)$ is a cosimplicial object in $\cdga$. The natural functor
$F: \Tate_A\to \holim\,  \Tate_{B_n}$ is fully faithful by embedding into the functor
\[
\Ind(\Pro(\Perf_A))\lra \holim_n \Ind(\Pro(\Perf_{B_n})).
\]
Fully faithfulness of this last functor follows from \'etale descent for the categories of
dg-modules. Now, fully faithfulness of $F$ implies that each $GL(V)$ is a stack. 

Next, we show that $GL(V)$ is infinitesimally cartesian. We will prove that $GL(V)$ is actually cartesian: it preserves any pullback of cdga's. Let $D := A \times_B C \in \cdga$.
The canonical base change functor
\[
 \alpha \colon \dgMod_D \lra \dgMod_A \underset{\dgMod_B}{\times} \dgMod_C
\]
admits a right adjoint $\beta$. For any $M \in \dgMod_D$, the unit $M \to \beta \alpha(M) \simeq (M \otimes_D A) \times_{M \otimes_D B} (M \otimes_D C) \simeq M$ is an equivalence and $\alpha$ is thus fully faithful.
It follows that the induced functor $\Tate_D \to \Tate_A \times_{\Tate_B} \Tate_C$ is fully faithful, and therefore that $GL(V)$ preserves fiber products of cdga's, whatever $V \in \Tate_\k$ is

\vskip .2cm

(b)  We construct a morphism of dg-Lie algebras 
\be\label{eq:morphism-phi}
\phi: \Lie(GL(V))\lra\End(V).
\ee
For this we construct a natural transformation $\Phi$ of the functors represented by
both Lie algebras on  $L\in \on{FLie}^{\geq 1}_\k$.

\vskip 2mm
Given such an $L$, we denote by $\Rep_\Tate(L)$ the category of Tate complexes over $\k$ endowed with an action of $L$. By definition, the Lie algebra $\End(V)$ represents the functor
\[
 L \mapsto \Rep_\Tate(L) \times_{\Tate_\k} \{V\}.
\]
We also denote by $\Rep_{\Perf}(L)$ the category of perfect $\k$-complexes endowed with an action of $L$. The functors $\Rep_\Tate$ and $\Rep_{\Perf}$ come with a pointwise fully faithful natural transformation $\Rep_{\Perf} \to \Rep_\Tate$.

Koszul duality between $L$ and its Chevalley-Eilenberg cochains gives a functorial equivalence $\Rep_{\Perf}(L) \simeq \Perf_{\CE^\bullet(L)}$ (see for instance \cite[Lemma 3.38]{hennion-tgtlie}). Moreover, both $\Rep_{\Perf}(L)$ and $\Perf_{\CE^\bullet(L)}$ have a natural functor to $\Perf_\k$ (namely forgetting the action and taking the fiber at $\CE^\bullet(L) \to \k$), and the above equivalence commutes with those functors. Composing with the inclusion, we find a natural transformation $\psi \colon \Perf_{\CE^\bullet(-)} \to \Rep_\Tate$ over $\Tate_\k$.
From the universal property of the category of Tate objects (see \cite[Theorem 2.7]{hennion-tate}), the transformation $\psi$ extends to a natural transformation $\Tate_{\CE^\bullet(-)} \to \Rep_\Tate$ over $\Tate_\k$.

By taking the fiber over $V$, we find a natural transformation
\[
 \Psi \colon \ul\bTate^{\wh{\hskip 2mm}}_V(\CE^\bullet(-)) \lra \Map(-, \End(V)).
\]
We get a natural transformation between the induced formal moduli problems, and thus the announced dg-Lie algebra map $\phi \colon \Lie(GL(V)) \to \End(V)$.

\vskip .2cm

We now prove that $\phi$ is an equivalence or, what this the same, that $\Psi_L$ is an equivalence
whenever $L$ is free on one generator of degree  $d\geq 1$. In this case $\CE^\bullet(L) = \k[\epsilon]$
is the algebra of dual numbers with generator $\epsilon=\epsilon_{d-1}$ of degree $1-d$. 

We include $\Psi_L$ into an adjunction
\[
\begin{gathered}
\psi_L: \Ind(\Pro(\Perf_{\k[\epsilon]})) \leftrightarrow \left\{ \text{Representations of $L$ in $\Ind(\Pro(\Perf_{\k}))$}\right\}:
\eta_L
\\
\psi_L(W) = \k\otimes^L_{\k[\epsilon]} W, \quad \eta_L (V, f: V\to V[d]) = \CE^\bullet(L, V) \simeq \on{Cone}(f)[-1]. 
\end{gathered}
\]
Here $f$ is the action of (the generator of) $L$ on $V$. We note that $\on{Cone}(f)[-1]$ can be
written as $V[\epsilon]=V\otimes_\k \k[\epsilon]$ (free $\k[\epsilon]$-module) with an additional differential given by $f$. 

Note that $\psi_L$ is fully faithful. This is  a formal consequence of   the fact  that 
\[
\psi^{\Perf}_L:  \Perf_{\k[\epsilon]} \leftrightarrow \text{ Representations of $L$ in $\Perf_{\k}$ }
\]
is fully faithful. In fact, $\psi_L^{\Perf}$ is an equivalence (Koszul duality). 

\vskip .2cm

Therefore $\Psi_L$ is fully faithful. Let us prove that it is essentially surjective.

\begin{lem}
Let $(V,f)$ be a representation of $L$ in $\Tate_\k$. Then $\eta_L(V,f)\in\Tate_{\k[\epsilon]}$.
\end{lem}

\noindent{\sl Proof:} We can assume that $V$ is a graded Tate space with a zero differential,
and so $f$ is a morphism of graded Tate spaces. So $V$ has two lattices $V_1^c \subset V_2^c$
so that $f$ induces a morphism of pro-finite-dimensional spaces $f^c: V_1^c\to V_2^c[d]$.
We make $\Cone(f^c)[-1]$ into a (pro-perfect) $\k[\epsilon]$ module by making $\epsilon$ acts by the
embedding $V_1^c\to V_2^c$. 

Similarly, $f$ induces a morphism of ind-finite-dimensional spaces $f^d: V_1^d\to V_2^d$
where $V_i^d=V/V_i^c$. Note that we have the quotient map $V_1^d\to V_2^d$
and so $\Cone(f^d)[-1]$ is made into an (ind-perfect) $\k[\epsilon]$-module. Now we have
a short exact sequence
\[
0\to \Cone(f^c)[-1]\lra\Cone(f)[-1]\lra\Cone(f^d)[-1]\to 0
\]
which implies that $\eta_L(V,f)=\Cone(f)[-1]\in\Tate_{\k[\epsilon]}$. \qed

\vskip .2cm

We now prove that the canonical map $c: \psi_L(\eta_L(V,f))\to (V,f)$ is a quasi-isomorphism
in the category of representations of $L$. 
i.e., that $C= (\Cone(c), g)$ is contractible. By the above, $\eta_L(c)$ is a quasi-isomorphism,
i.e., $\eta_L(C,g) = \Cone(g)$ is contractible. So $g$ is a degree $d$ quasi-isomorphism
of $C$ to itself. But $C$ is bounded by our assumption. So $C$ is contractible. 
This finishes the proof of Proposition \ref{prop:GL-end}. \qed

\paragraph{B. K-theoretic extensions.}

Recall the stack of categories $\bPerf$ defined by
\[
\bPerf(A) = \Perf_A
\]
for any $A \in \cdga$.

\vskip .2cm

For a perfect dg-category $\Ac$ we denote by $K(\Ac)$ the space of K-theory of $\Ac$,
so that $\pi_i K(\Ac)=K_i(\Ac)$. Explicitly, we can define $K(\Ac) = \Omega |S_\bullet(\Ac)|$
as the loop space of the Waldhausen S-construction (in which all $S_n(\Ac)$ are understood
as $\infty$-groupoids). 

We now make K-theory into a prestack 
\[
\K = K\circ \bPerf: A\mapsto K(\Perf(A)).
\]
By composing $K$ with $\bTate$ we get the prestack
\[
\K\bTate : A \mapsto K(\Tate_A).
\]
We have the morphism of prestacks
\[
\ul\bTate \lra \K\bTate
\]
induced by the identification
\[
\Tate_A^{\on {grp}}
 \lra S_1(\Tate_A). 
\]

The following is proven in \cite{hennion-tate} (see also \cite{saito} for a result about the exact category $\Ta_\k$).

\begin{thm}\label{thm:K-deloop}
$\K\bTate$ is identified with $B(\K) = |S_\bullet(\bPerf)|$ (after stackification on Nisnevich topology).\qed
\end{thm}

\begin{rem}
Theorem \ref{thm:K-deloop} is a geometric analog of Theorem \ref{thm:Hc-dloop}. In particular, as Theorem \ref{thm:Hc-dloop} allows us to build central Lie algebra extensions, Theorem \ref{thm:K-deloop} gives us group central extensions. The construction goes as follows.
\end{rem}

 We have the determinantal $\GG_m$-torsor $\Det\to \K$ or, equivalently, a group morphism $\K\to B \GG_m$.
 Applying the classifying stack on both ends, we get the {\em determinantal gerbe} 
\be\label{eq:det-2}
\Det^{(2)}: \K\bTate \lra \EM(\GG_m, 2).
\ee
This gerbe gives, for any $V\in\Tate_\k$, a central extension of group prestacks 
\[
1\to \GG_m\lra \wt{\ul\Aut}(V) \lra\ul\Aut(V) \to 1. 
\]
 
Recall that using Theorem \ref{thm:Hc-dloop}, we built in Definition \ref{def:tateclassHC} a cyclic class $\tau_V \in HC^1(\End(V))$ for each $V\in\Tate_\k$. With Loday's map from Proposition \ref{prop:theta}, we get Lie algebra cohomology classes $\theta(\tau_V)\in H^2_\Lie(\End(V))$.
Those classes give central extensions of Lie algebras
\[
0\to\k \lra \wt\End(V) \to \End(V)\to 0.
\]
 
\begin{thm}\label{thm:det-tr}
Let $V$ be a strict Tate complex. The Lie algebra of $ \wt{\ul\Aut}(V)$ is identified with $\wt\End(V)$.
\end{thm}

\begin{rem}
It is very natural to expect that the cyclic homology
of a dg-category $\Ac$ can be recovered from its K-theory by some functorial procedure 
(``taking the tangent space") so that, in particular, 
  the trace class $\tr\in HC^0(\k)$
corresponds to the determinantal character (the identification $\det: K_1(\k)\to \k^*$). Then one could argue that
the Tate class (the delooping of  $\tr$) is similarly ``tangent"  to the determinantal gerbe (the delooping of $\det$),
thus obtaining a very natural proof of Theorem \ref{thm:det-tr}.  This would also justify the name
``additive K-theory"  for cyclic homology. 

However, such a direct construction seems to be unknown. The closest statement in this direction is
the recovery (due to L. Hesselholt)  of the Hochschild homology of a (dg-)algebra $R$ in terms of the rational K-theory of the ring of
dual numbers $R[\epsilon]/\epsilon^2$, see \cite{dundas}.  

We therefore dedicate the rest of this section to  a   proof of Theorem  \ref{thm:det-tr} by a series of reductions.

\end{rem}

\paragraph{C. Primitivity of the Lie cohomology classes.} 
For $V\in\Tate_\k$ we denote by $\gamma_V\in\HH^2_\Lie(\End(V))$ the class corresponding to the Lie algebra of 
 $ \wt{\ul\Aut}(V)$. We need to prove the equality
  \be\label{eq:gamma=tau}
  \gamma_V = \theta^*(\tau_V)
  \ee
   where $\tau_V\in HC^1(\End(V))$
 is induced by the Tate class $\tau\in HC^1(\Tate_\k)$ and $\theta$ is the Loday homomorphism, see \S \ref{subsec:cyc-Lie-hom}B.

Note that the statement is known (and classical) in the case when $V=\k((z))$ is the most standard example of a Tate space. 
We will now reduce to this case by showing that the system of classes $\gamma_V$ satisfies compatibilities that hold for the system
  of $\theta^*(\tau_V)$.

  \begin{Defi}
  Let $\eta_V\in \HH^2_{\Lie}(\End(V))$, $V\in \Tate_\k$ be a system of Lie algebra cohomology classes. We say that
  $(\eta_V)$ is a  {\em primitive system} if,  for any direct sum decomposition $V\simeq V_1\oplus V_2$  in
  the abelian category of strict Tate complexes we have 
    \[
  \eta_V |_{\End(V_1) \oplus \End(V_2)} \,\,=\,\,
  p_1^*\eta_{V_1} + p_2^* \eta_{V_2}. 
  \]
  Here $p_\nu: \End(V_1) \oplus \End(V_2)\to\End(V_\nu)$ is the projection. 
  \end{Defi}
  
 We  note that direct sum decompositions with $V_2=0$  are given by isomorphisms $\phi: V\to V_1$, so a primitive system satisfies,
  in particular, the compatibility condition:  $\Ad_\phi^*(\eta_{V_1})=\eta_V$. Here
   \[
 \Ad_\phi: \End(V)\to\End(V_1), \quad u\mapsto \phi\circ u \circ\phi^{-1}. 
 \]
  
   \begin{lem}
  The classes $\gamma_V$ form a {\em primitive system}.  
  \end{lem}

  \noindent {\sl Proof:} 
This is because the determinantal gerbe, being a K-theory datum,  is ``group-like",
 i.e., gives a local system on the Waldhausen space of $\Tate_A$ for any $A$.
 That is, for any triangle (simplest cell on the Waldhausen space)
 \[
 V_1\to V\to V_2
 \]
 in $\Tate_A$ we have an isomorphism of  $\GG_m$-gerbes over $\Spec(\k)$
 \[
 \Det^{(2)}(V_1)\otimes\Det^{(2)}(V_2) \to \Det^{(2)}(V)
 \]
 satisfying coherent compatibilties. In particular, for $V\simeq V_1 \oplus V_2$, a direct
 sum decomposition in $\Tate_\k$, we have 
 \[
  \wt{\ul\Aut}(V_1\oplus V_2)_{ {\ul\Aut}(V_1)\times {\ul\Aut}(V'_2)} \,\,
  \simeq \,\,  \wt{\ul\Aut}(V_1) \bigstar  \wt{\ul\Aut}(V_2)
 \] 
 (Baer sum). This, by differentiation (passing to the Lie algebras of  group stacks), implies that the system $(\gamma_V)$
 is primitive. \qed.
 
 \begin{lem}
 The classes $\theta^*\tau_V$ form a primitive system as well. 
 \end{lem}

 \noindent{\sl Proof:}
 This is a general property of cyclic homology. 
  Let $A=\End(V)$ and $A_i=\End(V_i)$. 
We then have the enbedding of
dg-algebras $A_1\oplus A_2\to A$. It is enough to prove that the restriction of  the Loday homomorphism
$\theta^A$ to
$\HH_2^\Lie(A_1 \oplus A_2)$ is equal to the sum of the restrictions on $\HH^2_\Lie(A_1)$ and
$\HH^2_\Lie(A_2)$.  This restriction is the left path in
  the commutative diagram
\[
\xymatrix{
\HH_2^\Lie(A_1\oplus A_2) \ar[d] \ar[r]^{\theta^{A_1 \oplus A_2}} & HC_{r-1}(A_1\oplus A_2)\ar[d]
\\
\HH_2^\Lie(A) \ar[r]^{\theta^A}& HC_{r-1}(A).
}
\]
Looking at the right path we see that $HC_{1}(A_1\oplus A_2)$ being identified with
$HC_{1}(A_1)\oplus HC_{1}(A_2)$, the composition splits into the direct sum of the
two restrictions, as claimed. \qed
 
 \paragraph{D. Comparison of Lie cohomology classes.}
 It remains now to prove the following statement.
 
 \begin{prop}\label{prop:eta-eta}
 Let $\eta$ and $\eta'$ be two primitive system of classes in $\HH^2_\Lie(\End(V))$, $V\in \Tate_\k$. 
 Suppose that $\eta_{\k((z))} = \lambda\cdot\eta'_{\k((z))}$ for some $\lambda\in\k$. Then
 $\eta_V=\lambda\cdot\eta'_V$ for any strict Tate complex $V$. 
 \end{prop}
 
 We notice first: 
\begin{prop}\label{prop:pback-prim}
Let $(\eta_V)$ be a primitive system. 
If $V\simeq V_1\oplus V_2$ as before, then the pullback of $\eta_V$ to $\End(V_1)\subset\End(V)$, is equal to $\eta_{V_1}$.   \qed
\end{prop}

Let now $V^\bullet$ be a strict Tate complex. Decomposing it
as $V^\bullet=H^\bullet \oplus E^\bullet$ as in  Corollary \ref{cor:tate-triang-explicit}(c), we have an isomorphism of associative
dg-algebras (and hence of dg-Lie algebras) $\End(H^\bullet )\to\End(V^\bullet)$.   It implies an isomorphism
$\HH^\bullet_\Lie(\End(V^\bullet))\simeq \HH^\bullet_\Lie(\End(H^\bullet ))$. Since $H^\bullet$ has no differential, 
$\End(H^\bullet)$ is a graded Lie algebra without differential. 
  
\begin{prop}\label{prop:H-Lie-Tate}
Let $H$ be a graded Tate space (situated in finitely many degrees)  which is  neither discrete nor linearly compact.
Assume the graded components of $H$ are of dimension either $0$ or $\infty$. Then we have $\HH^2_\Lie(\End(H))\simeq \k$.
\end{prop}
  
\noindent{\sl Proof:} This is a modification of the result of \cite{feigin-tsygan}  which can be considered
as corresponding to $V$ being $\k((z))$ in degree $0$.  We first relate $HC^1(\End(H))$
with $HC^1(\Tate_\k)\simeq \k$. More precisely, we note:
  
\begin{lem}\label{lem:HC-rho-functor} 
Let $H$ be as above.

\vskip .2cm

(a) The functor
\[
\rho = R\Hom(H, -): \Tate_\k \lra \dgMod_{\End(H)}. 
\]
takes values in perfect dg-modules over $\End(H)$. 
 
 \vskip .2cm
 
 (b) This functor gives a quasi-equivalence between  $\Tate_\k$   and $\Perf_{\End(H)}$. 
 In particular, $HC_\bullet (\End(H))\simeq HC_\bullet (\Tate_\k)$ is spanned by generators in degrees
 $1,3,5, \dots$.
  \end{lem}
  
  The lemma is very close to being a consequence of Th. 17 of \cite{bgw-operators} which identified
  {\em exact} categories of (higher) Tate objects with the categories of {\em projective} modules
  over appropriate endomorphism rings. In our case we deal with objects not lying in the heart of the
  t-structure and consider perfect dg-modules over dg-algebras which are, of course, derived analogs
  of projective modules. 
  \qed
 
 \noindent{\sl Proof:}  (a) Since $H$ is neither linearly
 compact nor discrete, it decomposes as $C \oplus D$, where $C \in C_\k$ and $D \in D_\k$
 have at least one graded component infinite-dimensional. That is, $D$ admits a shift of $\bigoplus_{\ZZ_+}\k$
 as a direct summand, while $C$ admits a shift of $\prod_{\ZZ_+}\k$ as a direct summand. It follows that any
 Tate complex $W$ can be obtained, up to  a quasi-isomorphism, by a finite number of extensions and retracts from $H$.
 This means that $R\Hom(H ,W)$ can be obtained by a finite number of extensions and retracts from
 $R\Hom(H, H)=\End(H)$, so it is perfect. 
 
 \vskip .2cm
 
 (b) Denote $A=\End(H)$.   We first prove that $\rho$ is fully faithful in the dg-sense,
  i.e., induces quasi-isomorphisms on Hom-complexes.
 This is certainly true for the complex $\Hom(H, H)$ which is sent by $\rho$ to $\Hom_A(A,A)=A$. Further,
 $\rho$ is exact and takes direct summands (retracts) to direct summands. So $\rho$ induces a quasi-isomorphism
 on  $\Hom(W_1, W_2)$, where $W_1$ and $W_2$ are any Tate complexes obtained from $H$
 by  a finite number of extensions and retracts. But by the above, all Tate complex are obtained in such a way.
 
 Next, we show that $\rho$ is essentially surjective. This is immediate since $\Perf_A$  is generated, under extensions
 and retracts by $A=\rho(H)$ itself.   \qed
 
 \vskip .2cm
 
 We now prove Proposition \ref{prop:H-Lie-Tate} by the same arguments as in \cite{feigin-tsygan}.
 We keep the notation $A=\End(H)$ and 
 apply the
 dg-algebra analog of the Loday-Quillen-Tsygan theorem \cite {burguelea-cyclic} which gives that
 $\HH_\bullet^\Lie(\gl_\infty(A))$  is the symmetric algebra on the graded space  $HC_{\bullet-1} (A)$.
 Next, because each component of $H$ is infinite-dimensional, $H \simeq (H)^{\oplus r}$
 and therefore $A\simeq\gl_r(A)$  for each $r\geq 1$. This allows us to pass from $\gl_\infty(A)$
 to $A$ itself and conclude that $\HH^2_\Lie(A)\simeq HC^1(A)=\k$. Proposition \ref{prop:H-Lie-Tate}
 is proved.

 \vskip .2cm
 
 We now finish the proof of Proposition \ref{prop:eta-eta}. If $\eta_{\k((z))}=\lambda\cdot\eta'_{\k((z))}$, we have
 $\eta_V = \lambda\cdot\eta'_V$ for any $V = \k((z))\otimes_\k F$, where $F$ is a finite-dimensional graded $\k$-vector space. In other words,  $V$ is a direct sum of shifts of $\k((z))$. Indeed, up to a shift
 $\k((z))$ is a direct summand of $V$ and so the statement follows from Proposition \ref{prop:pback-prim}. 
Further, if $V$ is any strict Tate complex, then there exists an $F$ as above such that $V$ is a direct summand of $\k((z))\otimes_\k F$, and so the statement again follows from Proposition \ref{prop:pback-prim}.
 Proposition \ref{prop:eta-eta} and Theorem \ref{thm:det-tr} are proved.
 
\subsection{Action  on   determinantal torsors}\label{subsec:act-det-tor}

\paragraph{A. Global sections on $D^\circ_n$.}
Let us denote by $\bPerf_{D^\circ_n}$ the derived prestack in categories
\[
\bPerf_{D^\circ_n} \colon A \mapsto \Perf_{\Spec(A[[z_1,\dots,z_n]]) - \{0\}}
\]
It is actually a stack \cite[Theorem 6.10]{HennionPortaVezzosi}.
\begin{prop}\label{prop:buildRGamma}
The global section functor $R\Gamma$ from Proposition \ref{prop:RGammak} naturally extends to a morphism of prestacks
\[
R\Gamma \colon \bPerf_{D^\circ_n} \lra \bTate.
\]
\end{prop}
\begin{proof}
For any $A \in \cdga$ and $p \geq 0$, denote by $A_p$ the Koszul resolution associated to $(z_1^p,\dots, z_n^p)$ in $A[z_1,\dots,z_n]$.
The cdga $A[[z]]$ is the homotopy limit of the diagram $q \mapsto A[z]/I_q$, where $I_q$ is the ideal generated by degree $q$ monomials.
For any $p$, we have $(z_1^p,\dots,z_n^p) \subset I_p$ and $I_{np} \subset (z_1^p,\dots,z_n^p)$. We get
\[
A[[z]] \simeq \holim A_p.
\]
Let $\bPerf_{D_n}$ denote the functor $A \mapsto \Perf_{A[[z]]}$ and $\bPerf_{D_n^{(p)}}$ be the functor $A \mapsto \Perf_{A_p}$.
The canonical morphism of stacks
\[
\bPerf_{D_n} \lra \holim_p \bPerf_{D_n^{(p)}}
\]
is pointwise fully faithful. Indeed, for any cdga $A$, the functor $\alpha \colon \Perf_{A[[z]]} \to \holim_p \Perf_{A_p}$ admits a right adjoint $\beta$ computing the inverse limit. For $E \in \Perf_{A[[z]]}$, the unit map $E \to \beta \alpha(E) \simeq \holim_p (E \otimes_A A_p) \simeq E$ is an equivalence and $\alpha$ is fully faithful.

For any $A \in \cdga$ and any $p \in \NN$, the category $\Perf_{A_p}$ is canonically equivalent to the category of $A_p$-modules in $\Perf_A$ (as $A_p$ is perfect on $A$). In particular, it embeds fully faithfully into the category of $A_p$-modules in $\Ind\Pro\Perf_A$. We get
\[
\bPerf_{D_n} \subset \holim_p \mathbf{Mod}^\mathbf{IPP}_{\Oc_p},
\]
where $\mathbf{Mod}^\mathbf{IPP}_{\Oc_p}$ is the functor $A \mapsto \holim_p \on{Mod}_{A_p}(\Ind \Pro \Perf_A)$.
Denote by $\mathbf{Mod}^\mathbf{IPP}_{\Oc[[z]]}$ the functor $A \mapsto \on{Mod}_{A[[z]]^{\mathrm{top}}}(\Ind \Pro \Perf_A)$ where $A[[z]]^{\mathrm{top}}$ is $"\holim" A_p$, considered as a commutative algebra in $\Pro \Perf_A$.
The base change natural transformation
\[
\mathbf{Mod}^\mathbf{IPP}_{\Oc[[z]]} \lra \holim_p \mathbf{Mod}^\mathbf{IPP}_{\Oc_p}
\]
admits a pointwise right adjoint $\psi$. However $\psi$ is not a natural transformation as it does not commute with base change. It does once restricted to $\bPerf_{D_n}$ though and we get  a natural transformation
\[
\bPerf_{D_n} \lra \mathbf{Mod}^\mathbf{IPP}_{\Oc[[z]]}.
\]
We now consider $\Aen^\bullet_n = R\Gamma(D^\circ_n,\Oc)$ as a $\k[[z]]^{\mathrm{top}}$-module in $\Ind \Pro \Perf_\k$. Tensoring with $\Aen^\bullet_n$ defines an endotransformation of $\mathbf{Mod}^\mathbf{IPP}_{\Oc[[z]]}$.
We then deduce the proposition from the following lemma.
\begin{lem}
The composite natural transformation
\[
\eta \colon \bPerf_{D_n} \lra \mathbf{Mod}^\mathbf{IPP}_{\Oc[[z]]}  \buildrel{- \otimes \Aen^\bullet_n} \over \lra \mathbf{Mod}^\mathbf{IPP}_{\Oc[[z]]} \buildrel \on{Forget} \over \lra \Ind \Pro \bPerf
\]
has values in $\bTate$ and is null-homotopic once restricted to perfect complexes supported at $0 \in D_n$.
\end{lem}
\begin{proof}
Let $A \in \cdga$. The functor $\eta_A$ has by construction values in $\Tate_A$.
It now suffices to prove that the image $\eta_A(A)$ vanishes (where $A$ is seen as a $A[[z]]$-modules with the trivial action). Using base-change, we can assume $A = \k$.

The $\k[[z]]^{\mathrm{top}}$-complex $\Aen^\bullet_n$ is, by \v Cech descent, the homotopy limit of modules of the form $\k[[z]]^{\mathrm{top}}[z_I^{-1}]$ for a none empty $I \subset \{1,\dots,n\}$. Resolving $\k$ as a $\k[[z]]$-module using the natural Koszul complex, we get
\[
\k \otimes^L_{\k[[z]]^{\mathrm{top}}} \k[[z]]^{\mathrm{top}}[z_I^{-1}] \simeq 0
\]
for any $I \neq \emptyset$.
The functor $\eta_\k$ therefore maps $k$ to an acyclic complex.
\end{proof}

We now finish the proof of Proposition \ref{prop:buildRGamma}.
Notice that for any $A \in \cdga$, the category $\bPerf_{D^\circ_n}(A)$ is a quotient of the category $\bPerf_{D_n}(A)$ by the stable full subcategory of perfect complexes supported at $0$.
It follows from the lemma that $\eta$ factors through the morphism of prestacks 
\[
R\Gamma \colon \bPerf_{D^\circ_n} \lra \bTate
\]
which coincides with the functor from Proposition \ref{prop:RGammak} over $\k$-points. This concludes the proof of Proposition \ref{prop:buildRGamma}.
\end{proof}

\begin{rem}
Note that the above construction can be mimicked to define a global section morphism 
\[
R\Gamma \colon \bPerf_{\wh x ^\circ} \to \bTate
\]
for any $\k$-point $x$ in a variety of dimension $n$.
\end{rem}

\paragraph{B. Determinantal torsors and determinantal gerbes.}
 Let  $\phi: G\to GL_r$ be  a representation of $G$. 
Each $G$-bundle $E$ on any derived stack $Z$ induces a vector bundle
$\phi_* E\in\Perf_Z$. 

As before, $X$ is a  smooth projective variety, $\dim(X)=n$. 
We construct a $\GG_m$-torsor $\det^\phi\in\Pic(\RBun_G(X))$ as the morphism $\RBun_G(X)\to
B\GG_m$
defined as the composition
\[
\RBun_G(X)\buildrel \phi_* \over\lra \ul\bPerf_X \buildrel R\Gamma\over\lra \ul\bPerf_\k \lra \K\bPerf_\k
\buildrel\det\over\lra B\GG_m. 
\]
We also denote by $\det^\phi$ the pullback of this torsor to 
$\RBunrig_G(X,x)$. 
 
\vskip .2cm
 
The determinantal gerbe  $\Det^\phi: \RBun_G(\wh x^\circ)\to K(\GG_m, 2)$ is the composition 
\[
\RBun_G(\wh x^\circ) \buildrel \phi_* \over\lra \ul\bPerf_{\wh x^\circ} \buildrel R\Gamma\over\lra \ul\bTate \lra \K\bTate
\buildrel \Det^{(2)}\over\lra K(\GG_m, 2)
\]
where $\Det^{(2)}$ defined in \eqref{eq:det-2}.

\begin{prop}\label{prop:Det-and-det}
(a) The determinantal gerbe $\Det^\phi$ comes with canonical trivializations $\wh \tau$ and $\tau^\circ$ over $\RBun_G(\wh x)$ and $\RBun_G(X^\circ)$.
\vskip 2mm
(b) The determinantal torsor $\det^\phi$ is equivalent to $\Hom_{\Det^\phi}(\wh \tau,\tau^\circ)$.
\end{prop}
\begin{proof}
(a) Let us first deal with $\wh \tau$.
We have a canonical natural transformation
\[
\wh \alpha \colon R\Gamma_{\wh x} = R\Gamma(\wh x,-) \lra R\Gamma_{\wh x^\circ } = R\Gamma(\wh x^\circ,-)
\]
of maps $\bPerf_{\wh x} \to \bTate$. For any $A \in \cdga$, the K-theory of $\Tate_A$ is equivalent to that of the quotient $\Tate_A/\Db_A$. In particular, the natural transformation $\wh \alpha$ induces an equivalence of morphisms between the two composites
\[
\ul\bPerf_{\wh x} \buildrel R\Gamma_{\wh x} \over \lra \ul\bTate \lra \K\bTate
\hspace{6mm} \text{and} \hspace{6mm}
\ul\bPerf_{\wh x} \buildrel R\Gamma_{\wh x^\circ} \over \lra \ul\bTate \lra \K\bTate.
\]
The LHS composite factors through $\K\Cb \colon A \mapsto \K(\Cb_A)$ which vanishes, as the categories of compact complexes admit infinite sums.
The RHS composite appears in the restriction of $\Det^\phi$ to $\RBun_G(\wh x)$ and this identification induces the trivialization $\wh \tau$.

The case of $\tau^\circ$ is done similarly, using the natural transformation
\[
\alpha^\circ \colon R\Gamma(X^\circ,-) \lra R\Gamma(\wh x^\circ,-).
\]
This concludes the proof of (a).
\vskip 2mm
(b) Consider the equivalence 
\[
 \bPerf_X \stackrel\sim\lra \bPerf_{X^\circ} \underset{\bPerf_{\wh x^\circ}}{\times} \bPerf_{\wh x}
\]
from \cite[Corollary 6.13]{HennionPortaVezzosi}.
Its inverse simply computes the fiber products of the given perfect complexes. In particular, once composed with the global section functors, we find that for any family of perfect complexes $E$ over $X$, the canonical morphism 
\[
 R\Gamma(X,E) \lra \on{hoeq}\left( R\Gamma(X^\circ,E) \oplus R\Gamma(\wh x,E) \,\makebox[0pt][l]{\raisebox{2pt}{$\overset{\wh \alpha^{\phantom{\circ}}}{\lra}$}} \raisebox{-2pt}{$\underset{\alpha^\circ}{\lra}$} \, R\Gamma(\wh x^\circ,E) \right) 
\]
is an equivalence, where $\on{hoeq}$ computes the homotopy equalizer.
By construction, the torsor $\Hom_{\Det^\phi}(\wh \tau,\tau^\circ)$ is the determinant of the above homotopy equalizer.
\end{proof}

\paragraph{C. The action of central extensions.} 

Let $X$ be a projective variety of dimension $n$, and $x \in X(\k)$ a $\k$-point.
Let $\phi \colon G \to GL_r$ be a representation.
Pulling back the Tate class from Definition \ref{def:tateclassHC} along the functor $R\Gamma \colon \Perf_{\wh x^\circ} \to \Tate_\k$, we get a class
\[
\tau_x \in HC^1(\wh x^\circ) \simeq HC^1(\Aen^\bullet_{x}).
\]
Recall the definition $\Aen^\bullet_{x} = R\Gamma(\wh x^\circ, \Oc)$.
\begin{Defi}
The class $\tau_x$ induces a central extension of $\gen^\bullet_{x} = \gen \otimes \Aen^\bullet_{x}$ that we will denote by $\wt\gen^\bullet_{x, \phi}$
\[
\k \lra \wt\gen^\bullet_{x, \phi} \lra \gen^\bullet_{x}.
\]
\end{Defi}


This extension has a geometric counterpart. Recall the determinantal gerbe $\Det^\phi \colon \RBun_G(\wh x^\circ) \to K(\GG_m,2)$.
We denote by $[\Det^\phi]$ its total space
\[
[\Det^\phi] = \RBun_G(\wh x^\circ) \times^h_{K(\GG_m,2)} \{*\}.
\]
The diagonal map and the trivial bundle define a $\k$-point $d \in [\Det^\phi](\k)$.
\begin{Defi}
Let $\wt G(\wh x^\circ)_\phi$ denote the group stack $\Omega_d [\Det^\phi]$. It comes with a natural projection $\pi \colon \wt G(\wh x^\circ)_\phi  \to G(\wh x^\circ) = \Omega_\mathrm{Triv} \RBun_G(\wh x^\circ)$.
The homotopy fiber of $\pi$ at the unit is the group scheme $\GG_m$, so that we have a central group extension
\[
\GG_m \lra \wt G(\wh x^\circ)_\phi \lra G(\wh x^\circ).
\]
\end{Defi}

\begin{rem}
Note that the extension $\wt G(\wh x^\circ)_\phi$ is classified by the group morphism $G(\wh x^\circ) \to B\GG_m$ obtained by taking the pointed loops of the map $\Det^\phi \colon \RBun_G(\wh x^\circ) \to K(\GG_m,2)$.
\end{rem}

The following is a direct consequence of the above definitions and of Proposition \ref{prop:GL-end}. and Theorem \ref{thm:det-tr}.
\begin{prop}
The Lie algebra extensions $\Lie(\wt G(\wh x^\circ)_\phi)$ and $\wt\gen^\bullet_{x, \phi}$ of $\Lie(G(\wh x^\circ)) \simeq \gen_x^\bullet$ are equivalent.
\end{prop}

We denote by $[\det^\phi]$ the total space of the determinantal torsor $\det^\phi$ on $\RBunrig(X,x)$.

\begin{thm}\label{thm:action-det}
(a) The group $\wt G(\wh x^\circ)_\phi$ acts on $[\det^\phi]$ in a way compatible with the projections $\wt G(\wh x^\circ)_\phi \to G(\wh x^\circ)$
and $[\det^\phi] \to \RBunrig_G(X,x)$, and with the action from Proposition \ref{prop-groupaction}.

\vskip 2mm

(b) The dg-Lie algebra $\wt \gen_{x,\phi}^\bullet$ acts infinitesimally on $[\det^\phi]$ in a way compatible with the infinitesimal action of $\gen_x^\bullet$ on $\RBunrig(X,x)$.
\end{thm}

\begin{proof}
(a) By construction the group $\wt G(\wh x^\circ)_\phi$ is the pointed loop group of the point $d$ in $[\Det^\phi]$.
Moreover, the trivialization $\tau^\circ$ and Proposition \ref{prop:Det-and-det} give a homotopy cartesian square
\[
\xymatrix{
[\det^\phi] \ar[r] \ar[d] & \RBun_G(X^\circ) \ar[d]^{\tau^\circ} \\
\{d\} \ar[r] & [\Det^\phi].
}
\]
Furthermore, the inclusion $\{d\} \to [\Det^\phi]$ factors as
\[
\{d\} = \Spec \k = \{\mathrm{Triv}\} \lra \RBun_G(\wh x) \buildrel \wh{\tau} \over \lra [\Det^\phi].
\]
In particular, Proposition \ref{prop:actionofloops} defines the announced action of $\wt G(\wh x^\circ)_\phi$ on $[\det^\phi]$.
The above diagram being compatible with the various projections, we see that the action is indeed compatible with the one from Proposition \ref{prop-groupaction}.
\vskip 2mm
(b) It follows from (a), from Corollary \ref {cor:localRR} identifying the extension
and from the Kodaira-Spencer morphism.
\end{proof}

\appendix

  \numberwithin{equation}{section}

\section{Model categories of dg-algebras and dg-categories}

\paragraph{(A.A) Conventions on complexes. }  We recall that $\k$ is a field of characteristic $0$. 
We follow the usual sign conventions on
    (differential) graded $\k$-vector spaces, their tensor products, Koszul sign rule
    for symmetry and so on. The degree of the differential is always assumed to
    be $+1$. 
    The degree of a homogeneous element $v$ will be
    denoted $|v|$. Note, in particular,  the convention
    \be
    (f\otimes g) (v\otimes w) \,=\, (-1)^{|g|\cdot |v|} f(v) \otimes g(w)
    \ee
    for the action of the tensor product of  two operators $f: V\to V'$ and $g: W\to W'$.

    \vskip .2cm
    
    The {\em shift of grading} of a graded vector space $V^\bullet$ is defined by
    $V^\bullet[n] =  \k[n] \otimes V^\bullet $, where $\k[n]$ is the field $\k$ put in degree $(-n)$. 
    So the basis of $\k[n]$ is formed by the vector $1[n]$. 
    For $v\in V^\bullet$ we denote $v[n]=(1[n]) \otimes v\in V^\bullet[n]$. This gives the
    {\em suspension morphism} (an ``isomorphism" of degree $n$)
    \[
    s^n: V\lra V[n],\quad v\mapsto v[n]. 
    \]
    With respect to tensor products,  we have the {\em decalage isomorphism} (of degree $0$)
  \[
\dec: V_1^{\bullet}[1]\otimes V_2^{\bullet}[1]\otimes \dots \otimes V_n^{\bullet}[1]\to 
(V_1^{\bullet}\otimes V_2^{\bullet}\otimes \dots \otimes V_n^{\bullet})[n]
\]
given by:
\[
\dec(s(v_1)\otimes \dots \otimes s(v_n))= (-1)^{\sum^n_{i=1}(n-i)|v_i|}s^n(v_1\otimes \dots \otimes v_n)
\]  
This isomorphism induces an isomorphism of graded vector spaces 
\be\label{eq:dec-n}
\dec_n:  S^n\bigl(V^{\bullet}[1])\bigr)\lra \bigl(\Lambda^n V^{\bullet}\bigr)[n]. 
\ee 

\paragraph{(A.B) Model structures and categories of dg-algebras.  }

  We will freely use the concept of model categories, see, e.g., \cite{lurie-htt}
  for background. For a model category $\Mc$ we denote by $[\Mc] = \Mc[W^{-1}]$
  the corresponding homotopy category obtained by inverting weak equivalences. 
  
      \vskip .2cm
  
  We denote by $\dgVect$ the category of differential graded vector spaces (i.e.,
    cochain complexes) over $\k$ with no assumptions on grading. This is a symmetric
    monoidal category.  
    
    Let $\Pc$ be a $\k$-linear operad. 
   By a dg-algebra of type $\Pc$  we will
    mean a $\Pc$-algebra in $\dgVect$. They form a category denoted $\on{dgAlg}^\Pc$.
    Thus, for $\Pc$ being one of the three
    operads $\Ac s$, $\Cc om$, $\Lc ie$, describing associative, commutative
    and Lie algebras, we will speak about {\em associative
dg-algebras}, resp. {\em commutative  dg-algebras}, resp. 
{\em dg-Lie algebras} (over $\k$). 
The categories formed by such algebras will be denoted by 
  $\dgAlg$, $\Cdga$ and $\dgLie$  respectively. They have products, given by direct sums
  over $\k$,  and coproducts given by free products of algebras, denoted $A*B$.
  For commutative dg-algebras, $A*B=A\otimes_\k B$. 
  
  \vskip .2cm

The category  $\on{dgAlg}^\Pc$  carries a  natural model structure
\cite{hinich-homhom}, in which:
\begin{itemize}
\item Weak equivalences are quasi-isomorphisms.

\item Fibrations are surjective morphisms of dg-algebras. 

\end{itemize}
Cofibrations are uniquely determined by the   axioms of model categories.  In particular, 
 any dg-algebra $A$  of type $\Pc$ and any graded vector space $V$
 we can form $F_\Pc (V)$, the free algebra of  type $\Pc$ generated by $V$, and
  the embedding $A\to A *  F_\Pc(V)$ is a cofibration. 

\vskip .2cm

As usual, the model structure allows us to form homotopy limits and colimits in the categories 
$\on{dgAlg}^\Pc$.
  They will be denoted by $\holim$ and $\hocolim$.  
In particular, if $\Fc$ is a sheaf of  algebras  of type $\Pc$ on a
topological space $S$, then
\[
R\Gamma(S, \Fc) \,\,=\,\,\holim_{U\subset S\atop
\text{open}} \,\, \Fc(U)
\]
is canonically defined as an object of  the homotopy category $[\on{dgAlg}^\Pc]$.
An explicit way of calculating the homotopy limit of  a diagram of algebras represented
by  a cosimplicial algebra (this includes  $R\Gamma(S, \Fc)$) is provided by   the 
Thom-Sullivan  construction, see  \cite {hinich-schechtman}. 

\vskip .2cm

The above includes (for $\Pc$ being the trivial operad) the category $\dgVect$ itself.
In particular, let us note the following fact about homotopy limits in $\dgVect$
indexed by $\ZZ_+ = \{0,1,2,\cdots\}$.

\begin{prop}\label{prop:dgvect-holim}
 (a) Let $(E_i^\bullet)$ be an  inductive system over $\dgVect$, indexed by $\ZZ_+$.
 Then the natural morphism $\hocolim\,  E_i^\bullet\to \varinjlim
E_i^\bullet$
is a quasi-isomorphism. 

\vskip .2cm
 
(b) Let $(E_i^\bullet)$ be a projective system over $\dgVect$, indexed by $\ZZ_+$.
 Then the natural morphism $\varprojlim E_i^\bullet\to \holim\,  E_i^\bullet$
is a quasi-isomorphism in each of the following two cases:
\begin{itemize}
\item[(b1)] Each $E_i^\bullet$ is a perfect complex.

\item[(b2)] The morphisms in the projective system $(E_i^\bullet)$ are termwise
surjective. 
\end{itemize}

\end{prop}

\noindent {\sl Proof:} (a) $\hocolim$ is the left derived functor
of $\varinjlim$. Therefore  part (a) follows from the fact that the functor $\varinjlim^{\Vect_\k}$
is exact on the category of inductive systems of vector spaces indexed by $\ZZ_+$. 

\vskip .2cm

(b) $\holim$ is the right derived functor of $\varprojlim$. Therefofe we have 
the spectral sequence
\[
(R^q\varprojlim\nolimits_i ) (E_i^q)\,\,\Rightarrow \,\, H^{p+q} (\holim_i E_i^\bullet).
\]
As well known, the functors $R^q\varprojlim_i$ for countable filtering diagrams can be
nonzero only for $q=0,1$. Further,   $R^1\varprojlim_i$ vanishes for 
diagrams of finite-dimensional
spaces as well as for any diagrams formed by surjective maps. \qed

\paragraph{(A.C) Model structure on category of dg-categories.} 
We denote by $\dgCat$ the category of $\k$-linear dg-categories. 
For a dg-category $\Ac$ we denote by $[\Ac]$ the corresponding $H^0$-category:
it has the same objects as $\Ac$, while $\Hom_{[\Ac]}(x,y) = H^0 \Hom^\bullet_\Ac(x,y)$. 
 This notation, identical with the notation for the homotopy category of a model category,
 does not cause confusion: when both meanings are possible, the result is the same.

\vskip .2cm

We equip $\dgCat$ with the Morita model structure of Tabuada \cite{tabuada-invariantsdg}.
 Weak equivalences  in this structure are Morita equivalences. 
Fibrant objects   are {\em perfect dg-categories}, i.e.,
dg-categories quasi-equivalent to $\Perf_\Bc$ where $\Bc$ is some small dg-category.
We recall two additional characterizations of perfect dg-categories.

 First,
$\Ac$ is perfect, if and only if the Yoneda embedding $\Ac\to\Perf_\Ac$ is a quasi-equivalence.

Second, $\Ac$ is perfect,  if and only if $\Ac$ is pre-triangulated and $[\Ac]$
(which is then triangulated) is 
 closed
under direct summands.

\vskip 1cm

G. F.:  Max-Planck-Institut f\"ur Mathematik, Vivatsgasse 7, 53111 Bonn, Germany.\\
Email: G. F.: {\tt gf77@mpim-bonn.mpg.de}

\vskip 2mm

B. H.: Laboratoire de Math\'ematiques d'Orsay, Univ. Paris-Sud, CNRS, Universit\'e Paris-Saclay, 91405 Orsay, France.\\Email: {\tt benjamin.hennion@u-psud.fr}

\vskip .2cm

M.K.: Kavli IPMU, 5-1-5 Kashiwanoha, Kashiwa, Chiba, 277-8583 Japan. Email: 
{\tt mikhail.kapranov@ipmu.jp}

\ed